\documentclass[11pt]{amsart}
\usepackage{amscd,amsxtra,amssymb, bbm}
\usepackage[new]{old-arrows}

\makeatletter
\newcommand*{\rom}[1]{\expandafter\@slowromancap\romannumeral #1@}
\makeatother

\makeatletter
\newcommand{\doublewidetilde}[1]{{%
  \mathpalette\double@widetilde{#1}%
}}
\newcommand{\double@widetilde}[2]{%
  \sbox\z@{$\m@th#1\widetilde{#2}$}%
  \ht\z@=.9\ht\z@
  \widetilde{\box\z@}%
}
\makeatother

\usepackage{stackengine}
\stackMath
\newcommand\tsup[2][2]{%
 \def\useanchorwidth{T}%
  \ifnum#1>1%
    \stackon[-.5pt]{\tsup[\numexpr#1-1\relax]{#2}}{\scriptscriptstyle\sim}%
  \else%
    \stackon[.5pt]{#2}{\scriptscriptstyle\sim}%
  \fi%
}

\usepackage{multirow} % Thilo
\usepackage{stmaryrd} %part 4

\input xy
\xyoption{all}
\xyoption{arc}
\SelectTips{cm}{10}

\newtheorem{theorem}{Theorem}[section]
\newtheorem{corollary}[theorem]{Corollary}
\newtheorem{lemma}[theorem]{Lemma}
\newtheorem{proposition}[theorem]{Proposition}

\theoremstyle{definition}
\newtheorem{definition}[theorem]{Definition}
\newtheorem{remark}[theorem]{Remark}

\newtheorem{example}[theorem]{Example}

\DeclareMathOperator{\ad}{\mathsf{ad}}

\DeclareMathOperator{\rk}{\mathrm{rk}}

\newcommand{\idm}{\mathfrak{m}}

\newcommand{\diag}{\mathsf{diag}}

\newcommand{\id}{\mathrm{id}}

\newcommand{\charf}{\mathsf{char}}

\DeclareMathOperator{\res}{\mathsf{res}}
\DeclareMathOperator{\ev}{\mathsf{ev}}

\DeclareMathOperator{\Coh}{\mathsf{Coh}}
\DeclareMathOperator{\TF}{\mathsf{TF}}

\DeclareMathOperator{\VB}{\mathsf{VB}}

\DeclareMathOperator{\alt}{\mathsf{alt}}
\DeclareMathOperator{\Tri}{\mathsf{Tri}}

\DeclareMathOperator{\Pic}{\mathsf{Pic}}

\DeclareMathOperator{\Hom}{\mathsf{Hom}}

\DeclareMathOperator{\Aut}{\mathsf{Aut}}
\DeclareMathOperator{\Inn}{\mathsf{Inn}}
\DeclareMathOperator{\Out}{\mathsf{Out}}
\DeclareMathOperator{\End}{\mathsf{End}}
\DeclareMathOperator{\Mat}{\mathsf{Mat}}
\DeclareMathOperator{\Spec}{\mathsf{Spec}}

\DeclareMathOperator{\lieg}{\mathfrak{g}}
\DeclareMathOperator{\lieh}{\mathfrak{h}}
\DeclareMathOperator{\lien}{\mathfrak{n}}

\DeclareMathOperator{\liew}{\mathfrak{w}}
\DeclareMathOperator{\lieL}{\mathfrak{L}}

\DeclareMathOperator{\liea}{\mathfrak{a}}
\DeclareMathOperator{\liem}{\mathfrak{d}}
\DeclareMathOperator{\lied}{\mathfrak{c}}

\DeclareMathOperator{\lieA}{\mathfrak{A}}
\DeclareMathOperator{\lieB}{\mathfrak{B}}
\DeclareMathOperator{\lieC}{\mathfrak{C}}
\DeclareMathOperator{\lieE}{\mathfrak{E}}
\DeclareMathOperator{\lieG}{\mathfrak{G}}
\DeclareMathOperator{\lieH}{\mathfrak{H}}
\DeclareMathOperator{\lieD}{\mathfrak{C}}
\DeclareMathOperator{\lieP}{\mathfrak{P}}
\DeclareMathOperator{\lieR}{\mathfrak{R}}
\DeclareMathOperator{\lieW}{\mathfrak{W}}

\DeclareMathOperator{\lieV}{\mathfrak{V}}
\DeclareMathOperator{\lieO}{\mathfrak{O}}

\DeclareMathOperator{\lieI}{\mathfrak{I}}
\DeclareMathOperator{\lieM}{\mathfrak{D}}

\DeclareMathOperator{\lieZ}{\mathfrak{Z}}
\DeclareMathOperator{\lieK}{\mathfrak{K}}
\DeclareMathOperator{\lieu}{\mathfrak{u}}
\DeclareMathOperator{\lies}{\mathfrak{s}}

\newcommand{\kk}{\mathbbm{k}}
\newcommand{\KK}{\mathbbm{K}}
\newcommand{\CC}{\mathbb{C}}
\newcommand{\NN}{\mathbb{N}}
\newcommand{\ZZ}{\mathbb{Z}}

\newcommand{\FF}{\mathbb{F}}
\newcommand{\GG}{\mathbb{G}}

\newcommand{\EE}{\mathbb{E}}

\newcommand{\PP}{\mathbb{P}}

\newcommand{\Ad}{\mbox{\emph{Ad}}}

\newcommand{\kA}{\mathcal{A}}

\newcommand{\kB}{\mathcal{B}}

\newcommand{\kF}{\mathcal{F}}
\newcommand{\kG}{\mathcal{G}}

\newcommand{\kO}{\mathcal{O}}
\newcommand{\kL}{\mathcal{L}}
\newcommand{\kP}{\mathcal{P}}
\newcommand{\kQ}{\mathcal{Q}}

\newcommand{\kT}{\mathcal{T}}

\newcommand{\lar}{\longrightarrow}
\newcommand{\rightarrowdbl}{\longrightarrow\mathrel{\mkern-14mu}\rightarrow}
\newcommand{\twoheadarrow}{\rightarrow\mathrel{\mkern-14mu}\rightarrow}
\newcommand{\ttau}{\mathtt{t}}
\newcommand{\ttaus}{\mathtt{s}}

\newcommand{\llbrace}{(\!(}
\newcommand{\rrbrace}{)\!)}

\newcommand{\llangle}{\langle\!\langle}
\newcommand{\rrangle}{\rangle\!\rangle}

\newcommand{\lconvol}{[\!\langle}
\newcommand{\rconvol}{\rangle\!]}

\begin{document}

\title[Algebraic geometry of Lie bialgebras]{Algebraic geometry of Lie bialgebras  defined by  solutions of the classical Yang--Baxter equation}

\author{Raschid Abedin}
\address{
Universit\"at Paderborn\\
Institut f\"ur Mathematik \\
Warburger Stra\ss{}e 100 \\
33098 Paderborn \\
Germany
}
\email{rabedin@math.uni-paderborn.de}

\author{Igor Burban}
\address{
Universit\"at Paderborn\\
Institut f\"ur Mathematik \\
Warburger Stra\ss{}e 100 \\
33098 Paderborn \\
Germany
}
\email{burban@math.uni-paderborn.de}

\begin{abstract}
This paper is devoted to algebro-geometric study  of infinite dimensional Lie bialgebras, which arise from  solutions of the classical Yang--Baxter equation.    We  regard trigonometric solutions of this  equation as twists of the standard Lie bialgebra cobracket  on an appropriate affine Lie algebra and  work out  the corresponding theory of Manin triples,  putting it   into an algebro-geometric context. 
As a consequence of this approach, we prove that any trigonometric solution of the classical Yang--Baxter equation arises from an appropriate algebro-geometric datum. The developed theory   is illustrated by some concrete examples. 
\end{abstract}

\maketitle

\section{Introduction}

\smallskip
\noindent
The notion of a Lie bialgebra originates from the concept of a Poisson--Lie group. Let $G$ be any finite dimensional real Lie group and $\lieg_\diamond$ be its Lie algebra.  It was shown  by Drinfeld in \cite{DrinfeldPoissonLie} that 
Poisson algebra structures  on the algebra $C^\infty(G)$ of smooth functions on $G$ making the group product $G \times G \to G$ to a Poisson map  correspond,  on the Lie algebra level, to linear maps $\lieg_\diamond \stackrel{\delta}\lar 
\wedge^2(\lieg_\circ)$ satisfying the cocycle and  the co-Jacobi identities.  Such a pair $(\lieg_\diamond, \delta)$ is a Lie bialgebra. Conversely, if $G$ is simply connected then any Lie bialgebra cobracket 
$\lieg_\diamond \stackrel{\delta}\lar \wedge^2(\lieg_\diamond)$ defines a Poisson bracket on $C^\infty(G)$ such that   $G \times G \to G$ is a Poisson map; see \cite{DrinfeldPoissonLie}.

\smallskip
\noindent
Assuming that $\lieg_\diamond$ is a  simple Lie algebra, it follows from Whitehead's Lemma that  any Lie bialgebra cobracket $\lieg_\diamond  \stackrel{\delta}\lar \wedge^2(\lieg_\diamond)$ has the form $\delta = \partial_\ttau$ for some tensor $\ttau \in \lieg_\diamond \otimes \lieg_\diamond$, where 
$$
\lieg_\diamond \stackrel{\partial_\ttau}\lar \lieg_\diamond \otimes \lieg_\diamond, \; a \mapsto [a \otimes 1 + 1 \otimes a, \ttau]
$$
 and $\ttau$ satisfies the  classical Yang--Baxter equation for constants (cCYBE):
\begin{equation}\label{E:CYBEconstant}
[\ttau^{12}, \ttau^{13}] + [\ttau^{12}, \ttau^{23}] + [\ttau^{13}, \ttau^{23}] = 0 \quad \mbox{\rm and} \quad 
\ttau^{12} + \ttau^{21} = \lambda \gamma
\end{equation}
Here, 
  $\gamma \in \lieg_\diamond \otimes \lieg_\diamond$ is the Casimir element with respect to the Killing form $\lieg_\diamond \times \lieg_\diamond \stackrel{\kappa_\diamond}\lar \mathbb{R}$ and $\lambda \in \mathbb{R}$.  For  any $a, b, c, d \in \lieg_\diamond$ we put:
$
\bigl[(a\otimes b)^{12}, (c \otimes d)^{13}\bigr] =  [a, c] \otimes b \otimes d \, \in \, \lieg_\diamond^{\otimes 3},$  which determines  the expression
 $[\ttau^{12}, \ttau^{13}]$; the two other summands  $[\ttau^{12}, \ttau^{23}]$ and $[\ttau^{13}, \ttau^{23}]$  of  (\ref{E:CYBEconstant})  are defined in a similar way.

\smallskip
\noindent
 Suppose now that $\lieg$ is a finite dimensional complex simple Lie algebra and $\lieg \times \lieg \stackrel{\kappa}\lar \mathbb{C}$ is its Killing form. Solutions of cCYBE for $\lambda \ne 0$ were classified by Belavin and Drinfeld; see 
\cite[Chapter  6]{BelavinDrinfeldBook}. In a work of Stolin \cite{StolinBialg}  it was shown that such solutions  stand in bijection with direct sum decompositions 
\begin{equation}\label{E:ManinTriplesFiniteDimensional}
\lieg \times \lieg = \lied \dotplus \liew,
\end{equation}
where $\lied := \bigl\{(a, a) \, |\, a \in \lieg \bigr\}$ is the diagonal and $\liew = \liew_\ttau$ is a Lie subalgebra of $\lieg \times \lieg$ which is Lagrangian with respect to the bilinear form 
\begin{equation}\label{E:FormDoubleLieAlg}
(\lieg \times \lieg) \times (\lieg \times \lieg) 
\stackrel{F}\lar \CC, \quad \bigl((a_1, b_1), (a_2, b_2)\bigr) \mapsto \kappa(a_1, a_2) - \kappa(b_1, b_2).
\end{equation}
Such datum $\Bigl(\bigl(\lieg \times \lieg, F\bigr), \lied, \liew\Bigr)$ is an example of a Manin triple. 

\smallskip
\noindent
Let $\doublewidetilde{\lieG} = \doublewidetilde{\lieG}_A$ be the Kac--Moody Lie algebra, associated with a symmetrizable generalized Cartan matrix $A$. It turns out  that
 $\doublewidetilde{\lieG}$ possesses a non-degenerate invariant symmetric bilinear form $\doublewidetilde{\lieG} \times \doublewidetilde{\lieG} \stackrel{B}\lar \CC$ and decomposes into a direct sum of root spaces \cite{Kac}.
From these facts one can deduce that   $\doublewidetilde{\lieG}$ carries a distinguished Lie bialgebra cobracket $\doublewidetilde{\lieG} \stackrel{\delta_\circ}\lar 
\wedge^2\bigl(\doublewidetilde{\lieG}\bigr)$ called \emph{standard}; see \cite{Drinfeld}.

\smallskip
\noindent
Especially interesting and important phenomena in this context arise  in the case of affine Lie algebras. 
Assume  that $A$ is a generalized Cartan matrix of affine type. Then the corresponding affine Lie algebra 
$\widetilde\lieG = \bigl[\doublewidetilde{\lieG}, \doublewidetilde{\lieG}\bigr]$ has a one-dimensional center $\langle c\rangle$ and both $B$ and $\delta_\circ$ induce the corresponding structures on the Lie algebra $\lieG = \widetilde{\lieG}/\langle c\rangle$. Namely, we have a non-degenerate invariant symmetric bilinear form $\lieG \times \lieG \stackrel{B}\lar \CC$ and a Lie bialgebra cobracket
$\lieG \stackrel{\delta_\circ}\lar  \wedge^2(\lieG)$.  According to a theorem of  Gabber and Kac (see \cite[Theorem 8.5]{Kac}), there exists a finite dimensional simple Lie algebra $\lieg$ and 
an automorphism  $\sigma \in \Aut_{\CC}(\lieg)$ of finite order $m$ such that 
$\lieG$ is isomorphic to the twisted loop algebra $\lieL = \lieL(\lieg, \sigma)  := \bigoplus\limits_{k \in \ZZ} \lieg_{k}  z^k \subset \lieg\bigl[z, z^{-1}\bigr]$ (where $\lieg_k$ are eigenspaces of $\sigma$).  The Lie algebra $\lieL$ is a free module of rank $q = \dim_{\CC}(\lieg)$ over the ring $R = \CC\bigl[t, t^{-1}\bigr]$, where $t = z^m$.  It turns out that (up to an appropriate rescaling) the bilinear form 
$\lieL \times \lieL \stackrel{B}\lar \CC$ factorizes as $\lieL \times \lieL \stackrel{K}\lar R \stackrel{\res_0^\omega}\lar \CC$, where $K$ is the Killing form of $\lieL$ (viewed as a Lie algebra over $R$) and $\res_0^\omega$ is the residue map at the zero point with respect to the differential one-form   $\omega = \dfrac{dt}{t}$. Moreover, one can show  that the standard Lie bialgebra cobracket 
$\delta_\circ$ on $\lieL \cong \lieG$ is given by the following formula:
\begin{equation}\label{E:StandBracket}
\lieL \stackrel{\delta_\circ}\lar \wedge^2(\lieL), \quad f(z) \mapsto 
\bigl[f(x) \otimes 1 + 1 \otimes f(y), r_\circ(x, y)\bigr],
\end{equation}
where $r_\circ(x, y)$ is the  so-called  \emph{standard} trigonometric solution of the classical Yang--Baxter equation with spectral parameters (CYBE)
\begin{equation*}
\left\{
\begin{array}{l}
\bigl[r^{12}(x_1, x_2), r^{13}(x_1, x_3)\bigr] + \bigl[r^{13}(x_1, x_3), r^{23}(x_2, x_3)\bigr] +
\bigl[r^{12}(x_1, x_2),
r^{23}(x_2, x_3)\bigr] = 0
\\
r^{12}(x_1, x_2) = 
-r^{21}(x_2, x_1),
\end{array} 
\right.
\end{equation*}
attached to the pair $(\lieg,\sigma)$, see for instance Corollary \ref{C:StandradStructure}. 

\smallskip
\noindent
Following the approach of Karolinsky and Stolin \cite{KarolinskyStolin}, we study twisted Lie bialgebra cobrackets $\delta_\ttau = \delta_\circ + \partial_\ttau$ on $\lieL$, where $$\ttau \in \lieL \wedge \lieL \subseteq (\lieg \otimes \lieg)\bigl[x, x^{-1}, y, y^{-1}\bigr] \;  
\mbox{\rm and}\; \partial_\ttau\bigl(f(z)\bigr) = \bigl[f(x) \otimes 1 + 1 \otimes f(y), 
\ttau(x,y)\bigr].$$ One can show   that $(\lieL, \delta_\ttau)$ is a Lie bialgebra if and only if $r_{\ttau}(x, y) = r_\circ(x, y) + \ttau(x, y)$ is a solution of CYBE (see Theorem \ref{T:TwistsCYBEGeometry}). It is not hard to see  that (after an appropriate change of variables) all trigonometric solutions of CYBE (classified by Belavin and Drinfeld in \cite[Theorem 6.1]{BelavinDrinfeld}) are of the form $r_\ttau(x, y)$ for an appropriate  $\ttau \in \wedge^2(\lieL)$. Conversely, one can show that any solution of CYBE of the form  $r_\ttau(x, y)$ is equivalent to a trigonometric solution of CYBE; see Proposition \ref{P:TwistIsTrigonometric}. We prove that such Lie bialgebra twists $\ttau \in 
\wedge^2(\lieL)$ are parametrized by  Manin triples of the form 
\begin{equation}\label{E:ManinTriplesForTrigSol}
\lieL \times \lieL^\ddagger := \lieD \,\dotplus\, \lieW,
\end{equation}
where $\lieL^\ddagger = \lieL\bigl(\lieg, \sigma^{-1}\bigr)$, $\lieD = \bigl\{(f, f^\ddagger) \,\big| \, f \in \lieL \bigr\}$ (here, $(az^k)^\ddagger = a z^{-k}$ for $a \in \lieg$ and $k \in \ZZ$) and the symmetric non-degenerate bilinear invariant form  
$\bigl(\lieL \times \lieL^\ddagger) \times \bigl(\lieL \times \lieL^\ddagger) \stackrel{F}\lar \CC$ is given similarly to  (\ref{E:FormDoubleLieAlg}), but replacing the Killing form $\kappa$ by the standard form $B$; see Theorem \ref{T:ManinTriplesLoopTwists}. This results establishes another   analogy  between solutions of cCYBE for $\lambda \ne 0$ and trigonometric solutions of CYBE (parallels between both theories were  already highlighted by Belavin and Drinfeld in \cite{BelavinDrinfeldBook}). We expect (in the light of the works \cite{STS, LuWeinstein}) that the constructed Manin triples (\ref{E:ManinTriplesForTrigSol}) will be useful  in the  study of symplectic leaves of  Poisson--Lie structures on the affine Kac--Moody groups and loop groups, associated to trigonometric solutions of CYBE.

\smallskip
\noindent
Using results obtained in this paper,  Maximov together with the first-named author proved in \cite{AbedinMaximov} that up to $R$--linear automorphisms of $\lieL$, the Lie bialgebra  twists of the standard Lie bialgebra cobracket (\ref{E:StandBracket})
are classified by Belavin--Drinfeld quadruples $\bigl((\Gamma_1, \Gamma_2, \tau), \ttaus\bigr)$, which parametrize 
trigonometric solutions of CYBE (see Subsection  \ref{SS:ReviewBDTheory} for details).

\smallskip
\noindent
Based on the work \cite{BurbanGalinat}, we put the theory of Manin triples of the form 
(\ref{E:ManinTriplesForTrigSol}) into an algebro-geometric context. We show that
for any twist $\ttau \in \wedge^2(\lieL)$ of the standard Lie bialgebra structure 
on $\lieL$ there exists an acyclic isotropic coherent  sheaf of Lie algebras $\kA = \kA_\ttau$ on a plane nodal cubic $E = \overline{V(y^2 - x^3 - x^2)} \subset \PP^2$ such that $\Gamma(U, \kA) \cong \lieL$ and such that 
the completed Manin triple $\widehat{\lieL} \times \widehat{\lieL}^\ddagger := \lieD \,\dotplus\, \widehat{\lieW}_\ttau$ is isomorphic to  the Manin triple
$
\widetilde{\lieA}_s = \Gamma(U, \kA) \dotplus \widehat{\lieA}_s,
$
where $s$ is the singular point of $E$, $U = E \setminus \{s\}$, $\widehat{\lieA}_s$
is the completion of the germ of $\kA$ at $s$  and $\widetilde{\lieA}_s$ is its rational hull. Moreover,  $\lieL \stackrel{\delta_\ttau}\lar \wedge^2(\lieL) \subset \lieL \otimes \lieL$ can be  identified with the Lie bialgebra cobracket
$$
\Gamma(U, \kA) \lar \Gamma(U, \kA) \otimes \Gamma(U, \kA) \cong 
\Gamma(U \times U, \kA \boxtimes \kA), \; f(z) \mapsto \bigl[f(x) \otimes 1 + 1 \otimes f(y), \rho(x, y)\bigr],
$$
where $\rho \in \Gamma\bigl(U \times U \setminus \Sigma, \kA \boxtimes \kA)$ is the \emph{geometric $r$-matrix} attached to the pair $(E, \kA)$ (here, $\Sigma \subset U \times U$ is the diagonal); see Theorem \ref{T:TwistsCYBEGeometry}. From this we deduce that any trigonometric solution of CYBE arises from an appropriate pair $(E, \kA)$, completing the program of geometrization  of solutions of CYBE started in \cite{Cherednik, BurbanGalinat}. Another proof of this result was recently obtained by Polishchuk along quite different lines \cite{Polishchuk}. 

\smallskip
\noindent
The theory  of twists of the standard Lie bialgebra cobracket on $\lieL \cong \lieG$   can be regarded as an alternative approach  to the classification of  trigonometric solutions of CYBE. In particular, it is adaptable for  the  study of trigonometric solutions of CYBE for arbitrary \emph{real} simple Lie algebras, which is of the most interest from the point of view of applications in the theory of integrable systems (see \cite{BabelonBernardTalon, ReymanST2})  as well as for simple Lie algebras over  arbitrary fields of characteristic zero.

\smallskip
\noindent
For a completeness of exposition, we also discuss in this paper an algebro-geometric viewpoint  on the theory of  Manin triples of the form $\lieg\llbrace z\rrbrace = \lieg\llbracket z\rrbracket \dotplus \lieW$, which can be associated to an \emph{arbitrary} formal solution of CYBE (see Subsection
\ref{SS:SurveyCYBE}) as well as of Manin triples of the form $\lieg\llbrace z^{-1}\rrbrace = \lieg[z] \dotplus \lieW$, which (according to a work of Stolin \cite{Stolin}) parametrize  the rational solutions of CYBE; see Remark \ref{R:CuspidalCase} and Remark \ref{R:CuspidalCaseQuasiConst}.

\medskip
\noindent
The plan of this paper is the following. 

\smallskip
\noindent
In Section \ref{S:LieBilgLagrDec} we elaborate (following the work of Karolinsky and 
Stolin \cite{KarolinskyStolin}) the theory of twists of  a given Lie bialgebra cobracket.  
The main result of this section is Theorem \ref{T:ManinTriplesTwists}, which describes such 
twists  in the terms of appropriate Manin triples. 

\smallskip
\noindent
Necessary notions and results of the structure theory of affine Lie algebras 
and twisted loop algebras are  reviewed in  Section \ref{S:StandardLiebialgStructure}. In particular, we recall the description   of  the standard Lie bialgebra cobracket  $\lieG \stackrel{\delta_\circ}\lar \wedge^2(\lieG)$ for an affine  Lie algebra $\lieG \cong \lieL$. The main new result of this section   is Theorem \ref{P:CoisotropicSubalgebras} asserting that any bounded Lie subalgebra $\lieO \subset \lieL$, which is coisotropic with respect to the standard bilinear form $\lieL \times \lieL \stackrel{B}\lar \CC$, is stable under the  multiplication with elements of the polynomial algebra $\CC[t]$.

\smallskip
\noindent
In Section \ref{S:TwistsStandardStructure}, we apply the theory of twists of Lie bialgebra cobrackets,  developed in Section \ref{S:LieBilgLagrDec},  to the  particular case of $(\lieL, \delta_\circ)$. The main results of this section are Theorem
\ref{T:ManinTriplesLoopTwists} and Proposition \ref{P:CompletedTwistsManinTriples}, giving a classification  of the twisted Lie bialgebra cobrackets  $\lieL \stackrel{\delta_\ttau}\lar \wedge^2(\lieL)$ via  appropriate Manin triples. 

\smallskip
\noindent
Section \ref{S:ReviewGeomCYBE} is dedicated to the algebro-geometric theory of CYBE. In Subsection 
\ref{SS:SurveyCYBE}, we recall a well-known connection between solutions of CYBE and Manin triples
of the form  $\lieg\llbrace z\rrbrace = \lieg\llbracket z\rrbracket \dotplus \lieW$. In Subsection
\ref{SS:GeomCYBEDatum} we give a survey of the algebro-geometric theory of CYBE developed in \cite{BurbanGalinat}. 
In Subsection \ref{SS:NodalManinTriples}, we study properties of  geometric
CYBE  data $(E, \kA)$, where $E$ is a singular Weierstra\ss{} curve.  The main result of this section is Theorem \ref{T:main2} (see also Remark \ref{R:CuspidalCase}),  which gives a recipe to compute the geometric $r$-matrix attached to a datum $(E, \kA)$. 

\smallskip
\noindent
In Section \ref{S:GeometrizationTrigonomSolutions}, we continue the algebro-geometric study of solutions of CYBE, started in Section \ref{S:ReviewGeomCYBE}. In Subsection \ref{SS:BasisBBDG}, we review the theory of torsion free sheaves on degenerations of  elliptic  curves, following the work \cite{Survey}. Subsections \ref{SS:GeometrizationTwists} and \ref{SS:GeomRMatTwists} are dedicated to the problem of geometrization of twists of the standard Lie bialgebra structure on $\lieL$. In Proposition \ref{P:GeomStandardStructure}, we derive a formula for the standard trigonometric $r$-matrix, associated to an \emph{arbitrary} 
finite order automorphism $\sigma \in \Aut_{\CC}(\lieg)$.  We give a geometric proof of the known fact that the standard Lie bialgebra cobracket $\lieL \stackrel{\delta_\circ}\lar \wedge^2(\lieL)$ is given by the standard solution $r_\circ(x, y)$ of CYBE; see Corollary \ref{C:StandradStructure}. After these preparations been established,  we prove  in Theorem \ref{T:TwistsCYBEGeometry} that an arbitrary twist $\lieL \stackrel{\delta_\ttau}\lar \wedge^2(\lieL)$  arises from an appropriate geometric CYBE datum $(E, \kA)$, where $E$ is 
a  nodal Weierstra\ss{} curve.  After reviewing in Subsection \ref{SS:ReviewBDTheory} the theory  of trigonometric solutions of CYBE due to Belavin and Drinfeld \cite{BelavinDrinfeld, BelavinDrinfeldBook}, we prove in Proposition \ref{P:TwistIsTrigonometric} that any twist $r_\ttau(x, y)$ of the standard solution $r_\circ(x, y)$ of CYBE is equivalent to a trigonometric solution. 

\smallskip
\noindent
Some explicit computations are performed in 
Section \ref{S:ExplicitComputations}. 
In particular, we explicitly describe  Manin triples of the form (\ref{E:ManinTriplesForTrigSol}) and the corresponding geometric data for the  quasi-constant trigonometric solutions of CYBE (see Theorem \ref{T:QuasiConstantGeometrization}) as well as for  a distinguished class of (quasi-)trigonometric solutions $r^{\mathrm{trg}}_{(c, d)}$ for the Lie algebra $\lieg = \mathfrak{sl}_n(\CC)$, 
which are attached to a pair of mutually prime natural numbers $(c, d)$ such that $c + d = n$ (see Theorem \ref{T:CremmerGervaisGeometry}).

\smallskip
\noindent
In the final Section \ref{S:Appendix}, we review various constructions of  Lie bialgebras arising from solutions of the classical Yang--Baxter equation.

\smallskip
\noindent
\textit{List of notation}. For convenience of the reader we introduce now  the most important  notation used in this paper.

\smallskip
\noindent
$-$ We use  Gothic letters as a notation for Lie algebras. In particular, $\lieg$ is  a finite dimensional complex simple Lie algebra of dimension $q$ and $\lieL = \lieL(\lieg, \sigma)$ is the twisted loop algebra associated with an automorphism $\sigma \in \Aut_{\CC}(\lieg)$ of order $m$, whereas $\overline{\lieL} = \lieg\bigl[z, z^{-1}\bigr]$ denotes the full loop algebra. We put $t = z^m$ and $R = \CC\bigl[t, t^{-1}\bigr]$ and denote
by $\lieg \times \lieg \stackrel{\kappa}\lar \CC$ (respectively, $\lieL \times \lieL \stackrel{K}\lar 
R$)  the   Killing form of $\lieg$ (respectively, of $\lieL$)  and by $\gamma \in \lieg \otimes \lieg$ (respectively, $\chi \in \lieL \otimes_{R}
\lieL$)  the corresponding Casimir element. 

\smallskip
\noindent
$-$ Unless otherwise stated,  by $\otimes$ we mean the tensor product over the field of definition. We use $\dotplus$ to denote the (inner) direct sum of vector spaces. Given a vector space $V$ over a field $\kk$ and $v_1, \dots, v_n \in V$, we denote by $\langle v_1, \dots, v_n\rangle_{\kk}$ the corresponding linear hull. If $V$ is a Lie algebra then $\llangle v_1, \dots, v_n\rrangle$ is the Lie subalgebra of $V$ generated by $v_1, \dots, v_n$. 

\smallskip
\noindent
$-$ We denote by $\widetilde{\lieG}$ an affine Lie algebra and by $\lieG$ its quotient modulo the center. Next,  $\widetilde{\lieG} \times \widetilde{\lieG} \stackrel{B}\lar \CC$ (respectively, $\lieL \times \lieL \stackrel{B}\lar \CC$)  is the  standard bilinear form and 
$\widetilde{\lieG} \stackrel{\delta_\circ}\lar  \wedge^2(\widetilde{\lieG})$  (respectively, $\lieL \stackrel{\delta_\circ}\lar  \wedge^2(\lieL)$) is the  standard Lie bialgebra cobracket.

\smallskip
\noindent
$-$ A Weierstra\ss{} curve $E$ is an irreducible projective curve over $\CC$ of arithmetic genus one. If $E$ is singular then $s$ denotes its singular point and $U = E \setminus\{s\}$ its regular part.  For a coherent sheaf $\kF$ on a scheme  $X$ and a point $p  \in X$,  we denote by
$\kF\big|_p$ the fiber of $\kF$ over $p$ and by $\kF_p$ the stalk of $\kF$ at $p$. 

\smallskip
\noindent
$-$ Next, $\kA$ denotes a coherent sheaf of Lie algebras on a (singular) Weierstra\ss{} curve $E$ such that $H^0(E, \kA) = 0 = H^1(E, \kA)$ and $\kA\big|_x \cong \lieg$ for any $x \in U$ (together with a certain extra condition at the singular point $s$). Such a pair $(E, \kA)$ is called geometric CYBE datum and 
$\rho$ is the corresponding geometric $r$-matrix.

\smallskip
\noindent
$-$ Given a geometric CYBE datum $(E, \kA)$ and a fixed point $p\in E$, we write $\kO$ for the structure sheaf of $E$ and put $E_p = E\setminus\{p\}$ and $U_p = U\setminus \{p\}$ as well as $R = \Gamma(U,\kO)$, $R_p = \Gamma(E_p,\kO)$ and $R_p^\circ = \Gamma(U_p,\kO)$. For the corresponding sections of $\kA$ we write $\lieA = \Gamma(U,\kA)$, $\lieA_{(p)} = \Gamma(E_p,\kA)$ and $\lieA^\circ_{(p)} = \Gamma(U_p,\kA)$. The completion of the stalk of $\kO$ at $p$ is denoted by $\widehat O_p$, while its field of fraction is denoted by $\widehat Q_p$. Finally,  the completion of the stalk of $\kA$ at $p$ is denoted by $\widehat\lieA_p$, whereas $\widetilde\lieA_p = \widehat Q_p \otimes_{\widehat O_p}\widehat \lieA_p$ is the corresponding rational hull. If $p$ is the singular point of $E$, we omit the indices $p$. 

\medskip
\noindent
\emph{Acknowledgement}. The work of both  authors  was supported  by the  DFG project Bu--1866/5--1. We are grateful to Stepan Maximov and Alexander Stolin for fruitful discussions as well as to both anonymous referees for their helpful comments and remarks. 

\clearpage
\section{Lie bialgebras and Lagrangian decompositions}\label{S:LieBilgLagrDec}

\smallskip
\noindent
In this section $\kk$ is a field of $\charf(\kk) \ne 2$

\subsection{Generalities on Lie bialgebras} Let $\lieR = \bigl(\lieR, [-\,,-\,]\bigr)$ be a Lie algebra over  $\kk$. Recall the following standard notions.

\begin{itemize}
\item For any $n \in \NN$ we denote:
$\lieR^{\otimes n} = \underbrace{\lieR \otimes \lieR \otimes \dots \otimes \lieR}_{\text{$n$ times }}$. For any $\ttau \in \lieR^{\otimes n}$ and $a \in \lieR$, we put: $a \circ \ttau = \ad_a(\ttau) := \bigl[a \otimes 1 \otimes \dots \otimes 1 + \dots + 1 \otimes \dots \otimes 1 \otimes a, \ttau \bigr].
$ A tensor $\ttau \in \lieR^{\otimes n}$ is called \emph{ad-invariant} if $a \circ \ttau = 0$ for all $a\in \lieR$. 
\item A linear map $\lieR \stackrel{\delta}\lar \lieR \otimes \lieR$ is a \emph{skew-symmetric cocycle} if $\mathsf{Im}(\delta) \subseteq \wedge^2(\lieR)$ and $$
\delta\bigl([a, b]\bigr) = a \circ \delta(b) - b \circ \delta(a)$$
 for all $a, b \in \lieR$.
\item For any $\ttau \in \lieR^{\otimes 2}$ we have a linear map 
$
\lieR \stackrel{\partial_\ttau}\lar \lieR^{\otimes 2}, \; a\mapsto a \circ \ttau.
$
If $\ttau \in \wedge^2 \lieR$ then $\partial_\ttau$ is automatically a skew-symmetric cocycle.
\end{itemize}

\begin{definition} A \emph{Lie bialgebra} is a pair $(\lieR, \delta)$, where 
$\lieR$ is a Lie algebra and  $\delta$ is a skew-symmetric cocycle satisfying the co-Jacobi identity
$\alt\bigl((\delta \otimes \mathbbm{1}) \circ \delta\bigr) = 0$, where $\lieR^{\otimes 3} \stackrel{\alt}\lar \lieR^{\otimes 3}$ is given  by the formula $\alt(a\otimes b \otimes c) := a\otimes b \otimes c + c \otimes a \otimes b + b \otimes c \otimes a$ for $a, b, c \in \lieR$.
\end{definition}

\begin{remark} Let $(\lieR, \delta)$ be a Lie bialgebra. 
\begin{itemize}
\item The Lie cobracket  $\delta$ defines an element in the Lie algebra cohomology 
 $H^1\bigl(\lieR, \wedge^2(\lieR)\bigr)$. For any $\ttau \in \wedge^2(\lieR)$ we have:  $[\partial_\ttau] = 0$ in 
 $H^1\bigl(\lieR, \wedge^2(\lieR)\bigr)$. 
\item The linear map 
$
\lieR^\ast \otimes \lieR^\ast \longhookrightarrow  \bigl(\lieR \otimes \lieR\bigr)^\ast \stackrel{\delta^\ast}\lar \lieR^\ast
$
 defines a Lie algebra bracket on the dual vector space  $\lieR^\ast$ of $\lieR$. \hfill $\lozenge$
 \end{itemize} 
\end{remark}

\smallskip
\noindent
Following the work \cite{KarolinskyStolin}, we have the following result.
\begin{proposition}\label{P:TwistLieBialgStructures}
Let $(\lieR, \delta)$ be a Lie bialgebra, $\ttau \in \wedge^2(\lieR)$ and 
$\delta_{\ttau} := \delta + \partial_{\ttau}$. Then $(\lieR, \delta_\ttau)$ is a Lie bialgebra if and only if the tensor $\bigl(\alt\bigl((\delta \otimes \mathbbm{1})(\ttau)\bigr) - [[\ttau, \ttau]]\bigr)
\in \lieR^{\otimes 3}$ is ad-invariant, where 
$$[[\ttau, \ttau]] := [\ttau^{12}, \ttau^{13}] + [\ttau^{12}, \ttau^{23}] + [\ttau^{13}, \ttau^{23}].$$ 
In this case, $\delta_\ttau$ is  called  a \emph{twist} of $\delta$. 
\end{proposition}

\begin{proof}
Clearly, $\delta_\ttau$ is a skew-symmetric cocycle. Hence, $(\lieR, \delta)$ is a Lie bialgebra if and only if $
\alt\bigl((\delta_\ttau \otimes \mathbbm{1}) 
\circ \delta_\ttau\bigr)(x) = 0$
for all $x \in \lieR$. Since $(\lieR, \delta)$ is a Lie bialgebra, we have:
$\alt\bigl((\delta \otimes \mathbbm{1}) \circ \delta\bigr) = 0$. Next,  for any $x \in \lieR$ the following formula is true: 
$$
\alt\bigl((\partial_\ttau \otimes \mathbbm{1}) \circ \partial_\ttau\bigr)(x) = - x \circ [[\ttau, \ttau]],
$$
see
\cite[Lemma 2.1.3]{ChariPressley}.
If  $\ttau = \sum\limits_{i = 1}^n a_i \otimes b_i$ then we have:
$$
x \circ \bigl((\delta \otimes \mathbbm{1})(\ttau)\bigr) = 
\sum\limits_{i=1}^n \bigl((x \circ \delta(a_i)) \otimes b_i + \delta(a_i) \otimes [x, b_i]\bigr).
$$
Since $(\delta \otimes \mathbbm{1}) (\partial_\ttau (x))  = (\delta \otimes \mathbbm{1}) [x \otimes 1 + 1 \otimes x, \ttau] = $
$$
= \sum\limits_{i=1}^n\bigl(\delta([x, a_i]) \otimes b_i + \delta(a_i) \otimes [x, b_i]\bigr) = \sum\limits_{i=1}^n\bigl(
\bigl(x \circ \delta(a_i) - a_i \circ \delta(x)\bigr) \otimes b_i + \delta(a_i) \otimes [x, b_i]
\bigr),
$$
we obtain: 
$
(\delta \otimes \mathbbm{1}) (\partial_\ttau (x)) = x \circ \bigl((\delta \otimes \mathbbm{1})(\ttau)\bigr)  -  \sum\limits_{i=1}^n  \bigl(a_i \circ \delta(x)\bigr) \otimes b_i.
$

\smallskip
\noindent
Let $\delta(x) = \sum\limits_{j = 1}^m x_j \otimes y_j$. Then we have: 
$$
(\partial_\ttau \otimes \mathbbm{1}) (\delta(x)) = \sum\limits_{j = 1}^m 
\sum\limits_{i = 1}^n \bigl([x_j, a_i] \otimes b_i \otimes y_j  + a_i \otimes [x_j, b_i] \otimes y_j\bigr)
$$
and
$$
\sum\limits_{i=1}^n  \bigl(a_i \circ \delta(x)\bigr) \otimes b_i = \sum\limits_{j = 1}^m 
\sum\limits_{i = 1}^n \bigl([a_i, x_j]  \otimes y_j \otimes b_i  + x_j  \otimes [a_i, y_j] \otimes b_i\bigr).
$$
Since $\ttau \in \wedge^2(\lieR)$, we have: $\ttau = - \sum\limits_{i = 1}^n b_i \otimes a_i$. It follows that
$$
\sum\limits_{j = 1}^m 
\sum\limits_{i = 1}^n [a_i, x_j]  \otimes y_j \otimes b_i  = 
\sum\limits_{j = 1}^m 
\sum\limits_{i = 1}^n [x_j, b_i] \otimes y_j \otimes a_i.$$ 
As a consequence, we obtain:
$
\alt\bigl(\sum\limits_{j = 1}^m 
\sum\limits_{i = 1}^n  a_i \otimes [x_j, b_i] \otimes y_j - [x_j, b_i] \otimes y_j \otimes a_i\bigr) = 0.
$
Similarly, since $\delta(x) \in \wedge^2(\lieR)$, we have: $\sum\limits_{j = 1}^m x_j \otimes y_j  = - \sum\limits_{j = 1}^m y_j \otimes x_j$. Hence, 
$$
\sum\limits_{j = 1}^m 
\sum\limits_{i = 1}^n x_j  \otimes [a_i, y_j] \otimes b_i = 
\sum\limits_{j = 1}^m 
\sum\limits_{i = 1}^n y_j \otimes[x_j, a_i] \otimes b_i
$$
and as a consequence, 
$
\alt\bigl(\sum\limits_{j = 1}^m 
\sum\limits_{i = 1}^n  [x_j, a_i] \otimes b_i \otimes y_j  - y_j \otimes[x_j, a_i] \otimes b_i\bigr) = 0.
$
Putting everything together, we finally obtain:
$
\alt\bigl((\delta_\ttau \otimes \mathbbm{1}) \circ \delta_\ttau\bigr)(x) = 
x \circ \bigl(\alt\bigl((\delta \otimes \mathbbm{1})(\ttau)\bigr)-[[\ttau, \ttau]]\bigr),
$
implying the statement. 
\end{proof}

\begin{corollary} Let $(\lieR, \delta)$ be a Lie bialgebra and $\ttau \in \wedge^2(\lieR)$. 
A sufficient condition
for $\delta_\ttau$ to be a twist of $\delta$ is provided by the \emph{twist equation} 
\begin{equation}\label{E:TwistEquation}
\alt\bigl((\delta \otimes \mathbbm{1})(\ttau)\bigr) -[[\ttau, \ttau]] = 0,
\end{equation}
introduced in \cite{KarolinskyStolin}.
\end{corollary}

\begin{definition}
Let $\lieR$ be a Lie algebra over $\kk$ and $\lieR \times \lieR \stackrel{F}\lar \kk$ be a symmetric invariant non-degenerate bilinear form, i.e.~$
F\bigl([a, b], c) = F\bigl(a, [b, c]\bigr)
$
for all $a, b, c \in \lieR$. Next, let $\lieR_\pm \subset \lieR$ be a pair of  Lie subalgebras such that
$$
\lieR = \lieR_+ \dotplus \lieR_- \quad \mbox{\rm and} \quad \lieR_\pm \subseteq \lieR_\pm^{\perp},
$$
where $\dotplus$ is the direct sum of vector subspaces. 
Then 
$\bigl((\lieR, F),  \lieR_+, \lieR_-\bigr) = \bigl(\lieR, \lieR_+, \lieR_-\bigr)
$ is called a \emph{Manin triple}. We say that a given splitting $\lieR = \lieR_+ \dotplus \lieR_- $ is a Manin triple, if $\bigl(\lieR, \lieR_+, \lieR_-\bigr)$ is. Two Manin triples $\bigl((\lieR, F), \lieR_+, \lieR_-)$ and
$\bigl((\widetilde{\lieR},\widetilde{F}), \widetilde{\lieR}_+, \widetilde{\lieR}_-)$ are isomorphic if there exists an isomorphism of Lie algebras $\lieR \stackrel{f}\lar \widetilde{\lieR}$, which is a homothety  with respect to the bilinear forms $F$ and $\widetilde{F}$ (i.e.~there exists $\lambda \in \kk^\ast$ such that $F(a, b) = \lambda \widetilde{F}(a, b)$ for all $a, b \in \lieR$) and such that $f(\lieR_\pm) = \widetilde{\lieR}_\pm$. 
\end{definition}

\begin{remark}
If  $\bigl(\lieR, \lieR_+, \lieR_-)$ is a Manin triple, then we automatically  have:  $\lieR_\pm =  \lieR_\pm^{\perp}$; see Lemma \ref{L:LagrangianBasic} below. \hfill $\lozenge$
\end{remark}

\begin{definition}\label{D:CobracketManinTriple}
Let $(\lieR_+, \delta)$ be a Lie bialgebra. We say that the Lie bialgebra cobracket $\lieR_+ \stackrel{\delta}\lar \wedge^2(\lieR_+)$ \emph{is determined} by a Manin triple 
$\bigl((\lieR,F), \lieR_+, \lieR_-)$ if 
\begin{equation}
F\bigl(\delta(a), b_1 \otimes b_2\bigr) = F\bigl(a, [b_1, b_2]\bigr)
\end{equation}
for all $a \in \lieR_+$ and $b_1, b_2 \in \lieR_-$. 
\end{definition}

\noindent
It is clear that if $\lieR_+ \stackrel{\tilde\delta}\lar \wedge^2(\lieR_+)$ is another Lie bialgebra cobracket which is  determined by the same Manin triple 
$\bigl(\lieR, \lieR_+, \lieR_-)$, then $\delta = \tilde\delta$. 

\subsection{Some basic results on Lagrangian decompositions}
Let $V$ be a (possibly infinite dimensional) vector space over $\kk$. Recall that two vector subspaces $W', W'' \subset V$ are called \emph{commensurable}  (which will be denoted $W' \asymp 	 W''$) if $\dim_{\kk}\bigl((W' + W'')/(W' \cap W'')\bigr) < \infty$.

\begin{lemma}\label{L:LagrangianBasic}
Let $V = U \dotplus W$, where $U, W \subset V$ are isotropic subspaces  with respect to a non-degenerate symmetric bilinear form $V \times V \stackrel{F}\lar \kk$.
Then we have:
\begin{enumerate}
\item[(a)] The linear map $U \stackrel{\widetilde{F}}\lar W^\ast, u \mapsto F(u, \,-)$ is injective and both subspaces $U$ and $W$ are automatically Lagrangian, i.e.~  $V = U \dotplus W$ is  a \emph{Lagrangian decomposition}.
\item[(b)] The linear map
$$
U \otimes U \stackrel{\jmath}\lar \Hom_{\kk}(W, U), \; \ttau = \sum\limits_{i = 1}^n a_i \otimes b_i \mapsto \bigl(W \stackrel{f_\ttau}\lar U, \; w \mapsto \sum\limits_{i = 1}^n F(w, a_i) b_i\bigr)
$$
is injective.
\item[(c)] For any $\ttau \in U^{\otimes 2}$ let 
$
W_\ttau := \bigl\{w + f_\ttau(w) \, |\, w \in W\bigr\}.
$
Then we have:
\begin{enumerate}
\item[(1)] $V = U \dotplus W_\ttau$ and $W \asymp 	 W_\ttau$.
\item[(2)] The map $W \lar W_\ttau, w \mapsto w + f_\ttau(w)$ is an isomorphism of vector spaces and $W_\ttau = W_{\ttau'}$ if and only if $\ttau = \ttau'$. 
\end{enumerate}
\end{enumerate}
\end{lemma}

\begin{proof} (a) Since $U \subseteq U^{\perp}$ and $F$ is non-degenerate, the linear map $\widetilde{F}$ is injective.  Let $v \in U^\perp$. Then there exist uniquely determined $u \in U$ and $w \in W$ such that $v = u+w$. For any $u' \in U$ and $w' \in W$ we have:
$$
F(w, u') = F(v, u') = 0 \quad \mbox{and} \quad F(w, w') = 0.
$$
It follows that $w = 0$ and $v = u \in U$, hence  $U = U^\perp$ is Lagrangian. 

\smallskip
\noindent
(b) Since $U$ is isotropic and $F$ is non-degenerate, the  linear map $U \stackrel{\widetilde{F}}\lar W^\ast, \, u \mapsto F(-, u)$ is injective. The linear map $\jmath$ coincides with the composition
$$
U \otimes U \stackrel{\widetilde{F} \otimes \mathbbm{1}}\longhookrightarrow W^\ast \otimes U \longhookrightarrow \Hom_{\kk}(W, U),
$$
and is therefore injective. 

\smallskip
\noindent
(c1) Let $\ttau = \sum\limits_{i = 1}^n a_i \otimes b_i$. Then $\mathsf{Im}(f_\ttau) \subseteq  \bigl\langle b_1, \dots, b_n\bigr\rangle_{\kk}$ and $\dim_{\kk}\bigl(\mathsf{Im}(f_\ttau)\bigr) \le n$. Since $W/\mathsf{Ker}(f_\ttau) \cong \mathsf{Im}(f_\ttau)$, there exists a finite dimensional vector subspace $W' \subset W$ such that $W = W' \dotplus \mathsf{Ker}(f_\ttau)$. It follows that
$$
\mathsf{Ker}(f_\ttau) \subseteq W \cap W_t \subseteq W + W_\ttau \subseteq \mathsf{Ker}(f_\ttau) + \bigl(W' +  \mathsf{Im}(f_\ttau)\bigr).
$$
Hence, $W \asymp 	 W_\ttau$. It is easy to see that $U \cap W_\ttau = 0$ and $W \subset U + W_\ttau$. It follows that $V = U + W \subseteq U +  W_\ttau$, hence $V = U \dotplus W_\ttau$ as asserted. 

\smallskip
\noindent
(c2) The linear map $W \rightarrow W_\ttau$ is by construction surjective. It is also easy to see that it is injective. 

\smallskip
\noindent
Assume that $\ttau, \ttau' \in U^{\otimes 2}$ are such that $W_\ttau = W_{\ttau'}$. Then for any $w \in W$ there exists a uniquely determined $w' \in W$ such that 
$
w + f_\ttau(w) = w' + f_{\ttau'}(w').
$
It follows from $U \cap W = 0$ that $w = w'$. Hence, $f_\ttau(w) = f_{\ttau'}(w)$ for all $w \in W$. Since  $\jmath$ is injective, we have: $\ttau = \ttau'$.
\end{proof}

\begin{proposition}\label{P:LagrangianAdvanced}
Let $V = U \dotplus W$ be a Lagrangian decomposition and
$$
\mathsf{LG}\bigl(V, U; W) := \left\{
\widetilde{W} \subseteq V \left|
\begin{array}{l}
V = U \dotplus \widetilde{W} \\
\widetilde{W}^\perp = W \quad \mbox{\rm and} \quad \widetilde{W} \asymp 	 W
\end{array}
\right.
\right\}.
$$
Then the map 
$
\wedge^2 U \lar \mathsf{LG}\bigl(V, U; W), \; \ttau \mapsto W_\ttau
$
is a bijection.
\end{proposition}
\begin{proof}
Let $\ttau \in U^{\otimes 2}$. Then $W_\ttau \subset V$ is Lagrangian if and only if
$$
F\bigl(f_\ttau(w), w'\bigr) + F\bigl(w, f_\ttau(w')\bigr) = 0 \quad \mbox{for all} \quad 
w, w' \in W.
$$
It follows that $\widehat{F}\bigl(\ttau + \ttau^{21}, w \otimes w'\bigr) = 0$ for all $w, w' \in W$, where $V^{\otimes 2} \times V^{\otimes 2} \stackrel{\widehat{F}}\lar \kk$ is the bilinear form induced by $F$. Since $V = U \dotplus W$ is a Lagrangian decomposition, it follows that $\widehat{F}\bigl(\ttau + \ttau^{21}, v \otimes v'\bigr) = 0$ for all $v, v' \in V$. Thus,  $\ttau + \ttau^{21} = 0$, i.e. $\ttau \in \wedge^2(U)$. Lemma \ref{L:LagrangianBasic}  implies that
$
\wedge^2 U \lar \mathsf{LG}\bigl(V, U; W), \; \ttau \mapsto W_\ttau
$
is a well-defined injective map and it remains to prove its surjectivity.

\smallskip
\noindent
Let $\widetilde{W} \in \mathsf{LG}\bigl(V, U; W)$. Then for any $w \in W$ there exist uniquely determined $u \in U$ and $\tilde{w} \in \widetilde{W}$ such that 
$w = \tilde{w} - u$. We define a linear map $W \stackrel{f}\lar  U$ by setting 
$u:= f(w)$. Since $W \asymp 	 \widetilde{W}$, 
$
\mathsf{Ker}(f) = W \cap \widetilde{W} \subseteq W
$
is a subspace of finite codimension and $\dim_{k}\bigl(\mathsf{Im}(f)\bigr) < \infty$. 

\smallskip
\noindent
We also get an isomorphism  $W \rightarrow \widetilde{W}, w \mapsto \tilde{w} = w + f(w)$. Since $\widetilde{W}$ is a Lagrangian subspace of $V$, we have:
$
F\bigl(f(w), w'\bigr) + F\bigl(w, f(w')\bigr) = 0$ for all
$w, w' \in W$.
It follows  that $\mathsf{Ker}(f)  = \bigl(\mathsf{Im}(f)\bigr)^\perp \cap W$. Moreover, we obtain a bilinear pairing
$$
W/\mathsf{Ker}(f) \times \mathsf{Im}(f) \stackrel{\bar{F}}\lar \kk, \quad (\bar{w}, u) \mapsto F(w, u).
$$
It is not hard to show that  $\bar{F}$ is non-degenerate.
Let $w_1, v_1, \dots, w_n, v_n \in W$ be such that
\begin{itemize}
\item $\bigl(f(w_1), \dots, f(w_n)\bigr)$ is a basis of $\mathsf{Im}(f)$.
\item $\bigl(\bar{v_1}, \dots, \bar{v}_n\bigr)$ is a basis of 
$W/\mathsf{Ker}(f)$.
\item For all $1 \le i, j \le n$ we have: $F\bigl(v_i, f(w_j)\bigr) = \delta_{ij}$. 
\end{itemize} 
Then we have:
$
\sum\limits_{i = 1}^n F\bigl(w_j, - f(v_i)\bigr) f(w_i) = 
\sum\limits_{i = 1}^n F\bigl(f(w_j), v_i\bigr) f(w_i)   = f(w_j).
$

\smallskip
\noindent
Let $\ttau:= -\sum\limits_{i = 1}^n f(v_i) \otimes f(w_i) \in U^{\otimes 2}$. Then for any $1 \le j \le n$ we have: $f_\ttau(w_j) = f(w_j)$, hence $\mathsf{Im}(f) = \mathsf{Im}(f_\ttau)$. Since $\mathsf{Ker}(f)  = \bigl(\mathsf{Im}(f)\bigr)^\perp \cap W \subseteq \mathsf{Ker}(f_\ttau)$, it follows that $\mathsf{Ker}(f) = \mathsf{Ker}(f_\ttau)$ implying that $f = f_\ttau$. Thus, we have found $\ttau \in U^{\otimes 2}$ such that 
$\widetilde{W} = W_\ttau$. Finally, the assumption $\widetilde{W}^\perp = \widetilde{W}$ implies that $\ttau \in \wedge^2(U)$, as asserted. 
\end{proof}

\begin{theorem}\label{T:ManinTriplesTwists} Let $(\lieR,\lieR_+,\lieR_-) = ((\lieR,F),\lieR_+,\lieR_-)$ be a Manin triple determining a Lie bialgebra cobracket $\lieR_+ \stackrel{\delta}\lar \wedge^2(\lieR_+)$ and
$$
\mathsf{MT}\bigl(\lieR, \lieR_+; \lieR_-):=  \left\{\lieW \subset \lieR \left|
\begin{array}{l}
(\lieR,\lieR_+,\lieW) \,  \mbox{\rm is a Manin triple} \\
\lieW \asymp 	 \lieR_-
\end{array}
\right.\right\}.
$$
Let $\ttau \in  \wedge^2(\lieR_+)$. Then  the corresponding subspace 
$\lieR_{-}^\ttau := \bigl({\lieR_{-}}\bigr)_{\ttau} \subset \lieR$ is a Lie subalgebra if and only if $\ttau$ satisfies the twist equation (\ref{E:TwistEquation}) and
the map
$$
\Bigl\{\ttau \in  \wedge^2(\lieR_+) \Big| 
\alt\bigl((\delta \otimes \mathbbm{1})(\ttau)\bigr) -[[\ttau, \ttau]] = 0
 \Bigr\} \lar
\mathsf{MT}\bigl(\lieR, \lieR_+; \lieR_-)
$$
assigning to a tensor $\ttau \in  \wedge^2(\lieR_+)$ the subspace  $\lieR_{-}^\ttau \subset \lieR$ 
is a bijection. Moreover, the Lie bialgebra cobracket 
$\lieR_+ \stackrel{\delta_\ttau}\lar  \wedge^2(\lieR_+)$ is determined by the Manin triple $\lieR = \lieR_+ \dotplus \lieR_{-}^\ttau$. 
\end{theorem}

\begin{proof}
Let $\ttau \in \wedge^2(\lieR_+)$. Then the corresponding vector subspace $
\lieR_{-}^\ttau \subset \lieR$ is Lagrangian, $\lieR = \lieR_+ \dotplus \lieR_{-}^\ttau$ and 
$\lieR_{-}^\ttau  \asymp  \lieR_{-}$. Conversely, any such Lagrangian subspace $\lieW$
 has the form $\lieW = \lieR_{-}^\ttau$ for some uniquely determined
$\ttau \in \wedge^2\bigl(\lieR_+\bigr)$; see Proposition \ref{P:LagrangianAdvanced}.

\smallskip
\noindent 
Since $\lieR = \lieR_+ \dotplus \lieR_-^\ttau$
is a Lagrangian decomposition, 
the subspace $\lieR_-^\ttau \subset \lieR$  is closed under the Lie bracket if and only if  
$F\bigl([\tilde{w}_1, \tilde{w}_2], 
\tilde{w}_3 \bigr) = 0$ for any $\tilde{w}_1, \tilde{w}_2, \tilde{w}_3 \in 
\lieR_-^\ttau$.
 
\smallskip
\noindent
For any $w \in \lieR_-$ let $\tilde{w} = w + f_\ttau(w)$ be the corresponding element of $\lieR_-^\ttau$. 
The same computation as in  \cite[Theorem 7]{KarolinskyStolin} shows that for all $w_1, w_2, w_3 \in \lieR_-$ we have:
$$
F\bigl(w_1 \otimes w_2 \otimes w_3, [[\ttau, \ttau]] - \alt\bigl((\delta \otimes \mathbbm{1})(\ttau)\bigr)\bigr) = F\bigl([\tilde{w}_1, \tilde{w}_2], 
\tilde{w}_3 \bigr).
$$
This implies that $\lieR_-^\ttau$ is a Lie subalgebra of $\lieR$ if and only if 
$\alt\bigl((\delta_\circ \otimes \mathbbm{1})(\ttau)\bigr) -[[\ttau, \ttau]] = 0$.

\smallskip
\noindent
Since $\ttau \in  \wedge^2\bigl(\lieR_+\bigr)$, it follows that 
$
F\bigl(\partial_\ttau(a), w_1 \otimes w_2\bigr) = F\bigl(a, 
\bigl[w_1, f_\ttau(w_2)\bigr] + \bigl[f_\ttau(w_1), w_2\bigr]\bigr)
$
for any $a \in \lieR_+$ and $w_1, w_2 \in \lieR_-$. A straightforward computation shows that
$$
F\bigl(\delta_t(a), \tilde{w}_1 \otimes \tilde{w}_2\bigr) = 
F\bigl(a, 
\bigl[\tilde{w}_1, \tilde{w}_2\bigr]\bigr) \quad \mbox{for any} \, a \in \lieR_+
\; \mbox{\rm and} \;  
\tilde{w}_1, \tilde{w}_2 \in \lieR_-^\ttau,
$$
implying that $\lieR_+ \stackrel{\delta_\ttau}\lar  \wedge^2\bigl(\lieR_+\bigr)$ is determined  by the Manin triple $\lieR = \lieR_+ \dotplus \lieR_-^\ttau$.
\end{proof}

\section{Review of  affine Lie algebras and twisted loop algebras}\label{S:StandardLiebialgStructure}

\subsection{Basic facts on affine Lie algebras}
Let $\widehat{\Gamma}$ be an affine Dynkin diagram, $|\widehat{\Gamma}| = r + 1$ and $A  \in \Mat_{(r+1) \times (r+1)}(\ZZ)$ be the corresponding  generalized Cartan matrix.
We choose a labelling of vertices of $\widehat{\Gamma}$ as in \cite[Section 17.1]{Carter}. 
The corresponding affine Lie algebra 
$\widetilde\lieG = \widetilde{\lieG}_{\widehat{\Gamma}} = \widetilde\lieG_{A}$  is by definition the Lie algebra over $\CC$ generated by the elements $e_0^\pm, \dots, e_r^\pm, \tilde{h}_0, \dots, \tilde{h}_r$ subject to the following relations: 
$$
\left\{
\begin{array}{cc}
\begin{array}{l}
[\tilde{h}_i, \tilde{h}_j]    = 0\\
\smallskip
[e_{i}^{+}, e_{j}^{-}] =   \delta_{ij}\,  \tilde{h}_i \\
\smallskip
[\tilde{h}_i, e_j^{\pm}]  =  \pm a_{ij} \, e_j^{\pm} 
\end{array}
& \; \mbox{for  all} \; 0 \le i, j \le r\\
\end{array}
\right.
$$
and
$$
\left\{
\begin{array}{lcl}
\ad_{e_i^{\pm}}^{1 - a_{ij}}(e_j^\pm) = 0 & \mbox{for all} & 0 \le i \ne j \le r  
\end{array}
\right.
$$
see \cite{Carter, Kac}. Recall the following standard facts.

\smallskip
\noindent
1. There exist unique vectors $\vec{k} = (k_0, \dots, k_r)$ and $\vec{l} = (l_0, \dots, l_r)$  in $\NN^{r+1}$ such that 
$$
\gcd(k_0, \dots, k_r) = 1 = \gcd(l_0, \dots, l_r)
$$
and 
$
\vec{l} A = \vec{0} = A \vec{k}^t$; see \cite[Section 17.1]{Carter}. 
\begin{itemize}
\item For any $0 \le i \le r$ let $d_i:= \dfrac{k_i}{l_i}$. Then for any $0 \le i, j \le r$ we have: $a_{ij} d_j = a_{ji} d_i$. In other words, the matrix $D^{-1} A$ is symmetric, where $D:= \diag\bigl(d_0, \dots, d_r\bigr)$. 
\item  
The center
of the Lie algebra $\widetilde{\lieG}$ is one-dimensional and generated by   the element $c: = l_0 \tilde{h}_0+ \dots + l_r \tilde{h}_r$; see \cite[Proposition 17.8]{Carter}.
\end{itemize}

\smallskip
\noindent
2. There exists a symmetric invariant bilinear form 
$\widetilde\lieG \times \widetilde\lieG \stackrel{\widetilde{B}}\lar \CC$ (called \emph{standard form}) given on the generators  by the following formulae:
\begin{equation}
\left\{
\begin{array}{cc}
\begin{array}{l}
\widetilde{B}(\tilde{h}_i, x_j^\pm) = 0  \\
\smallskip
\widetilde{B}(\tilde{h}_i, \tilde{h}_j)  \; = \; d_j a_{ij}\\
\smallskip
\widetilde{B}(e_{i}^{\pm}, e_{j}^{\mp}) =   d_i \delta_{ij} \\
\smallskip
\widetilde{B}(e_{i}^{\pm}, e_{j}^{\pm}) =   0  \\
\end{array}
& \; \mbox{for  all} \; 0 \le i, j \le r\\
\end{array}
\right.
\end{equation}
see \cite[Theorem 16.2]{Carter}. This form is degenerate and its radical is the vector space $\CC c$.

\smallskip
\noindent
3. The Lie algebra 
 $\widetilde\lieG $ carries a so-called \emph{standard} Lie bialgebra cobracket $\widetilde\lieG \stackrel{\tilde\delta_\circ}\lar \wedge^2 \widetilde\lieG$ (discovered by Drinfeld \cite{Drinfeld})  given by the formulae
$$
\tilde\delta_\circ(e_i^\pm) = \dfrac{1}{d_i} \tilde{h}_i \wedge e_i^\pm \; \mbox{and} \; 
\tilde\delta_\circ(\tilde{h}_i) = 0 \; \mbox{for all} \; 0 \le i \le r.
$$
4. Consider the Lie algebra  $\lieG =  \widetilde{\lieG}/\langle c\rangle$. Then we have   the induced  
 \emph{non-degenerate} symmetric invariant bilinear form  
$\lieG \times \lieG \stackrel{B}\lar  \CC$, which will be also called standard, as well as  a Lie bialgebra cobracket $\lieG \stackrel{\delta_\circ}\lar \wedge^2 \lieG$,  given by the formulae
\begin{equation}\label{E:standardbialgstructure}
\delta_\circ(e_i^\pm) = \dfrac{1}{d_i} h_i \wedge e_i^\pm \; \mbox{and} \; 
\delta_\circ(h_i) = 0 \; \mbox{for all} \; 0 \le i \le r,
\end{equation}
where $h_i$ denotes the image of $\tilde{h}_i$ in $\lieG$. 

\medskip
\noindent
5. Denote by $\lieG_\pm = \llangle e_0^\pm, \dots, e_r^\pm\rrangle$ the Lie subalgebra of $\lieG$ generated by the elements 
$e_0^\pm, \dots, e_r^\pm$ and put $\lieH := \bigl\langle h_1, \dots, h_r\bigr\rangle_\CC$. Then we have the triangular decomposition $\lieG = \lieG_+ \oplus \lieH \oplus \lieG_-$ as well as the following symmetric non-degenerate invariant bilinear form:
\begin{equation}
\bigl(\lieG \times \lieG\bigr) \times \bigl(\lieG \times \lieG\bigr) 
\stackrel{F}\lar \CC, \; \bigl((a', b'), (a'', b'')\bigr) \mapsto B(a', b') - B(a'', b'').
\end{equation}
We identify $\lieG$ with the diagonal  $\bigl\{(a, a) \, \big| \, a \in \lieG \bigr\} \subset \lieG \times \lieG$ and put
\begin{equation}\label{E:StandardW}
\lieH' = \bigl\{(a, -a) \, \big| \, a \in \lieH \bigr\} \quad
\mbox{\rm and} \quad 
\overline{\lieW}^\circ := \bigl(\lieG_+ \times \lieG_-\bigr) + \lieH'.
\end{equation}

\smallskip
\noindent
The following result is  essentially  due to Drinfeld \cite[Example 3.2]{Drinfeld};  see also 
\cite[Example 1.3.8]{ChariPressley} for a detailed proof.
\begin{theorem}\label{T:Drinfeld}
We have a Manin triple 
\begin{equation}\label{E:StandardManinTriple}
\lieG \times \lieG  = \lieG \dotplus \overline{\lieW}^\circ,
\end{equation}
which  moreover determines the standard Lie bialgebra cobracket $\lieG 
\stackrel{\delta_\circ}\lar \wedge^2\lieG$.
\end{theorem}

\subsection{Basic facts on  twisted loop algebras}\label{SS:BasicsTwistedLoops}

\noindent
Let $\lieg$ be  a finite dimensional complex simple Lie algebra of dimension $q$, $\lieg \times \lieg 
\stackrel{\kappa}\lar \CC$  its  Killing form, $\sigma \in \Aut_{\CC}(\lieg)$ an automorphism of order $m$ and $\varepsilon = \exp\Bigl(\dfrac{2\pi i}{m}\Bigr)$.  For  any  $k  \in \ZZ$,  let
$
\lieg_{k} := \left\{x \in \lieg \, \big| \, \sigma(x) = \varepsilon^k x \right\}.
$
Then we have a direct sum decomposition
$
\lieg = \oplus_{k = 0}^{m-1} \lieg_{k}$. 
First note the following easy and well-known fact. 

\begin{lemma}\label{L:Orthogonality}  For any $k, l \in \ZZ$,  the  pairing $\lieg_{k} \times \lieg_{l} \stackrel{\kappa}\lar \CC$ is non-zero if and only if $m | (k + l)$. Moreover, the pairing
$\lieg_{k} \times \lieg_{-k} \stackrel{\kappa}\lar \CC$ is non-degenerate for any $k \in  \ZZ$.
\end{lemma}

\begin{proof}
Let $a \in \lieg_{k}$ and $b \in \lieg_{l}$. Then we have:
$
\kappa(a, b) = \kappa\bigl(\sigma(a), \sigma(b)\bigr) = \varepsilon^{k+l} \kappa(a, b),
$
implying the first statement. The second statement follows from the first one and non-degeneracy of the form $\kappa$. 
\end{proof}

\begin{corollary}
The Casimir element  $\gamma \in \lieg \otimes \lieg$ (with respect to the Killing form $\kappa$) admits the decomposition 
$
\gamma = \sum\limits_{k = 0}^{m-1} \gamma_{k}
$
with components  $\gamma_{k} \in \lieg_{k} \otimes \lieg_{-k}$. 
\end{corollary}

\smallskip
\noindent
Let $\overline{\lieL} = \lieg[z, z^{-1}]$ be the loop algebra of $\lieg$, where
$\bigl[az^k, b z^l\bigr] := [a, b] z^{k+l}$ for any $a, b \in \lieg$ and $k, l \in \ZZ$. 
The \emph{twisted loop algebra}  is the following   Lie subalgebra of $\overline{\lieL}$:
\begin{equation}
\lieL = \lieL(\lieg, \sigma)  := \bigoplus\limits_{k \in \ZZ} \lieg_{k}  z^k .
\end{equation}
Let $\Inn(\lieg)$ be the group of inner automorphisms of $\lieg$. It is a normal subgroup of  the group $\Aut(\lieg)$ of Lie algebra automorphisms of $\lieg$. The quotient $\Out(\lieg):= \Aut(\lieg)/\Inn(\lieg)$ can be identified with the group $\Aut(\Gamma)$ of automorphisms of the Dynkin diagram $\Gamma$ of $\lieg$; see e.g. \cite[Chapter 4]{OnishchikVinberg}. Moreover, 
given  two automorphisms $\sigma, \sigma' \in \Aut(\lieg)$ of finite order, the corresponding twisted loop algebras $\lieL(\lieg, \sigma)$ and $\lieL(\lieg, \sigma')$ are isomorphic if and only if the classes of $\sigma$ and $\sigma'$ in 
$\Out(\lieg)$ are conjugate; see \cite[Chapter 8]{Kac} or 
\cite[Section X.5]{Helgason}. 

\smallskip
\noindent
Let   $\overline{R} = \CC[z, z^{-1}]$ and $R = \CC[t, t^{-1}]$, where 
$t = z^m$.

\begin{proposition}\label{P:basicsonloops}
The following results are true.
\begin{enumerate}
\item[(a)] $\lieL$ is a free module of rank $q$ over $R$.  Moreover, for any $\lambda \in \CC$, we have an isomorphism of Lie algebras
$\bigl(R/(t - \lambda)\bigr) \otimes_R \lieL \cong \lieg$. 
\item[(b)] Consider the symmetric $\CC$-bilinear form
\begin{equation}\label{E:FormB}
\lieL \times \lieL \stackrel{B}\lar \CC, \quad B(a z^k, bz^l) = \kappa(a, b) \delta_{k+l, 0}.
\end{equation}
Then $B$ is non-degenerate and invariant. Moreover, the rescaled bilinear form $mB$ coincides with the composition
$
\lieL \times \lieL \stackrel{K}\lar  R \xrightarrow{\res_0^\omega}  \CC,
$
where $K$ is the Killing form of $\lieL$, $\omega = \dfrac{dt}{t}$ and $\res_0^\omega(f) = \res_0(f \omega)$ for any $f \in R$.

\item[(c)] For any $n \in \NN$, the $(n+1)$-fold tensor product 
$\lieL^{\otimes (n+1)}$ does not contain any non-zero ad-invariant elements. 
\end{enumerate}
\end{proposition}

\begin{proof}
(a) Let $(f_1, \dots, f_q)$ be any  basis of the  vector space 
$\bigoplus\limits_{j = 0}^{m-1} \lieg_{j}  z^j$.   Then for any  $f \in \lieL$ there exist unique  $p_1, \dots, p_q \in R$ such that $f = p_1 f_1 + \dots + p_q f_q$. 
Hence, $\lieL$ is a free $R$-module  of rank $q$. 

\smallskip
\noindent
The canonical map $\overline{R} \otimes_R \lieL \stackrel{\pi}\lar \overline{\lieL}, z^n \otimes a z^k \mapsto a z^{n+k}$ is an $\overline{R}$--linear surjective morphism of Lie algebras. Since $\overline{R} \otimes_R \lieL$ and $\overline{\lieL}$ are both free $\overline{R}$--modules of the same rank, 
$\pi$ is an isomorphism. Finally, the extension $R \subset \overline{R}$ is unramified, hence for any $\mu \in \CC^*$ the following canonical linear maps
$$R/(t - \mu^m) \otimes_R \lieL  \rightarrow
 \overline{R}/(z-\mu) \otimes_{R} \lieL \rightarrow
 \overline{R}/(z-\mu)
 \otimes_{\overline{R}} \overline{R} \otimes_R \lieL \rightarrow
\overline{R}/(z-\mu) \otimes_{\overline{R}} \overline{\lieL} \rightarrow \lieg
$$
are isomorphisms of Lie algebras. 

\smallskip
\noindent
(b) Let $\overline{\lieL} \times \overline{\lieL} \stackrel{\overline{K}}\lar  \overline{R}$ be the Killing form of $\overline{\lieL}$. Then we have:
$
\overline{K}(a z^k, bz^l) = \kappa(a, b) z^{k+l}.
$
The isomorphism of Lie algebras $\overline{R} \otimes_R \lieL \cong  \overline{\lieL}$ as well as invariance of the Killing form under automorphisms  imply that the following diagram is commutative: 
$$
\xymatrix{
\lieL \times \lieL  \ar@{_{(}->}[d]  \ar[rr]^-{K} & & R \ar@{^{(}->}[d] \\
\overline{\lieL} \times \overline{\lieL}   \ar[rr]^-{\overline{K}} & & \overline{R}.
}
$$
Since  $\omega = \dfrac{dt}{t} = m \dfrac{dz}{z}$, we get  the second statement. 

\smallskip
\noindent
(c) Assume that $\ttau \in \lieL^{\otimes (n+1)}$ is such that 
\begin{equation}\label{E:invariancecond}
\bigl[x \otimes 1 \otimes \dots \otimes 1 + \dots + 1 \otimes \dots \otimes 1 \otimes x, \ttau  \bigr] = 0
\end{equation}
for all $x \in \lieL$. Let $(b_k)_{k \in \NN}$ be an orthonormal basis of $\lieL$ with respect to the form $B$. Then we can express $\ttau$  as a sum
$
\ttau = \sum\limits_{j_1, \dots, j_n =1}^s a_{j_1 \dots j_n} \otimes b_{j_1} \otimes \dots \otimes b_{j_n}.
$
Consider the vector space
$
J := \left\langle a_{j_1\dots j_n} \, \left| \; 1 \le j_1, \dots, j_n \le s
\right.\right\rangle_{\CC} \subset \lieL
$
For any $1 \le i_1, \dots, i_n \le s$, we apply the map
$$\mathbbm{1}_{\lieL} \otimes B(b_{i_1}, -) \otimes \dots \otimes B(b_{i_n}, -): \lieL^{\otimes (n+1)} \lar \lieL$$ to the identity (\ref{E:invariancecond}). It follows that 
$
\bigl[x, a_{i_1 \dots i_n}] \in J
$
 for any $x \in \lieL$, implying that $J$ is an ideal in $\lieL$. However, $\lieL$ does not contain any non-zero finite-dimensional ideals; see \cite[Lemma 8.6]{Kac}. Hence, $\ttau = 0$, as asserted. 
\end{proof}

\smallskip
\noindent
A proof of the following key result can be found in \cite[Lemma 8.1]{Kac}.
\begin{proposition} 
 The algebra $\lieg_{0} = \left\{a \in \lieg \, \big| \, \sigma(a) =  a \right\}$ is non-zero and reductive. 
\end{proposition}

\begin{remark}\label{R:Triangular} In what follows, we choose a Cartan subalgebra $\lieh \subset \lieg_0$. Let 
$\Delta_{0}$ be the root system of $(\lieg_{0}, \lieh)$. We fix a polarization
$\Delta_{0} = \Delta_{0}^{+} \sqcup \Delta_{0}^{-}$, which 
gives a triangular decomposition $
\lieg_{0} = \lieg_{0}^+ \oplus \lieh \oplus \lieg_{0}^-.
$
One can show that
$
\tilde\lieh:= \bigl\{a \in \lieg \big| [a, h] = 0 \; \mbox{for all} \; h \in \lieh\bigr\}
$
is a Cartan subalgebra of $\lieg$; see \cite[Lemma 8.1]{Kac}. However, in general $\tilde\lieh \ne \lieh$. 
The algebra $\lieg_{0}$ is simple if $\sigma$ is a so-called diagram automorphism of $\lieg$; see \cite[Chapter 8]{Kac}. \hfill $\lozenge$
\end{remark}

\medskip
\noindent
Now we  review the structure theory of twisted loop algebras as well as their relations with affine Lie algebras. For that we  need  the following notions,  notation and facts.

\medskip
\noindent
1. For any $j \in \ZZ$ we put: $\lieL_j = \lieg_{j} z^j \subset \lieL$. Since 
$\bigl[\lieg_{0}, \lieg_{j} \bigr] \subseteq \lieg_{j}$, it follows that $\bigl[\lieg_{0}, \lieL_j \bigr] \subseteq \lieL_j$, too.   A pair $(\alpha, j) \in \lieh^\ast \times \, \ZZ$ is a \emph{root} of 
$(\lieL, \lieh)$ if
$$
\lieL_{(\alpha, j)} := \bigl\{x \in \lieL_j \, \big| \, [h, x] = \alpha(h)x \; \mbox{\rm for all} \; h \in \lieh \bigr\} \ne 0. 
$$
In our convention,  $(0, 0)$ is a root of $(\lieL, \lieh)$. Note that   $\lieL_{(0, 0)} := \lieh$. 

\smallskip
\noindent
Let  $\Phi$ be the set of all roots of $(\lieL, \lieh)$. It is clear that 
$(-\alpha, -j), (\alpha, j + km)  \in \Phi$ for all
$k \in \ZZ$ and $(\alpha, j) \in \Phi$.

\smallskip
\noindent
2.
For any $(\alpha, j), (\alpha', j') \in \lieh^* \times \ZZ$ we put: 
$(\alpha, j) + (\alpha', j') = (\alpha+ \alpha', j + j')$. We have:
$$
\left[\lieL_{(\alpha, j)},  \lieL_{(\alpha', j')}\right] \subseteq 
\lieL_{(\alpha+ \alpha', j + j')}.
$$
A root $(\alpha, j)$ is called real if $\alpha \ne 0$ and imaginary otherwise. 
There exists $m' | m$  such that any imaginary root has the form $(0, k m')$ for some $k \in \ZZ$.  For any real root $(\alpha, j) \in \Phi$ we have: 
$
\dim_{\CC}\bigl(\lieL_{(\alpha, j)}\bigr) = 1$ (see e.g.  \cite[Lemma X.5.4']{Helgason}). A formula for $\dim_{\CC}\bigl(\lieL_{(0, km')}\bigr)$ can be found in   \cite[Corollary 8.3]{Kac}. 

\smallskip
\noindent
Since $\lieg_{0}$ is a reductive Lie algebra, we have a direct sum decomposition
$
\lieL = \bigoplus_{(\alpha, j) \in \Phi} \lieL_{(\alpha, j)}.
$
The sets of positive and negative roots of $(\lieL, \lieh)$ are defined as follows:
\begin{equation}\label{E:LoopsPositiveRoots}
\Phi_\pm := \left\{(\alpha, j) \in \Phi  \, \big|\, \pm j > 0 \right\} \sqcup
\left\{(\alpha, 0) \in \Phi  \, \big|\, \pm \alpha \in \Delta_{0}^+\right\},
\end{equation}
where $\Delta_{0}^+$ is the set positive roots  of  $(\lieg_{0}, \lieh)$. We have: $\Phi = \Phi_+ \sqcup \Phi_- \sqcup\bigl\{(0, 0)\bigr\}$ and $\Phi_- = - \Phi_+$. 

\smallskip
\noindent
3. Since the bilinear form $\lieL \times \lieL \stackrel{B}\lar \CC$ is invariant and non-degenerate, analogously to Lemma \ref{L:Orthogonality} we obtain the following results:
\begin{itemize}
\item  The pairing $\lieL_{(\alpha, j)} \times \lieL_{(\alpha', j')} \stackrel{B}\lar \CC$ is zero unless $(\alpha, j) + (\alpha', j') = (0, 0)$. 
\item For any $(\alpha, j) \in \Phi$, the pairing
$\lieL_{(\alpha, j)} \times \lieL_{(-\alpha, -j)} \stackrel{B}\lar \CC$ is non-degenerate.
\item In particular,  since $B\big|_{\lieh \times \lieh} = \kappa|_{\lieh \times \lieh}$, the pairing
$
\lieh \times \lieh \stackrel{\kappa}\lar \CC
$
is non-degenerate. 
\end{itemize}

\smallskip
\noindent
4. The set $\Pi$ of simple roots of $(\lieL, \lieh)$ is defined as follows:
\begin{equation}\label{E:simplerootsLoop}
\Pi: = 
 \left\{(\alpha, j) \in \Phi_+  \, \big| \, (\alpha - \beta, j - i) \notin \Phi_+ \; \mbox{for all} \, (\beta, i) \in \Phi_+\right\}.
\end{equation}
All elements of $\Pi$ are real roots and we have: $|\Pi| = r + 1$; see \cite[Lemma X.5.7 and Lemma X.5.9]{Helgason}. We use the following notation:
\begin{equation}\label{E:simplerootsloops}
\Pi = \bigl\{(\alpha_0, s_0), \dots, (\alpha_r, s_r)\bigr\}.
\end{equation}

\smallskip
\noindent
5. Since the pairing $
\lieh \times \lieh \stackrel{\kappa}\lar \CC
$
is non-degenerate, we get the induced isomorphism of vector spaces
$\lieh \stackrel{\tilde\kappa}\lar \lieh^\ast$. Abusing the notation, let  $
\lieh^\ast \times \lieh^\ast \stackrel{\kappa}\lar \CC
$
be the transfer of the Killing form $\kappa$ under the isomorphism 
$\tilde\kappa$.

\begin{itemize}
\item For any $0 \le i \le r$ we put: $y_i := \dfrac{2}{\kappa(\alpha_i, \alpha_i)} \bigl(\tilde\kappa\bigr)^{-1}(\alpha_i) \in \lieh$.
\item For any $0 \le i ,j \le r$ we set:
\begin{equation}\label{E:CartanCoeff}
a_{ij} := 2 \dfrac{\kappa(\alpha_i, \alpha_j)}{\kappa(\alpha_i, \alpha_i)}.
\end{equation}
 It turns out that $
a_{ij} \in \ZZ
$ and $A = (a_{ij}) \in \Mat_{(r+1) \times (r+1)}(\ZZ)$ is a generalized Cartan matrix of affine type;  see \cite[Lemma X.5.6 and Lemma X.5.11]{Helgason}.
In particular, we have: $\rk(A) = r$. 
\item For every $0 \le i \le r$ one can choose $x_i^\pm \in \lieL_{\pm(\alpha_i, s_i)}$ such that the following relations are satisfied for all $0 \le i, j \le r$:
$$
\left\{
\begin{array}{l}
[y_i, y_j]  =  0 \\
\smallskip
[x_{i}^{+}, x_{j}^{-}] =   \delta_{ij}\,  y_i \\
\smallskip
[y_i, x_j^{\pm}]  =  \pm a_{ij} \, x_j^{\pm}. 
\end{array}
\right.
$$
 Moreover, for any $0 \le i \ne j \le r$ we have:
$$
\ad_{x_i^{\pm}}^{1 - a_{ij}}(x_j^\pm) = 0
$$
and the elements 
$x_0^\pm, \dots, x_r^\pm, y_0, \dots, y_r$ generate $\lieL$; see 
\cite[Section X.5]{Helgason}.
\item 
Let $\lieG = \lieG_{A}$. 
A theorem of Gabber and Kac asserts that the linear map
\begin{equation}\label{E:GabberKac}
\lieG \stackrel{\varphi}\lar \lieL, \; e_i^\pm \mapsto, x_i^\pm, \, h_i \mapsto y_i
\end{equation}
is an isomorphism of Lie algebras, which identifies  both standard forms 
 on $\lieG$ and on  $\lieL$ (up to an appropriate rescaling);  see 
 \cite[Theorem 8.5]{Kac}.
\end{itemize}

\begin{corollary}\label{C:StandardLieBialgLoops}
We have a Lie bialgebra cobracket  $\lieL \stackrel{\delta_\circ}\lar \wedge^2 \lieL$ (also called \emph{standard}),  given by the formulae
\begin{equation}\label{E:standardbialgstructure2}
\delta_\circ(x_i^\pm) = \dfrac{\kappa(\alpha_i, \alpha_i)}{2} y_i \wedge x_i^\pm  \;\;  \mbox{and} \;\; 
\delta_\circ(y_i) = 0 \;\;  \mbox{for all} \; \;  0 \le i \le r.
\end{equation}
This cobracket is determined by the Manin triple, which is isomorphic to  (\ref{E:StandardManinTriple}).
\end{corollary}

\subsection{Bounded Lie subalgebras of twisted loop algebras}
For any $0 \le i \le r$, the corresponding (positive) maximal parabolic Lie subalgebras $\lieP_i \subset \lieL$ is defined as follows:
$$
\lieP_i := \llangle h_0, \dots, h_r, x_0^+, \dots, x_r^+, x_0^-, \dots, 
\widehat{x_i}^-, \dots, x_r^- \rrangle.
$$
A similar argument to \cite[Lemma 1.5]{KacWang} implies that 
\begin{equation}\label{E:ParabDecomp}
\lieP_i  = \lieB_+ \oplus \bigl(\oplus_{(\alpha, j) \in \Phi_i^-} \lieL_{(\alpha, j)} \bigr),
\end{equation}
where 
$
\Phi_i^- :=  \Phi_- \cap \bigl\langle \alpha_0, \dots, \hat{\alpha}_i, \dots, \alpha_r\bigr\rangle_{\NN_0^-}
$
and 
$\lieB_+ := (\lieg_0^+ \oplus \lieh) \oplus \bigl(\bigoplus\limits_{k = 1}^\infty \lieg_k z^k \bigr)$ is a positive Borel subalgebra of $\lieL$.

\begin{lemma}\label{L:MaximalParabolics}
For any $0 \le i \le r$ we have: $t \lieP_i \subseteq \bigl(\lieP_i\bigr)^\perp$, where the orthogonal space is taken with respect to the bilinear form 
$B$, given by the formula (\ref{E:FormB}).  
\end{lemma}

\begin{proof} 
 Since the roots $\alpha_0, \dots, \hat\alpha_i, \dots, \alpha_r$ are linearly independent elements of $\lieh^\ast$, it follows that  $(0, - km') \notin \Phi_i^-$ for all $k \in \NN$. 
Let $\Phi_i := \Phi_+ \sqcup\{(0, 0)\} \sqcup \Phi_i^-$. Then we have:
$$
\lieP_i  = \bigoplus\limits_{(\alpha, j) \in \Phi_i} \lieL_{(\alpha, j)}
\quad \mbox{\rm and} \quad t \lieP_i  = \bigoplus\limits_{(\beta, k) \in \Phi_i} \lieL_{(\beta, k +m)}.
$$
Let  $(\alpha, j), (\beta, k) \in \Phi_i$, $x \in \lieL_{(\alpha, j)}$ and $y \in \lieL_{(\beta, k +m)}$ are such that  $B(x, y) \ne 0$. Then we have: $\alpha = -\beta$ and $j = - k - m$. 

\smallskip
\noindent
\underline{Case 1}. Assume that  $\alpha = 0$.  Then $(\alpha, j) \in 
\Phi_+ \sqcup\{(0, 0)\}$ and $(\beta, k) = (0, - j - m) \in \Phi_i^-$ is a negative imaginary root. Contradiction. 

\smallskip
\noindent
\underline{Case 2}. Assume that $(\alpha, j)$ is a real root. Then there exist
 $x \in \lieL_{(\alpha, j)}$ and $y \in \lieL_{(\beta, k +m)}$ such that 
 $[x, y] \ne 0$; see \cite[Lemma X.5.5']{Helgason}. Hence, 
$
 \lieL_{(0, -m)} \cap \lieP_i \ne 0.
$
It follows from the decomposition (\ref{E:ParabDecomp}) that $(0, -m) \in \Phi_i^-$. Contradiction. 

\smallskip
\noindent
We have  shown that the pairing $t \lieP_i\times \lieP_i \stackrel{B}\lar \CC$ is zero, what implies the claim. 
\end{proof}

\smallskip
\noindent
For any $n \in \ZZ$ we put: $\lieL_{\ge n} := t^n \lieL_{\ge 0}$, where 
$\lieL_{\ge 0}:= \bigoplus\limits_{j \ge 0} \lieL_j$. 
 Note that for any $n \in \NN$ we have: $\bigl(\lieL_{\ge n}\bigr)^\perp \subseteq 
\lieL_{\ge -n}$.

\begin{definition}
 A Lie subalgebra $\lieO \subseteq \lieL$ is \emph{bounded} if 
$
\lieL_{\ge n} \subseteq \lieO \subseteq \lieL_{\ge - n}
$
for some $n \in \NN$.
\end{definition}

\smallskip
\noindent
Let $\widetilde\lieL = \lieL \dotplus \CC c$ be a central extension of $\lieL$ with the Lie bracket  given by the formulae 
\begin{equation}\label{E:CommRelationsCentralExt}
[az^k,  b z^l] := [a, b] z^{k+l} + k \delta_{k+l, 0}\,  \kappa(a, b)\,  c \quad \mbox{and} \quad [az^k, c] = 0
\end{equation}
for all  $k, l \in \ZZ$, $a \in \lieg_k$ and $b \in \lieg_l$. Let $A \in \Mat_{(r+1) \times (r+1)}(\ZZ)$ be the generalized Cartan matrix of affine type, given by (\ref{E:CartanCoeff}) and  
$\doublewidetilde{\lieG} = \doublewidetilde{\lieG}_A$ be the corresponding affine Kac--Moody Lie algebra. Then $\doublewidetilde{\lieG}$ has one-dimensional center $\lieZ$,  $\widetilde{\lieG} = \bigl[\doublewidetilde{\lieG}, \doublewidetilde{\lieG}\bigr]$ and  $\lieG = \widetilde{\lieG}/\lieZ$. The Gabber--Kac isomorphism 
$\lieG \stackrel{\varphi}\lar \lieL$ given by (\ref{E:GabberKac})  extends
to an isomorphism of Lie algebras $\widetilde\lieG \stackrel{\tilde\varphi}\lar \widetilde\lieL$. The entire picture can be  summarized in the following 
commutative diagram of Lie algebras and Lie algebra homomorphisms:
\begin{equation}\label{E:Panorama}
\begin{array}{c}
\xymatrix{
\doublewidetilde{\lieG} & \ar@{_{(}->}[l] \widetilde{\lieG} \ar@{->>}[r] \ar[d]_{\tilde\varphi} & \lieG \ar[d]^-\varphi\\
& \widetilde{\lieL} \ar@{->>}[r] & \lieL. 
}
\end{array}
\end{equation}
For $0\le i \le r$, let $\widetilde\lieG_+ \subset \widetilde\lieP_i := \lieP_i \dotplus 
\CC c \subset \doublewidetilde\lieG$ be the corresponding maximal parabolic Lie subalgebra.

\begin{proposition}\label{P:OrdersinLoops}
Let $\lieO \subseteq \lieL$ be a bounded Lie subalgebra. Then there exists an $R$-linear  automorphism $\phi$ of $\lieL$ and $0 \le i \le r$ such that $\lieO \subseteq \phi\bigl(\lieP_{i}\bigr)$. 
\end{proposition}

\begin{proof} 
Let $n \in \NN$ be such that $
\lieL_{\ge n} \subseteq \lieO \subseteq \lieL_{\ge - n}
$
and $\lieI := t^{2n+1} \lieO$.  Obviously, $\lieI$ is a Lie  ideal in $\lieO$ and  
$
\lieL_{\ge (3n+1)} \subseteq \lieI \subseteq \lieL_{\ge (n+1)}.
$
We can view $\lieI$ and $\lieL$ as vector subspaces in $\widetilde\lieL$.

\smallskip
\noindent
Let $\widetilde\lieO := \lieO \dotplus \CC c$.  Since $\lieI \subseteq \lieL_{\ge (n+1)}$ and $\lieO \subseteq \lieL_{\ge - n}$, the relations (\ref{E:CommRelationsCentralExt}) imply that $\bigl[x, y\bigr]_{\lieL} = \bigl[x, y\bigr]_{\widetilde\lieL}$ for all $x \in \lieI$ and 
$y \in \lieO$. Hence,  $\lieI \subset \widetilde{\lieO}$ is a Lie ideal with respect to the Lie bracket $\bigl[-,-\bigr]_{\widetilde\lieL}$. 
Embedding  $\widetilde\lieL$ into  $\doublewidetilde{\lieG}$ via  $\tilde{\varphi}$, we see that $\lieI \subseteq \widetilde{\lieG}_+$ and 
$\dim_{\CC}\bigl(\widetilde{\lieG}_+/\lieI\bigr) < \infty$.

\smallskip
\noindent
By   \cite[Proposition 2.8]{KacWang},  there exists an inner automorphism $\tilde\psi$ of $\widetilde\lieG$ and $0 \le i \le r$ such that 
$
\bigl[\widetilde\lieP_i , \tilde\psi(\lieO)\bigr] \subseteq \widetilde\lieP_i.
$
\smallskip
\noindent
According to \cite[Lemma 1.5]{KacWang}, for any Lie subalgebra $\widetilde\lieP \subset \doublewidetilde{\lieG}$ containing  $\widetilde{\lieB}_+$, there exists $0 \le i \le r$ such that $\widetilde\lieP \subseteq \widetilde{\lieP}_i$. Since the only proper ideals of $\doublewidetilde{\lieG}$ are
$\widetilde{\lieG}$ and $\lieZ$ (see e.g. ~\cite[Section 1.2]{KacWang}), we deduce from 
maximality of $\widetilde{\lieP}_i$ that
\begin{equation}\label{E:NormalizerParabolic}
N_{\doublewidetilde{\lieG}}\bigl(\widetilde\lieP_i\bigr) := \bigl\{x \in \doublewidetilde{\lieG} \, \big| \,
[x, y] \in \widetilde\lieP_i\; \mbox{for all} \, y \in \widetilde{\lieP}_i \bigr\} = 
\widetilde{\lieP}_i.
\end{equation}
It follows that 
\begin{equation}\label{E:Include}
\tilde\psi(\widetilde\lieO) \subseteq 
N_{\doublewidetilde{\lieG}}\bigl(\widetilde{\lieP}_i\bigr)  = 
\widetilde{\lieP}_i.
\end{equation}
Consider the automorphism $\lieG \stackrel{\psi}\lar \lieG$ induced by $\tilde\psi$. Since $\tilde\psi$ is inner, $\psi$ is $R$-linear.  
Applying to (\ref{E:Include}) the projection  $\widetilde\lieG \twoheadarrow \lieG$ and identifying $\lieG$ with $\lieL$, we finally  end up with an inclusion 
$\lieO \subseteq \phi\bigl(\lieP_{i}\bigr)$, where $\phi = \psi^{-1}$. 
\end{proof}

\begin{theorem}\label{P:CoisotropicSubalgebras}
Let $\lieO \subseteq \lieL$ be a bounded coisotropic Lie subalgebra of $\lieL$. Then we have: $t \lieO \subseteq \lieO^\perp$, i.e.~$\lieO$ is stable under the multiplication with the elements of $\CC[t]$. 
\end{theorem}

\begin{proof}
According to Proposition \ref{P:OrdersinLoops}, there exists $0 \le i \le r$ and 
$\phi \in \Aut_{R}(\lieL)$
 such that $\lieO \subseteq 
\phi(\lieP_i)$.  Since 
$B\bigl(\phi(f), \phi(g)\bigr) = B(f, g)$ for all $f, g \in \lieL$,  we get (applying Lemma \ref{L:MaximalParabolics}): 
$$
t \lieO \subseteq t \phi(\lieP_i) = \phi(t \lieP_i) \subseteq \phi\bigl(
\lieP_i^\perp\bigr) = \bigl(\phi(\lieP_i)\bigr)^\perp \subseteq \lieO^\perp \subseteq \lieO,
$$
as asserted. 
\end{proof}

\section{Twists of the standard Lie bialgebra structure on a twisted loop algebra}\label{S:TwistsStandardStructure}

\noindent
Recall our notation: $\lieg$ is a  simple complex Lie algebra of dimension $q$, $\sigma \in \Aut_{\CC}(\lieg)$ is an automorphism of order $m$ and $\varepsilon = \exp\Bigl(\dfrac{2\pi i}{m}\Bigr)$.  For any $k \in \ZZ$ we denote:
$$
\lieg_{k} := \left\{a \in \lieg \, \big| \, 
\sigma(a) = \varepsilon^k a \right\} 
\quad \mbox{\rm and} \quad 
\lieg_{k}^\ddagger := \left\{a \in \lieg \, \big| \, 
\sigma(a) = \varepsilon^{-k} a \right\}.
$$
Let 
$\lieL = \lieL(\lieg, \sigma) = \bigoplus\limits_{k \in \ZZ} \lieg_k z^k$ and 
$
\lieL^\ddagger = \lieL(\lieg, \sigma^{-1}) = \bigoplus\limits_{k \in \ZZ} \lieg_k^\ddagger z^k
$ 
be the corresponding twisted loop algebras and 
$\lieL \times \lieL \stackrel{B}\lar \CC$, respectively  $\lieL^\ddagger  \times \lieL^\ddagger \stackrel{B^\ddagger}\lar \CC$, be the corresponding standard bilinear forms.  Note that the linear map
\begin{equation}\label{E:DaggerMap}
\lieL \lar \lieL^\ddagger,  \quad az^k \mapsto \bigl(az^k\bigr)^\ddagger :=  az^{-k} 
\; \mbox{for any}\, k \in \ZZ \; \mbox{and}\, a \in \lieg_k
\end{equation}
is an  isomorphism of Lie algebras as well as an isometry with respect to the bilinear forms $B$ and $B^\dagger$. Let us denote 
$ \lieL_+ = \bigoplus\limits_{k \in \ZZ} \lieg_k z_{+}^k$ and $\lieL_-= \bigoplus\limits_{k \in \ZZ} \lieg_k^\ddagger z_{-}^k$. Then we put:
$$
\lieM := \lieL_+ \times \lieL_-  \cong \lieL \times \lieL^\ddagger.
$$
Note that we  have a non-degenerate invariant symmetric bilinear form
\begin{equation}\label{E:FormOnLiem}
\lieM \times \lieM \stackrel{F}\lar \CC, \; \bigl((f_+, f_-), (g_+, g_-)\bigr) \mapsto B(f_+, g_+) - B^\ddagger(f_-, g_-).
\end{equation}
We fix a triangular decomposition $\lieg_0 = \lieg_0^+ \oplus \lieh \oplus \lieg_0^- = \lieg_0^\ddagger$ and  denote:
$$
\lieB_+ := (\lieg_0^+ \oplus \lieh) \oplus \bigl(\bigoplus\limits_{k = 1}^\infty \lieg_k z_{+}^k \bigr) \quad \mbox{\rm and} \quad 
\lieB_- := (\lieg_0^- \oplus \lieh) \oplus \bigl(\bigoplus\limits_{k = 1}^\infty \lieg_k^\ddagger z_{-}^k \bigr).
$$
Let $\lieB_\pm \stackrel{\pi_\pm}\lar \lieh$ be the  canonical projections. Then we put:  
\begin{equation}
\lieW^\circ := \bigl\{(f_+, f_-) \in \lieB_+ \times \lieB_- \, \big| \, 
\pi_+(f_+) + \pi_-(f_-) = 0 \bigr\} \; \mbox{\rm and} \;
\lieD := \bigl\{(f, f^\ddagger) \, \big| \, f \in \lieL \bigr\}.
\end{equation}
Similarly to  Theorem \ref{T:Drinfeld}, we have a  Manin triple
\begin{equation}\label{E:StandardManinTripleLoops}
\lieM = \lieD \dotplus \lieW^\circ.
\end{equation}
Let 
$\lieL \stackrel{\delta_\circ}\lar  \lieL \wedge \lieL$ be the standard Lie bialgebra cobracket on $\lieL$. According to  Theorem \ref{T:Drinfeld},  
$\delta_\circ$ is determined by (\ref{E:StandardManinTripleLoops}), where we use the identification $\lieL \stackrel{\cong}\lar \lieD, f \mapsto (f, f^\ddagger)$ 

\smallskip
\noindent
For 
$\ttau = \sum\limits_{i = 1}^n a_i \otimes b_i \in \lieD^{\otimes 2}$,  let
$\lieW^\circ \stackrel{f_\ttau}\lar \lieD, \, w \mapsto \sum\limits_{i = 1}^n F(w, a_i) b_i$ be the corresponding linear map, $\delta_\ttau = \delta_\circ + \partial_\ttau$  and 
$
\lieW_\ttau := \bigl\{w + f_\ttau(w) \, |\, w \in \lieW^\circ\bigr\}.
$ 
\begin{theorem}\label{T:ManinTriplesLoopTwists} For $\ttau \in \wedge^2\lieL \cong \wedge^2\lieD$, the corresponding subspace $\lieW_\ttau \subset \lieM$ is a Lie subalgebra if and only $(\lieL, \delta_\ttau)$ is a Lie bialgebra. Moreover, 
the corresponding map
$$
\Bigl\{\ttau \in \wedge^2\lieL \Big| \, (\lieL, \delta_\ttau) \, \mbox{is a Lie bialgebra} \Bigr\} \lar
\mathsf{MT}\bigl(\lieM, \lieD; \lieW^\circ)
$$
is a bijection and  $\lieL \stackrel{\delta_\ttau}\lar  \lieL \wedge \lieL$ is determined by the Manin triple $\lieM = \lieD \dotplus \lieW_\ttau$. 
\end{theorem}

\begin{proof}
By  Proposition \ref{P:basicsonloops},  $\lieL^{\otimes 3}$ does not contain any non-zero ad-invariant elements. According to Proposition \ref{P:TwistLieBialgStructures}, 
  $(\lieL, \delta_\ttau)$ is a Lie bialgebra if and only if 
 $\ttau$ satisfies the twist equation (\ref{E:TwistEquation}). Hence, the result follows from Theorem \ref{T:ManinTriplesTwists}.
\end{proof}

\smallskip
\noindent
Let $\lieW \subset \lieM$ be a Lie subalgebra as in Theorem \ref{T:ManinTriplesLoopTwists} and $\lieW_\pm \subseteq \lieL_\pm$ be its image under the projections $\lieM \twoheadarrow \lieL_\pm$. Starting with the embedding $\lieW \subseteq \lieW_+ \times \lieW_-$, we get:
$$
\lieW^\perp_+ \times \lieW^\perp_- = 
\bigl(\lieW_+ \times \lieW_-)^\perp \subseteq \lieW^\perp = \lieW \subseteq \lieW_+ \times \lieW_-.
$$
It follows that $\lieW_\pm^\perp \subseteq \lieW_\pm$,  
$\lieW_+^\perp \times \{0\} \subseteq \lieW$ and $\{0\} \times \lieW_-^\perp \subseteq \lieW$.

\smallskip
\noindent
The assumption $\lieW \asymp 	 \lieW^\circ$ implies that there exists $n \in \NN$ such that $\bigl(\lieL_{+}\bigr)_{\ge n} \times \bigl(\lieL_{-}\bigr)_{\ge n}  \subseteq \lieW$. Hence, we obtain:
$$
\bigl(\lieL_{+}\bigr)_{\ge n} \times \bigl(\lieL_{-}\bigr)_{\ge n} 
\subseteq \lieW = \lieW^\perp \subseteq
\Bigl(\bigl(\lieL_{+}\bigr)_{\ge n} \times \bigl(\lieL_{-}\bigr)_{\ge n}\Bigr)^\perp \subseteq
\bigl(\lieL_{+}\bigr)_{\ge -n} \times \bigl(\lieL_{-}\bigr)_{\ge -n} 
$$
It follows that $\bigl(\lieL_{\pm}\bigr)_{\ge n} \subseteq \lieW_\pm \subseteq \bigl(\lieL_{\pm}\bigr)_{\ge -n}$, i.e.~$\lieW_\pm$ are   bounded coisotropic Lie subalgebras  of the twisted loop algebra $\lieL_\pm$. 

\begin{remark}
Since the linear map (\ref{E:DaggerMap})  is an isomorphism of Lie algebras, compatible with the standard bilinear forms, one can equally parametrize  twists of the standard Lie bialgebra cobracket $\lieL \stackrel{\delta_\circ}\lar  \wedge^2(\lieL)$ via Manin triples 
$$
\lieL \times \lieL = \overline{\lieD} \dotplus \overline{\lieW}, \quad \overline{\lieW} \asymp \overline{\lieW}^\circ
$$
where $\overline{\lieD} = \bigl\{(f, f) \,\big| \, f \in \lieL \bigr\}$ and $\overline{\lieW}^\circ$ is given by (\ref{E:StandardW}).  The usage of such Manin triples would be quite in  the spirit of the conventional  notation \cite{Drinfeld, ChariPressley} of  Theorem \ref{T:Drinfeld}. However,  as we shall see later on, Manin triples from  
Theorem \ref{T:ManinTriplesLoopTwists} are more natural from the 
algebro-geometric viewpoint. \hfill $\lozenge$
\end{remark}

\smallskip
\noindent
We put: $R = \CC[t, t^{-1}]$, $R_\pm = \CC[t_\pm, t_\pm^{-1}]\supset L_\pm = \CC[t_\pm]$, where $t = z^m$ and $t_\pm = z_\pm^m$. We shall use the identifications $R \stackrel{\cong}\lar R_\pm, t \mapsto t_{\pm}^{\pm 1}$.
Theorem  \ref{P:CoisotropicSubalgebras} implies that 
\begin{equation}\label{E:EmbeddingsOrders}
t_\pm \lieW_\pm \subseteq \lieW_\pm^\perp \subseteq \lieW_\pm.
\end{equation}
\begin{lemma}\label{L:OrdersFromMT1}
The following results are true.
\begin{enumerate}
\item[(a)] The Lie algebra $\lieW_\pm$ is a free module of rank $q$ over  $L_\pm$. Moreover, the canonical map
$
R_{\pm} \otimes_{L_{\pm}} \lieW_\pm \lar  \lieL_\pm
$
is an isomorphism of Lie algebras. 
\item[(b)] We have: $(t_+, t_-)\lieW = t_+ \lieW_+ \times t_- \lieW_- \subseteq \lieW$, where $(t_+, t_-)$ is the ideal in $R_+ \times R_-$ generated by $t_+$ and $t_-$. In particular, $\lieW$ is a finitely generated torsion free module over the algebra $O:= \CC[t_+, t_-]/(t_+ t_-)$. 
\item[(c)] The linear map $\lieW/(t_+, t_-)\lieW \lar  \bigl(\lieW_+/t_+ \lieW_+\bigr) \times \bigl(\lieW_-/t_- \lieW_-\bigr)$ is an injective morphism of Lie algebras, whereas both maps  $\lieW/(t_+, t_-)\lieW \lar  \lieW_\pm/t_\pm \lieW_\pm$ are surjective morphisms of Lie algebras. 
\end{enumerate}
\end{lemma}
\begin{proof} (a) We get from (\ref{E:EmbeddingsOrders}) that $\lieW_\pm$ is a $L_\pm$-submodule of $\lieL_\pm$. It follows from $\lieW_\pm \asymp 	 \lieW_\pm^\circ = \lieB_\pm$ that the canonical map $
R_{\pm} \otimes_{L_{\pm}} \lieW_\pm \lar  \lieL_\pm
$
is an isomorphism of Lie algebras as well as that  $\lieW_\pm$ is a free module of rank $q$ over $L_\pm$. 

\smallskip
\noindent
(b) It follows from the embedding $\lieW \subseteq \lieW_+ \times \lieW_-$ that 
$(t_+, t_-) \lieW \subseteq t_+ \lieW_+ \times t_-\lieW_-$. On the other hand, 
it follows from the inclusions (\ref{E:EmbeddingsOrders}) that
$$
t_+ \lieW_+ \times t_- \lieW_- \subseteq \lieW_+^\perp \times \lieW_-^\perp \subseteq \lieW^\perp = \lieW. 
$$
Abusing the notation, we can view $O$ as the subalgebra of $R_+ \times R_-$ generated by the elements 
$t_+ = (t_+, 0)$ and $t_- = (0, t_-)$.  It follows from the assumption $\lieW \asymp 	 \lieW^\circ$ and the fact that $\lieW^\circ$ is finitely generated over $O$ that 
$\lieW$ is a finitely generated torsion free $O$-module, as asserted. 

\smallskip
\noindent
(c) Both  results  follow from the definition of $\lieW_\pm$ and the previous statement. 
\end{proof}

\begin{lemma}\label{L:OrdersFromMT2}
Let $\lieV_\pm$ be the preimage of $\lieW_\pm$ under the isomorphism $\lieL \to \lieL_\pm$ and (abusing the notation) $L_\pm = \CC[t^{\pm 1}]$. Then the following results are true:
\begin{enumerate}
\item[(a)] $\lieL = \lieV_+ + \lieV_-$;
\item[(b)]  the linear map $\lieV_+ \cap \lieV_- \stackrel{\imath}\lar \bigl(\lieW_+ \times \lieW_-\bigr)/\lieW, f \mapsto \overline{(f, f^\ddagger)}$ is an isomorphism; 
\item[(c)] $\lieV_\pm$ is a free module of rank $q$ over $L_\pm$ and both canonical maps  $R   \otimes_{L_\pm} \lieV_\pm \to \lieL$ are isomorphisms of Lie algebras. 
\end{enumerate}
\end{lemma}
\begin{proof} (a) Take any $f \in \lieL$. It follows from the direct sum decomposition $\lieM = \lieD \dotplus \lieW$ that there exist $g \in \lieL$ and
$(w_+, w_-) \in \lieW$ such that 
$
(f, 0) = (g, g^\ddagger) + (w_+, w_-).
$
Let $v_\pm \in \lieV$ be the elements corresponding to $w_\pm \in \lieW_\pm$ under the isomorphisms $\lieL \to \lieL_\pm$. 
It follows that $f = v_+ - v_- \in \lieV_+ + \lieV_-$, as asserted. 

\smallskip
\noindent
(b) Let $v \in \lieV_+ \cap \lieV_-$ be such that $\overline{(v, v^\ddagger)} = 0$ in $\bigl(\lieW_+ \times \lieW_-\bigr)/\lieW$. It follows that 
$(v, v^\ddagger) \in \lieD \cap \lieW = 0$, hence $v = 0$, what implies injectivity of $\imath$. 

\smallskip
\noindent
Consider an arbitrary element $(w_+, w_-) \in \lieW_+ \times \lieW_-$. Then there exist $w \in \lieL$ and $(w'_+, w'_-) \in \lieW$ such that
$
(w_+, w_-) = (w, w^\ddagger) + (w'_+, w'_-).
$
It follows that $w = w_+ - w'_+ \in \lieW_+$ and $w^\ddagger = w_- - w'_- \in \lieW_-$, thus 
$
\overline{(w_+, w_-)} = \overline{(w, w^\ddagger)} \in \bigl(\lieW_+ \times \lieW_-\bigr)/\lieW.
$
We conclude  that  $\imath$ is surjective, hence  an isomorphism. 

\smallskip
\noindent
(c) This statement is a translation of the corresponding result from Lemma
\ref{L:OrdersFromMT1}.
\end{proof}

\smallskip
\noindent
Let $\widehat{L} = \CC\llbrace t\rrbrace$ and $\widehat\lieL := \widehat{L} \otimes_R \lieL$. We identify elements of $\widehat\lieL$ with formal power series $\sum\limits_{k \gg -\infty} a_k z^k$ (where $a_k \in \lieg_k$ for all $k \in \ZZ$). Obviously, we have an embedding of Lie algebras $\lieL \longhookrightarrow \widehat{\lieL}$. We  extend the standard form ${\lieL} \times {\lieL} \stackrel{{B}} \lar \CC$ to a bilinear  form $\widehat{\lieL} \times \widehat{\lieL} \stackrel{\widehat{B}} \lar \CC$,  defining it by the same formula (\ref{E:FormB}). Next, we put:
$\widehat{\lieM}:= \widehat{\lieL} \times \widehat{\lieL}^\ddagger$ and denote by
$\widehat{\lieM} \times \widehat{\lieM} \stackrel{\widehat{F}}\lar \CC$ the bilinear form given by the same recipe as in (\ref{E:FormOnLiem}). Note that
$\lieL \to \widehat\lieM, f \mapsto (f, f^\ddagger)$ is an embedding of Lie algebras, whose image is an isotropic subspace with respect to $\widehat{F}$. 

\smallskip
\noindent
Let 
$
\widehat\lieB_+ := (\lieg_0^+ \oplus \lieh) \oplus \bigl(\prod\limits_{k = 1}^\infty \lieg_k z_{+}^k \bigr)$  and $
\widehat\lieB_- := (\lieg_0^- \oplus \lieh) \oplus \bigl(\prod\limits_{k = 1}^\infty \lieg_k^\ddagger z_{-}^k \bigr).
$
We put:
$$
\widehat{\lieW}^\circ := \bigl\{(f_+, f_-) \in \widehat{\lieB}_+ \times \widehat{\lieB}_- \, \big| \, 
\pi_+(f_+) + \pi_-(f_-) = 0 \bigr\} \quad \mbox{\rm and} \quad
\lieD := \bigl\{(f, f^\ddagger) \, \big| \, f \in \lieL \bigr\}.
$$
Analogously to (\ref{E:StandardManinTripleLoops}), we have a Manin triple
\begin{equation}\label{E:CompletedStandardManinTripleLoops}
\widehat{\lieM} = \lieD \, \dotplus\,  \widehat{\lieW}^\circ.
\end{equation}
\smallskip
\noindent
Our next goal is to reformulate the theory of twists of the standard Lie bialgebra structure on $\lieL$  in the terms of completed Manin triples. 

\smallskip
\noindent
For any $n \in \NN$, we define the linear map $\widehat\lieL \stackrel{\jmath_n}\lar \lieL$  as the composition 
$$
\widehat\lieL \rightarrowdbl \bigoplus\limits_{k \gg -\infty}^n \lieg_k z^k 
\longhookrightarrow  \lieL, \; \sum\limits_{k \gg -\infty}^\infty  a_k z^k \mapsto \sum\limits_{k \gg -\infty}^n a_k z^k.
$$
Next, for any $f, g \in \widehat\lieL$ there exists $n_0 \in \NN$ such that for all $n \ge n_0$ we have: 
\begin{equation}\label{E:FormJets}
\widehat{B}(f, g) = B\bigl(\jmath_n(f), \jmath_n(g)\bigr).
\end{equation}

\smallskip
\noindent
Let $\lieM = \lieD \dotplus \lieW$ be a Manin triple from Theorem \ref{T:ManinTriplesLoopTwists}. According to Lemma \ref{L:OrdersFromMT1}, $\lieW$ is a finitely generated $O$-module. We can definite the completed Lie algebra $\widehat\lieW$ as follows: 
\begin{equation}\label{E:CompletedLagrSubspace}
\widehat\lieW:= \varprojlim \bigl(\lieW/\mathfrak{m}^k \lieW\bigr) \cong \widehat{O} \otimes_O \lieW \subset (\widehat{L}_+ \times \widehat{L}_-)\otimes_O \lieW \cong \widehat{\lieM}, 
\end{equation}
where $\mathfrak{m} = (t_+, t_-)$, $\widehat{L}_\pm = \CC\llbrace t_\pm\rrbrace$ and $\widehat{O}= \varprojlim O/\mathfrak{m}^k \cong \CC[\![t_+,t_-]\!]/(t_+t_-) \subset \widehat{L}_+ \times \widehat{L}_-$.  It follows from $\lieW \asymp 	 \lieW^\circ$ that $\widehat\lieW = \lieW \,+\, \mathfrak{m}^k \widehat{\lieW}^\circ$ for all sufficiently large $k \in \NN$, which can serve as an alternative definition
of $\widehat{\lieW}$.

\begin{proposition}\label{P:CompletedTwistsManinTriples}
We have the following commutative diagram of bijections:
\begin{equation}\label{E:DiagramOfBijections}
\begin{array}{c}
    \xymatrix{&\Bigl\{\ttau \in \wedge^2\lieL \Big| \, (\lieL, \delta_\ttau) \, \mbox{is a Lie bialgebra} \Bigr\}\ar[dr]\ar[dl]&\\\mathsf{MT}\bigl(\lieM, \lieC; \lieW^\circ)\ar[rr]&&\mathsf{MT}\bigl(\widehat\lieM, \lieC; \widehat\lieW^\circ)}
    \end{array}.
\end{equation}
Here, the left diagonal arrow is given in Theorem \ref{T:ManinTriplesLoopTwists} and the horizontal arrow is given by \(\lieW \mapsto \widehat\lieW\). Moreover, if \(\delta_\ttau\) is a Lie bialgebra cobracket for some \(\ttau \in \wedge^2\lieL\), it is determined by the Manin triple \(\widehat\lieM = \lieD \, \dotplus \, \widehat\lieW_\ttau\).
\end{proposition}
\begin{proof}

We first show that the Manin triple $\widehat{\lieM} = \lieD \dotplus \widehat{\lieW}^\circ$ determines $\delta_\circ$.
By abuse of notation, we write $\jmath_n((f,g^\ddagger)) = (\jmath_n(f),\jmath_n(g)^\ddagger) \in \lieM$ for any $(f,g^\ddagger) \in \widehat\lieM$, where $g^\ddagger(z) = g(z^{-1}) \in \widehat\lieL^\ddagger$ and $f,g \in \widehat\lieL$.
Let $f \in \lieD$ and $g', g'' \in \widehat\lieW^\circ$. Then for all $n \in \NN$ we have: $\jmath_n(g'), \jmath_n(g'') \in \lieW^\circ$ and 
$$
\widehat{F}\bigl(\delta_\circ(f), g' \otimes g''\bigr) = 
F\bigl(\delta_\ttau(f), \jmath_n(g') \otimes \jmath_n(g'')\bigr) = 
F\bigl(f, \bigl[\jmath_n(g'), \jmath_n(g'') \bigr]\bigr),
$$
where the last equality follows from Theorem \ref{T:ManinTriplesLoopTwists}.
Taking $n$ sufficiently large, we continue:
$$
F\bigl(f, \bigl[\jmath_n(g'), \jmath_n(g'') \bigr]\bigr) = F\bigl(f, 
\jmath_n\bigl[g', g''\bigr]\bigr)= \widehat{F}\bigl(f, [g',  g'']\bigr),
$$
what implies that $\widehat{F}\bigl(\delta_\circ(f), g' \otimes g''\bigr) = \widehat{F}\bigl(f, [g',  g'']\bigr)$, as asserted.

\smallskip
\noindent
By  Proposition \ref{P:basicsonloops},  $\lieL^{\otimes 3}$ does not contain any non-zero ad-invariant elements, hence according to Proposition \ref{P:TwistLieBialgStructures}, 
$(\lieL, \delta_\ttau)$ is a Lie bialgebra if and only if 
$\ttau$ satisfies the twist equation (\ref{E:TwistEquation}) and we obtain the right diagonal bijection using Theorem \ref{T:ManinTriplesTwists}. It remains to show that the diagram is commutative.

\smallskip
\noindent
Next,  for any tensor $\ttau = \sum_{i = 1}^n a_i \otimes b_i$ there exists  $k \in \NN$ such that for all $w \in \widehat\lieW^\circ$ we have 
$$
 \widehat f_\ttau(w) := \sum\limits_{i = 1}^n \widehat F(w,a_i)b_i = \sum_{i = 1}^n  F(\jmath_k (w),a_i)b_i.
$$
Since \(\widehat F\) extends \(F\), we obtain: $\lieW_\ttau = \bigl\{w + \widehat f_\ttau(w)\mid w \in \lieW^\circ\bigr\}$. As a consequence, 
$$\left\{w + \widehat f_\ttau(w)\Big| w \in \widehat\lieW^\circ\right\} = \lieW_\ttau + \mathfrak{m}^{k+1}\widehat\lieW^\circ = \widehat\lieW_\ttau,$$
showing that the diagram (\ref{E:DiagramOfBijections}) is indeed commutative.
\end{proof}

\section{On algebraic geometry of the classical Yang--Baxter equation}\label{S:ReviewGeomCYBE}

\smallskip
\noindent
Let $\lieg$ be  a finite dimensional  simple Lie algebra over $\CC$ of dimension $q$, $\lieg \times \lieg 
\stackrel{\kappa}\lar \CC$  be its  Killing form and $\gamma  \in \lieg \otimes \lieg$ the Casimir element.

\subsection{Classical Yang--Baxter equation and associated Lie subalgebras of $\lieg\llbrace z\rrbrace$}\label{SS:SurveyCYBE}

\smallskip
\noindent
Recall that the germ of a tensor-valued meromorphic function $\bigl(\CC^2, 0\bigr) \stackrel{r}\lar \lieg \otimes \lieg$   is a skew-symmetric solution of  the classical Yang--Baxter equation  (CYBE) if 
\begin{equation}\label{E:CYBE}
\left\{
\begin{array}{l}
\bigl[r^{12}(x, y), r^{13}(x, z)\bigr] + \bigl[r^{13}(x ,z), r^{23}(y,z)\bigr] +
\bigl[r^{12}(x, y),
r^{23}(y, z)\bigr] = 0
\\
r^{12}(x, y) = 
-r^{21}(y, x).
\end{array} 
\right.
\end{equation}
\noindent
The Killing form $\lieg \times \lieg \stackrel{\kappa}\lar \CC$ induces an isomorphism of vector spaces
\begin{equation}\label{E:KillingIsom}
\mathfrak{g} \otimes \mathfrak{g} \stackrel{\widetilde\kappa}\lar \End_{\CC}(\mathfrak{g}), \quad a \otimes b \mapsto \bigl(c \mapsto
\kappa(a,c) \cdot b\bigr).
\end{equation}
A solution $r$ of (\ref{E:CYBE}) is called \emph{non-degenerate}, if for  a generic point $(x_1^\circ, x_2^\circ)$ in the domain of definition of $r$, the linear
map $\widetilde\kappa\bigl({r}(x_1^\circ, x_2^\circ)\bigr) \in \End_{\CC}(\lieg)$ is an isomorphism.

\smallskip
\noindent
One can perform the following transformations with solutions   of (\ref{E:CYBE}).
\begin{itemize}
\item \emph{Gauge transformations}. For any holomorphic  germ
$(\mathbb{C},0) \stackrel{\phi}\lar {\mathsf{Aut}}_{\CC}(\mathfrak{g})$,  the  function
\begin{equation}\label{E:equiv1}
 \overline{r}(x, y) :=
\bigl(\phi(x) \otimes \phi(y)\bigr) r(x,y).
\end{equation}
is again a  solution of  (\ref{E:CYBE}).
\item \emph{Change of variables}. Let $(\CC, 0) \stackrel{\eta}\lar (\CC, 0)$ be a non-constant  map of germs. Then 
\begin{equation}\label{E:equiv2}
 \overline{\overline{r}}(x, y) :=
r\bigl(\eta(x), \eta(y)\bigr).
\end{equation}
is again a  solution of  (\ref{E:CYBE}).
\end{itemize}
It is clear that both  transformations $(\ref{E:equiv1})$ and $(\ref{E:equiv2})$ map  non-degenerate  solutions of  (\ref{E:CYBE}) into non-degenerate ones.

\smallskip
\noindent
Belavin and Drinfeld proved in \cite{BelavinDrinfeld2} that any non-degenerate solution of  (\ref{E:CYBE}) can be transformed by above transformations  to a solution of  the form
\begin{equation}\label{E:formcanonique}
r(x, y) = \frac{1}{x-y} \gamma + h(x, y), \quad h(x, y) = - h^{21}(y, x)
\end{equation}
where $\bigl(\CC^2, 0\bigr) \stackrel{h}\lar \lieg \otimes \lieg$ is the germ of a  holomorphic function.
Moreover, they showed  that one can  always find  a gauge transformation  $\phi$ and a change of variables $\eta$ such that 
 $\bigl(\phi(x) \otimes \phi(y)\bigr) r\bigl(\eta(x), \eta(y)\bigr) =  \varrho(x-y)$
 for some meromorphic $(\CC, 0) \stackrel{\varrho}\lar \lieg \otimes \lieg$. 
 In other words, (\ref{E:CYBE}) reduces to the equation
\begin{equation}\label{E:CYBE1}
\bigl[\varrho^{12}(x), \varrho^{13}(x+y)\bigr]  +
\bigl[\varrho^{12}(x), \varrho^{23}(y)\bigr] + \bigl[\varrho^{13}(x+y), \varrho^{23}(y)\bigr] = 0
\end{equation}
(the so-called CYBE with \emph{one} spectral parameter).
Belavin and Drinfeld proved in \cite{BelavinDrinfeld} that  any  non--degenerate
solution of (\ref{E:CYBE1}) is automatically skew-symmetric, has a simple pole at $0$ with  residue equal
to a multiple of the Casimir element $\gamma \in \lieg \otimes \lieg$. Moreover, $\varrho$  can be meromorphically extended on the entire plane $\CC$ and its poles form 
an additive subgroup  $\Lambda \subseteq \CC$ such that $\rk(\Lambda) \le 2$; see 
\cite[Theorem 1.1]{BelavinDrinfeld}.
\begin{itemize}
\item If $\rk(\Lambda) = 2$ than the corresponding solution $\varrho$ is elliptic. Elliptic solutions exist only for $\lieg \cong  \mathfrak{sl}_n(\CC)$. A full list of them is given in \cite[Section 5]{BelavinDrinfeld}.
\item If $\rk(\Lambda) = 1$ than the corresponding solution $\varrho$ is trigonometric. A full classification of these solutions  is given in \cite[Section 6]{BelavinDrinfeld},  see also \cite[Chapter 7]{BelavinDrinfeldBook}.
\item If $\Lambda = 0$ then $\varrho$ is a rational solution, i.e. $\varrho(x) = \dfrac{\gamma}{x} + \xi(x)$, where $\xi \in (\lieg \otimes \lieg)[x]$. The problem of classification of all rational solutions for $\lieg = \mathfrak{sl}_n(\CC)$ contains a representation-wild problem of classification of pairs of matrices $a, b \in \lieg$ such that 
$[a, b] = 0$, see Remark \ref{R:Wildness} below.  Nonetheless, the structure theory of rational solutions was developed by Stolin in \cite{Stolin}.
\end{itemize}

\smallskip
\noindent
Among various  constructions which attach to a  solution of (\ref{E:CYBE}) a Lie bialgebra there is the following \emph{universal one}, which  dates  back to the works \cite{GelfandCherednik, ReymanST}.

\smallskip
\noindent
Consider the Lie algebra of formal Laurent series $\lieR:= \lieg\llbrace z \rrbrace$. It is equipped with a symmetric non-degenerate invariant form
\begin{equation}\label{E:FormSeries}
\lieR\times \lieR  \stackrel{F}\lar \CC, (a z^k, b z^l) \mapsto \delta_{k+l+1, 0} \, \kappa(a, b). 
\end{equation}
Let $r$ be a solution of (\ref{E:CYBE})  having the form (\ref{E:formcanonique}).
We write its formal power series expansion 
\begin{equation}\label{E:ExpansionRMatr}
\tilde{r}(x; y) = \sum\limits_{k = 0}^\infty 
r_{k}(x) y^k \in \bigl(\lieR \otimes \lieg\bigr)\llbracket y\rrbracket, \; \mbox{\rm where} \; \left. r_k(x) = \frac{1}{k!} \frac{\partial^k r}{\partial y^k}\right|_{y = 0}.
\end{equation}
For any $k \in \NN_0$  let $\lieW_k := \bigl\langle (1 \otimes \lambda)r_k(x)  \, \big|\, \lambda \in \lieg^\ast\bigr\rangle_{\CC} \subseteq \lieR$. Then we put: 
\begin{equation}\label{E:FormalLieAlgW}
\lieW := \sum\limits_{k \in \NN_0} \lieW_k 
\subseteq \lieR.
\end{equation}
More concretely, let  $(g_1, \dots, g_q)$ be an orthonormal basis of $\lieg$ with respect of $\kappa$. Then 
$\gamma = g_1 \otimes g_1 + \dots + g_q \otimes g_q$  and the power series expansion (\ref{E:ExpansionRMatr}) can be written as 
\begin{equation}\label{E:ExpansionConcrete}
\tilde{r}(x; y) = \sum\limits_{k = 0}^\infty \sum\limits_{i = 1}^q 
\left(w_{(k, i)} \otimes g_i\right) y^k \in \bigl(\lieR \otimes \lieg\bigr)\llbracket y\rrbracket, \end{equation}
where  $w_{(k, i)} = g_i x^{-k - 1} + v_{k, i}$ for some $v_{k, i} \in \lieg\llbracket x\rrbracket$. We have: 
$$
\lieW := \bigl\langle w_{(k, i)} \, \big|\, 1 \le i \le q, k \in \NN_0 \bigr\rangle_{\CC} \subset \lieR.
$$
Let $\Upsilon = \left\{(k, i) \, \big| \, k \in \NN_0, 1 \le i \le q \right\}$ and $g_{(k, i)}:= g_i x^k$
for any $(k, i) \in \Upsilon$. Then we have: 
\begin{equation}\label{E:OrthBasis}
F\bigl(w_{(k', i')}, g_{(k'', i'')}\bigr) = \delta_{k', k''} \delta_{i', i''} \; \;  \mbox{for all}\; \;  (k', i'), (k'', i'') \in \Upsilon.
\end{equation}
Let $\bigl(\CC^2, 0\bigr) \stackrel{r}\lar \lieg \otimes \lieg$  be of the form (\ref{E:formcanonique}). Then (\ref{E:CYBE})  can be rewritten as the system of the following constraints  on the 
coefficients $r_k(x) \in \lieR$ of the series $\tilde{r}(x; y)$:
\begin{equation}\label{E:CYBEformali}
\left[r_k^{13}(x_1) + r_k^{23}(x_2), r^{12}(x_1, x_2)\right] =
\sum\limits_{\substack{k', k'' \ge 0 \\
k' + k'' = k}} \left[r_{k'}^{13}(x_1),  r_{k''}^{23}(x_2)\right] \quad \mbox{for all} \; k \in \NN_0.
\end{equation}
In more concrete terms,  (\ref{E:CYBEformali}) can be rewritten as the following equality:
\begin{multline}\label{E:CYBEformaliVariation}
\sum\limits_{i = 1}^q \left[w_{(k, i)}(x_1) \otimes 1 + 1 \otimes w_{(k, i)}(x_2), r(x_1, x_2)\right] \otimes g_i = \\ 
\sum\limits_{\substack{(k', i') \in \Upsilon \\
(k'', i'') \in \Upsilon \\
k' + k'' = k}}
 w_{(k', i')}(x_1) \otimes w_{(k'', i'')}(x_2) \otimes
\bigl[g_{i'}, g_{i''}\bigr]
\end{multline}
in the vector space $\bigl((\lieg \otimes \lieg)\llbrace x_1, x_2\rrbrace\bigr) \otimes \lieg$, where the right-hand side of (\ref{E:CYBEformaliVariation})  is embedded into $\bigl((\lieg \otimes \lieg)\llbrace x_1, x_2\rrbrace\bigr) \otimes \lieg$ via the canonical linear map 
$\lieg\llbrace x_1\rrbrace \otimes \lieg\llbrace x_2\rrbrace \longhookrightarrow
(\lieg \otimes \lieg)\llbrace x_1, x_2\rrbrace$
(it follows from (\ref{E:formcanonique}) that the left-hand side belongs to 
$\lieg\llbrace x_1\rrbrace \otimes \lieg\llbrace x_2\rrbrace \otimes \lieg$ as well). Therefore, we have a linear map
\begin{equation}\label{E:cobracketpowerseries}
\lieW \stackrel{\delta}\lar \lieW \otimes \lieW, \; w(x) \mapsto \bigl[w(x_1) \otimes 1 + 1 \otimes w(x_2), r(x_1, x_2)\bigr].
\end{equation}
The system of constraints (\ref{E:CYBEformali}) can be stated for any expression $r(x, y)$ of the form (\ref{E:formcanonique}) with $h(x,y) \in 
(\lieg \otimes \lieg)\llbracket x, y\rrbracket$ (without requiring the convergence of $h(x, y)$ and even passing from $\CC$ to an arbitrary  field $\kk$), so  one may speak on \emph{formal solutions} of CYBE.

\smallskip
\noindent
We have the following result, see e.g. \cite[Subsection 6.3.3]{EtingofSchiffmann}) for a proof.
\begin{theorem}\label{E:FormalCYBEManinTriples}
Let $r= \dfrac{1}{x-y} \gamma + h(x, y)$ be any formal  solution of CYBE. Then the corresponding 
vector subspace $\lieW \subseteq \lieR$, given by (\ref{E:FormalLieAlgW}),  is a Lagrangian Lie subalgebra  with respect to the  bilinear form (\ref{E:FormSeries}). Moreover, we have a direct sum decomposition $
\lieR  = \lieg\llbracket z \rrbracket \dotplus \lieW
$ and the map $\lieW \stackrel{\delta}\lar \lieW \otimes \lieW$, given 
by (\ref{E:cobracketpowerseries}), is a Lie bialgebra cobracket.

\smallskip
\noindent
Conversely, let $\lieR = \lieg\llbracket z \rrbracket \dotplus \lieW$ be a Manin triple. Then the linear map $\lieg\llbracket x\rrbracket \stackrel{\widetilde{F}}\lar \lieW^\ast$ is an isomorphism and 
there exists a uniquely determined family $(w_{(k, i)})_{(k, i)(x) \in \Upsilon}$ of elements of $\lieW$ such that $w_{(k, i)} = g_i x^{-k - 1} + v_{(k, i)}$ for some $v_{(k, i)} \in \lieg\llbracket x\rrbracket$. This family forms a basis
of $\lieW$, which  is dual to the topological basis  $(g_{(k, i)})_{(k, i) \in \Upsilon}$ of $\lieg\llbracket x\rrbracket$ and the formal power series 
 (\ref{E:ExpansionConcrete})
 is a formal  solution of CYBE.
\end{theorem}

\smallskip
\noindent
In the notation of Theorem \ref{E:FormalCYBEManinTriples}, we have the following result.
\begin{proposition}\label{P:BielgebraPowerSeriesDetermined}
The  Lie bialgebra cobracket $\lieW \stackrel{\delta}\lar \lieW \otimes \lieW$ is determined by the corresponding Manin triple $\lieR = \lieg\llbracket x\rrbracket \dotplus \lieW$. 
\end{proposition}

\begin{proof} We have to show the following identity  for any $w \in \lieW$ and $f_1, f_2 \in \lieg\llbracket x\rrbracket$:
\begin{equation}\label{E:CobracketDeterminacy}
F\bigl(\bigl[w(x_1) \otimes 1 + 1 \otimes w(x_2), r(x_1, x_2)\bigr], f_1(x_1) \otimes f_2(x_2)\bigr) = F\bigl(w(x), \bigl[f_1(x), f_2(x)\bigr]\bigr). 
\end{equation}
Note that for any $w \in \lieW$ there exists $n \in \NN$ such that
$$
F\bigl(\delta(w), f_1 \otimes f_2\bigr) = 0  = 
F\bigl(w, \bigl[f_1, f_2\bigr] \bigr),
$$
provided $f_1 \in x^n \lieg\llbracket x\rrbracket$ or $f_2 \in x^n \lieg\llbracket x\rrbracket$. Therefore, it is sufficient to prove that
$$
F\bigl(\delta(w_{(l, j)}), g_{(k', i')} \otimes g_{(k'', i'')} \bigr) = 
F\bigl(w_{(l, j)}, \bigl[g_{(k', i')}, g_{(k'', i'')}\bigr] \bigr)
$$
for all $(l, j), (k', i'), (k'', i'') \in \Upsilon$.

\smallskip
\noindent
First note that we have a finite sum:
$
\bigl[g_{(k', i')},  g_{(k'', i'')}\bigr] = \sum\limits_{(k, i) \in \Upsilon}
\lambda^{(k, i)}_{(k', i'), (k'', i'')} g_{(k, i)},
$
where $\lambda^{(k, i)}_{(k', i'), (k'', i'')} \in \CC$. It is clear that
$\lambda^{(k, i)}_{(k', i'), (k'', i'')} \ne 0$ only if $k = k' + k''$. 
In particular, for any $(k, i) \in \Upsilon$ there exist only finitely many
$(k', i'), (k'', i'') \in \Upsilon$ such that $\lambda^{(k, i)}_{(k', i'), (k'', i'')} \ne 0$.

\smallskip
\noindent
Next, we can rewrite the classical Yang--Baxter equation (\ref{E:CYBEformaliVariation}) as
$$
\sum\limits_{(k, i) \in \Upsilon} \delta\bigl(w_{(k, i)}\bigr) \otimes g_{(k, i)} = \sum\limits_{\substack{(k', i') \in \Upsilon \\
(k'', i'') \in \Upsilon}}
 w_{(k', i')}  \otimes w_{(k'', i'')} \otimes
\bigl[g_{(k', i')}, g_{(k'', i'')}\bigr],
$$
implying  that
$
\delta\bigl(w_{(k, i)}\bigr) = \sum\limits_{\substack{(k', i') \in \Upsilon \\
(k'', i'') \in \Upsilon}} \lambda^{(k, i)}_{(k', i'), (k'', i'')} w_{(k', i')} \otimes w_{(k'', i'')}$. Applying (\ref{E:OrthBasis}) we get
$
F\Bigl(\delta\bigl(w_{(l, j)}\bigr), g_{(k', i')} \otimes g_{(k'', i'')}\Bigr) = \lambda^{(l, j)}_{(k', i'), (k'', i'')} = F\bigl(w_{(l, j)}, 
\bigl[g_{(k', i')}, g_{(k'', i'')}\bigr]\Bigr),
$
as asserted. 
\end{proof}

\subsection{Geometric CYBE datum}\label{SS:GeomCYBEDatum}
Now we make  a quick  review of the algebro-geometric theory of the classical Yang--Baxter equation (\ref{E:CYBE}),
following the work  \cite{BurbanGalinat}.

\smallskip
\noindent
 A \emph{Weierstra\ss{} curve} is an irreducible projective curve over $\CC$ of arithmetic genus one. 
For $g_2, g_3 \in \CC$, let
$
E_{(g_2, g_3)} = \overline{V\bigl(u^2 - 4v^3 + g_2 v + g_3\bigr)} \subset \PP^2.
$
It is well-known that any Weierstra\ss{} curve $E$ is isomorphic  to $E_{(g_2, g_3)}$ for some $g_2, g_3 \in \CC$.  Moreover, $E_{(g_2, g_3)}$ is smooth  if and only if $g_2^3  \ne   27 g_3^2$.
If $g_2^3  =   27 g_3^2$ then $E_{(g_2, g_3)}$ has a unique singular point $s$, which is
 a nodal singularity  if $(g_2, g_3) \ne (0, 0)$ and a cuspidal singularity  if $(g_2, g_3) =  (0, 0)$.
We have: $\Gamma(E, \Omega) \cong \CC$, where $\Omega$ is the sheaf of regular differential one-forms on $E$, taken in the Rosenlicht sense
if  $E$ is singular; see e.g.~\cite[Section II.6]{BarthHulekPetersVen}.

\smallskip
\noindent
Assume that $\kA$ is a coherent sheaf of Lie algebras  on $E$ such that:
\begin{enumerate}
\item $\kA$ is acyclic, i.e.~$H^0(E, \kA) = 0 = H^1(E, \kA)$;
\item $\kA$ is weakly $\lieg$--locally free on the regular part $U$ of $E$, i.e.~$\kA\big|_x \cong \lieg$ for all $x \in U$.
\end{enumerate}
From the first assumption it follows that  the sheaf $\kA$ is torsion free.
The second assumption on $\kA$  implies that the  canonical isomorphism of $\kO_{U}$-modules $\kA\big|_{U} \otimes  \kA\big|_{U} \rightarrow \mathit{End}_{U}\bigl(\kA\bigr)$, induced by the Killing forms of the Lie algebras of local sections of $\kA$, is an isomorphism. As a consequence, 
the space  $\lieA_{\mathbbm{K}}$ of global sections of the rational envelope of $\kA$ is  a simple Lie algebra over the field  $\mathbbm{K}$ of meromorphic functions on $E$.

\smallskip
\noindent
Choosing  a global regular one-form $0 \ne \omega \in \Gamma(E, \Omega)$, we get  the so-called residue short exact sequence:
\begin{equation}\label{E:residue1}
0 \lar \kO_{E \times U} \lar \kO_{E \times U}(\Sigma) \xrightarrow{\res_\Sigma^\omega} \kO_\Sigma \lar 0,
\end{equation}
where $\Sigma \subset E \times U$ denotes the diagonal, see  \cite[Section 3.1]{BurbanGalinat}. Tensoring  (\ref{E:residue1})  with $\kA \boxtimes \kA\big|_{U}$ and then applying the functor $\Gamma(E \times U, \,-\,)$, we obtain
a $\CC$--linear map $$\End_{U}(\kA) \stackrel{T_\omega}\lar \Gamma\bigl(U \times U \setminus \Sigma, \kA \boxtimes \kA\bigr),$$ making the following diagram
\begin{equation}\label{E:DefineGeomRMat}
\begin{array}{c}
\xymatrix{
\Gamma\bigl(U, \kA \otimes \kA\bigr) \ar[d]_-\cong & & \Gamma\bigl(E \times U, \kA \boxtimes \kA|_U(\Sigma)\bigr) \ar[ll]_-{\cong} \ar@{^{(}->}[d] \\
 \End_{U}(\kA) \ar[rr]^-{T_\omega} & & \Gamma\bigl(U \times U \setminus \Sigma, \kA \boxtimes \kA\bigr)
}
\end{array}
\end{equation}
commutative. In this way, we get a distinguished section
\begin{equation}
\rho:= T_\omega(\mathbbm{1}) \in \Gamma\bigl(U \times U \setminus \Sigma, 
\kA \boxtimes \kA\bigr),
\end{equation}
called a \emph{geometric $r$--matrix} attached to a pair $(E, \kA)$ as above.

\smallskip
\noindent
If the curve $E$ is singular, we additionally require  that
\begin{enumerate}
  \setcounter{enumi}{2}
  \item
$\kA$ is isotropic at $s$, i.e. the germ $\lieA_s$ of the sheaf $\kA$ at the singular point $s$ is an isotropic Lie subalgebra of $\lieA_\mathbbm{K}$ with respect to the pairing
  $$
F^\omega_{s}: \; \lieA_\mathbbm{K} \times \lieA_\mathbbm{K} \stackrel{K}\lar \mathbbm{K} \stackrel{\res_s^\omega}\lar \CC,
$$
where $K$ is the Killing form of $\lieA_{\KK}$ and $\res_s^\omega(f) = \res_s(f\omega)$ for $f \in \mathbbm{K}$ (taken in the Rosenlicht sense).
\end{enumerate}

\smallskip
\noindent
A pair $(E, \kA)$ satisfying the properties (1)--(3) above will be  called \emph{geometric CYBE datum}.

\smallskip
\noindent
We have  the following result; see  \cite[Theorem 4.3]{BurbanGalinat}.
\begin{theorem}\label{T:GeoemtryCYBE}
Let $(E, \kA)$ be a geometric CYBE datum. 
Then we have:

\smallskip
\noindent
1.~The geometric $r$-matrix $\rho$ satisfies the following sheaf-theoretic version of the classical Yang--Baxter equation: 
\begin{equation}\label{E:SheafyCYBE1}
\bigl[\rho^{12}, \rho^{13}\bigr] + \bigl[\rho^{12}, \rho^{23}\bigr] + \bigl[\rho^{13}, \rho^{23}\bigr] = 0,
\end{equation}
where both sides of the above  equality are viewed as meromorphic sections of   $\kA \boxtimes \kA \boxtimes \kA$ over the triple product $U \times U \times U$.

\smallskip
\noindent
2.~Moreover,  $\rho$ is skew-symmetric and non-degenerate i.e.
\begin{equation}\label{E:SheafyCYBE2}
\rho(x_1, x_2)^{12} = - \rho(x_2, x_1)^{21} \in \bigl(\kA \boxtimes \kA\bigr)\big|_{(x_1, x_2)} \cong \kA\big|_{x_1} \otimes \kA\big|_{x_2} \quad \mbox{\rm for any} \; x_1 \ne x_2 \in U
\end{equation}
and there exists an open subset $U' \subseteq U$ such that for any $x_1 \ne x_2 \in U'$, the tensor $\rho(x_1, x_2) \in \kA\big|_{x_1} \otimes \kA\big|_{x_2}$ is non-degenerate.
\end{theorem}

\smallskip
\noindent
In what follows,  we write \(\kO = \kO_E\).
Let $V \subseteq U$ be an open affine subset, $R_V = \Gamma(V, \kO)$  and 
$\lieA_V := \Gamma(V, \kA)$. Assume that $V$ is sufficiently small so that 
$\lieA_V$ is free as $R_V$-module. Since $\lieA$ is weakly $\lieg$-locally free, the Killing form $\lieA_V \times \lieA_V \rightarrow R_V$ is non-degenerate. Let
$(c_1, \dots, c_q)$ be a basis of $\lieA_V$ over $R_V$ and $(c_1^{\ast}, \dots, c_q^\ast)$ be the dual basis. Then $\chi:= c_1^\ast \otimes c_1 + \dots + c_q^\ast \otimes c_q \in \lieA_V \otimes_{R_V} \lieA_V$ is the canonical Casimir element. Let $\widetilde\chi:= c_1^\ast \otimes c_1 + \dots + c_q^\ast \otimes c_q \in \lieA_V \otimes_{\CC} \lieA_V$. Then $\widetilde\chi$ is a (non-canonical) lift 
of $\chi$ under the canonical  map $\lieA_V \otimes_{\CC} \lieA_V \rightarrowdbl  \lieA_V \otimes_{R_V} \lieA_V$. Choosing coordinates $(u, v)$ on 
$V \times V$, we may write:
\begin{equation}\label{E:GeomRMatrixLocal}
\rho\Big|_{(V \times V)\setminus \Sigma} = \dfrac{f(v)}{u-v} \widetilde{\chi} + 
h(u, v)
\end{equation}
for some $h(u, v) \in \lieA_V \otimes_{\CC} \lieA_V$, where $\omega\Big|_{V} = \dfrac{dv}{f(v)}$ for some invertible element  $f \in R_V$.

\smallskip
\noindent
There are two consistent ways to proceed from the abstract geometric $r$-matrix $\rho$  attached to $(E, \kA)$ to a concrete solution of (\ref{E:CYBE}), respectively (\ref{E:CYBEformali}).

\smallskip
\noindent
1. Let us view $E$ as a complex-analytic variety and $\kA$ as a sheaf of Lie algebras in the euclidean topology. As in  \cite[Lemma 2.1]{Kiranagi} one can show that for any $p \in U$ there exists   an  open neighbourhood  $p \in V \subset U$ together with  a   $\Gamma(V, 
\kO^{\mathrm{an}})$--linear isomorphism of Lie algebras
$\Gamma(V, \kA) \stackrel{\xi}\lar \lieg \otimes_{\CC} \Gamma\bigl(V, \kO^{\mathrm{an}}\bigr).$
Then the trivialized section  $\rho^\xi$  can be viewed as   a meromorphic tensor-valued function $V \times V \stackrel{\rho^\xi}\lar \lieg \otimes \lieg$.
It follows from (\ref{E:SheafyCYBE1}) and (\ref{E:SheafyCYBE2}) that  after a choice of a local coordinate on $V$, we get 
   a non-degenerate solution  of (\ref{E:CYBE}). Another choice of a trivialization $\xi$ and a local coordinate on $V$ leads to an equivalent solution (in the sense of (\ref{E:equiv1}) and (\ref{E:equiv2})).  \hfill $\lozenge$

\smallskip
\noindent
2. 
Let $p \in E$ be an arbitrary point,  $\widehat{O}_p$  (respectively $\widehat{\lieA}_p$) be the completion of the stalk of the structure sheaf $\kO$ (respectively, of  $\kA$) at $p$, $\widehat{Q}_{p}$ be the total ring of fractions  of $\widehat{O}_p$ , $E_p := E \setminus \{p\}$, $U_p := U \setminus \{p\}$, $R_p = \Gamma(E_p, \kO)$, $R_p^\circ = \Gamma(U_p, \kO)$, $\lieA_{(p)} := \Gamma(E_p, \kA)$, $\lieA_{(p)}^\circ := \Gamma(U_p, \kA)$ and $\widetilde{\lieA}_p := 
\widehat{Q}_{p} \otimes_{\widehat{O}_p} \widehat{\lieA}_p \cong 
\widehat{Q}_p \otimes_{R_p} \lieA_{(p)}.
$

\noindent
From now on 
suppose that $p \in U$. 
Then we  have the  bilinear form 
$\widetilde{\lieA}_p \times \widetilde{\lieA}_p \stackrel{\widetilde{F}_p^\omega}\lar \CC$ given as the composition 
\begin{equation}\label{E:FormAtp}
\widetilde{\lieA}_p \times \widetilde{\lieA}_p \stackrel{\widetilde{K}_p}\lar \widetilde{Q}_{p} \stackrel{\res_p^\omega}\lar \CC,
\end{equation}
 where $\widetilde{K}_p$ denotes the Killing form of $\widetilde{\lieA}_p$. 
Since the differential form $\omega$ is non-vanishing at $p$, there exists 
a unique isomorphism $\widehat{O}_p \stackrel{\vartheta}\lar \CC\llbracket y \rrbracket$ identifying $\widehat{\omega}_p$ with the differential form $dy$. Moreover, the assumption that $\kA$ is $\lieg$-weakly locally free implies that there exists a 
$\widehat{O}_p$--$\CC\llbracket y \rrbracket$--equivariant isomorphism of Lie algebras $\widehat{\lieA}_p \stackrel{\zeta}\lar \lieg\llbracket y\rrbracket$;
see \cite{GerstenhaberSchack}.
This isomorphism induces a $\widehat{Q}_p$--$\CC\llbrace y \rrbrace$--equivariant isomorphism of Lie algebras $\widetilde{\lieA}_p \stackrel{\tilde\zeta}\lar \lieg\llbrace y\rrbrace$. In this way, we identify  the bilinear form $\widetilde{F}_p^\omega$  with the bilinear form $F$ given by 
(\ref{E:FormSeries}).

\smallskip
\noindent
The following sequence of vector spaces and linear maps 
\begin{equation}\label{E:MayerVietoris}
0 \lar H^0(E, \kA) \lar \lieA_{(p)} \, \oplus \, \widehat{\lieA}_p \lar \widetilde{\lieA}_p  \lar H^1(E, \kA) \lar 0
\end{equation}
is exact, see e.g. \cite[Proposition 3]{Parshin} (it is a version of the Mayer--Vietoris exact sequence).
Since $H^0(E, \kA) = 0 = H^1(E, \kA)$, it follows  that $\lieA_{(p)} \cap \, \widehat{\lieA}_p = 0$ and $\lieA_{(p)} + \, \widehat{\lieA}_p = \widetilde{\lieA}_p$, where we identify the Lie algebras $\lieA_{(p)}$ and $\widehat{\lieA}_p$ with their images in $\widetilde{\lieA}_p$ under the corresponding canonical embeddings. It follows from the isotropy assumption (3) on the sheaf $\kA$ that  $\lieA_{(p)}$ and $\widehat{\lieA}_p$ are isotropic Lie subalgebras  of 
 $\widetilde{\lieA}_p$
with respect to the bilinear form $\widetilde{F}_p^\omega$, i.e.
$\widetilde{\lieA}_p = \widehat{\lieA}_p \dotplus \lieA_{(p)}$ is
 a Manin triple. Identifying $\widetilde{\lieA}_p$ with $\lieR$,  $\widehat{\lieA}_p$ with $\lieg\llbracket y\rrbracket$ and  $\lieA_{(p)}$ with 
 its image $\lieW$ in $\lieR$,  we  end up with a Manin triple
 $\lieR = \lieg\llbracket y\rrbracket \dotplus \lieW$ as in Theorem
 \ref{E:FormalCYBEManinTriples}.

\smallskip
\noindent
We have a family of compatible linear maps $\Gamma\bigl((E \times U) \setminus \Sigma, \kA \boxtimes \kA\bigr) \stackrel{\upsilon_n}\lar \lieW \otimes \lieg[y]/(y^n)$ given as the composition
$$
\Gamma\bigl((E \times U) \setminus \Sigma, \kA \boxtimes \kA\bigr) \stackrel{\nu_n}\lar \lieA_{(p)} \otimes \bigl(\widehat{\lieA}_{p}/\mathfrak{m}^n_p \widehat{\lieA}_{p}\bigr) \stackrel{\zeta_n}\lar  \lieW \otimes \lieg[y]/(y^n).
$$
Here, $\zeta_n$ is induced by the trivializations
$\zeta$ and $\tilde\zeta$ and $\nu_n := (i\times \iota_n)^*$, where the morphism $ \textnormal{Spec}(\widehat{O}_p/\mathfrak{m}^n) \stackrel{\iota_n}\lar E$ maps the unique closed point of $\textnormal{Spec}(\widehat{O}_p/\mathfrak{m}^n)$ to $p$ and $E_p \stackrel{i}\lar E$ is the canonical inclusion. Taking the projective limit of  $(\upsilon_n)_{n \in \NN}$, we get a linear map $$\Gamma\bigl((E \times U) \setminus \Sigma, \kA \boxtimes \kA\bigr) \stackrel{\upsilon}\lar (\lieW \otimes \lieg)\llbracket y\rrbracket.$$ 
In \cite[Theorem 6.4]{BurbanGalinat} it was shown that
$$
\tilde{r}^{\zeta}(x; y) := \upsilon(\rho) = \sum\limits_{k = 0}^\infty  \left(\sum\limits_{i = 1}^q w_{(k, i)}(x) \otimes g_i\right)y^k = \frac{\gamma}{x-y} + \sum\limits_{k=0}^\infty \left(\sum\limits_{i = 1}^q v_{(k, i)}(x) \otimes  g_i\right) y^k,
$$
where $w_{(k, i)} = g_i x^{-k-1} + v_{(k, i)} \in \lieW$ are such that $v_{(k, i)} \in \lieg\llbracket x\rrbracket$ for all $(k, i) \in \Upsilon$.
It follows from Theorem \ref{E:FormalCYBEManinTriples} that $\tilde{r}^{\zeta}(x; y)$
is a formal skew-symmetric solution of CYBE (\ref{E:CYBEformali}). \hfill $\lozenge$

\begin{remark}
According to  Theorem \ref{E:FormalCYBEManinTriples},  $\lieA_{(p)}$ is a Lie bialgebra. Now we give a sheaf-theoretic description of the corresponding Lie bialgebra cobracket
$\lieA_{(p)} \stackrel{\delta_p}\lar \lieA_{(p)} \otimes \lieA_{(p)}$. 

\smallskip
\noindent
Let $\varrho \in \Gamma\bigl(E \times U, \kA \boxtimes \kA(\Sigma)\bigr)$ the preimage of $\rho$ under the canonical restriction map (it follows from  (\ref{E:DefineGeomRMat}) that such preimage exists and is unique). Then we have a linear map
\begin{equation}\label{E:left}
\lieA_{(p)} \stackrel{\delta_p^{(l)}}\lar \Gamma\bigl(E_p \times U_p, (\kA \boxtimes \kA)(\Sigma)\bigr), \; 
f \mapsto \left[f \otimes 1 + 1 \otimes f, \varrho\big|_{E_p \times U_p} \right].
\end{equation}
Analogously, we have a distinguished section $\varrho^{\sharp} \in \Gamma\bigl(U \times E, \kA \boxtimes \kA(\Sigma)\bigr)$ such that
$$
\varrho^\sharp\Big|_{(x, y)} = \left(\varrho\Big|_{(y, x)} \right)^{21}
 \in \kA\Big|_{x} \otimes \kA\Big|_{y} \;  \mbox{for all} \; x \ne y \in U.
$$
Consider the  linear map
\begin{equation}\label{E:right}
\lieA_{(p)} \stackrel{\delta_p^{(r)}}\lar \Gamma\bigl(U_p \times E_p, (\kA \boxtimes \kA)(\Sigma)\bigr), \; 
f \mapsto \left[f \otimes 1 + 1 \otimes f, -\varrho^{\sharp}\big|_{U_p \times E_p} \right].
\end{equation}
It follows from the skew-symmetry of $\rho$ that both maps $\delta_p^{(l)}$  and 
$\delta_p^{(r)}$ can be glued to a linear map
 $\lieA_{(p)} \stackrel{\delta_p^{(t)}}\lar \Gamma\bigl(E_p \times E_p, (\kA \boxtimes \kA)(\Sigma)\bigr)$.
Let $\tilde{\delta}_p^{(t)}$ be the composition
$$
\lieA_{(p)} \stackrel{\delta_p^{(t)}}\lar \Gamma\bigl(E_p \times E_p, (\kA \boxtimes \kA)(\Sigma)\bigr)
\longhookrightarrow  \Gamma\bigl((E_p \times E_p) \setminus \Sigma, \kA \boxtimes \kA\bigr).
$$
Consider the linear map 
$$
\lieA_{(p)}^\circ \stackrel{\delta_p^{(\rho)}}\lar  \Gamma\bigl((U_p \times U_p) \setminus \Sigma, \kA \boxtimes \kA\bigr), \; f \mapsto \bigl[f \otimes 1 + 1 \otimes f, \rho\big|_{U_p\times U_p}\bigr].
$$
For any $f \in \lieA_{(p)}^\circ$ the section $\delta_p^{(\rho)}(f)$
has no pole along the diagonal; see \cite[Proposition 4.12]{BurbanGalinat}.
It follows from the  commutative diagram
\begin{equation}
\begin{array}{c}
\xymatrix{
\lieA_{(p)}  \ar[r]^-{\tilde{\delta}_p^{(t)}} \ar@{_{(}->}[d] & \Gamma\bigl((E_p \times E_p) \setminus \Sigma, \kA \boxtimes \kA\bigr) \ar@{^{(}->}[d] \\
\lieA_{(p)}^\circ \ar[r]^-{\delta_p^{(\rho)}} \ar[rd]_-{\delta_p^{(\rho)}} & \Gamma\bigl((U_p \times U_p) \setminus \Sigma, \kA \boxtimes \kA\bigr) \\
& \Gamma\bigl(U_p \times U_p, \kA \boxtimes \kA\bigr) \ar@{_{(}->}[u]
}
\end{array}
\end{equation}
that $\tilde{\delta}_p^{(t)}$ can be extended to a linear map
$\lieA_{(p)} \stackrel{\delta_p}\lar \Gamma\bigl((E_p \times E_p) \setminus \{(s, s)\}, \kA \boxtimes \kA\bigr)$. It remains to note that $R_p \otimes_{\CC} R_p$ is a reduced Cohen--Macaulay $\CC$--algebra of Krull dimension two and 
$\lieA_{(p)} \otimes_{\CC} \lieA_{(p)}$ is a maximal Cohen--Macaulay $(R_p \otimes_{\CC} R_p)$--module. As a consequence,  the canonical restriction map 
$$\lieA_{(p)} \otimes \lieA_{(p)} \cong \Gamma\bigl(E_p \times E_p, \kA \boxtimes \kA\bigr)\lar \Gamma\bigl((E_p \times E_p) \setminus \{(s, s)\}, \kA \boxtimes \kA\bigr)$$ is an isomorphism; see e.g.~\cite[Section 3]{SurvOnCM}. It follows that $\delta_p$ can be extended to a linear map $\lieA_{(p)} \stackrel{\delta_p}\lar \lieA_{(p)} \otimes \lieA_{(p)}$. According to \cite[Proposition 4.12]{BurbanGalinat}, 
 $\lieA_{(p)}^\circ \stackrel{\delta_p^{(\rho)}}\lar \lieA^\circ_{(p)} \otimes \lieA^\circ_{(p)}$  is a Lie bialgebra cobracket. It follows that 
 $(\lieA_{(p)}, \delta_p)$ is a Lie bialgebra, too. Moreover, identifying the Manin triples $\widetilde{\lieA}_p = \widehat{\lieA}_p \dotplus \lieA_{(p)}$ and  $\lieR =  \lieg\llbracket y\rrbracket \dotplus \lieW$, the cobracket $\delta_p$ gets identified with the cobracket (\ref{E:cobracketpowerseries}) on the Lie algebra $\lieW$.  \hfill $\lozenge$
\end{remark}

\begin{proposition}\label{P:LieBialgSmoothPoint}
Let $(E, \kA)$ be a geometric CYBE datum and $p \in U$. Then the Lie bialgebra cobracket  $\lieA_{(p)} \stackrel{\delta_p}\lar \lieA_{(p)} \otimes \lieA_{(p)}$ is determined by the Manin triple $\widetilde{\lieA}_p = \widehat{\lieA}_p \dotplus \lieA_{(p)}$.
\end{proposition}
\begin{proof}
It is a consequence of Proposition \ref{P:BielgebraPowerSeriesDetermined}.
\end{proof}

\subsection{Manin triples and geometric CYBE data on singular Weierstra\ss{} curves}\label{SS:NodalManinTriples}
Let $(E, \kA)$ be a geometric CYBE datum, where $E$ is a singular Weierstra\ss{} curve. As in the previous subsection, let $s$ be the singular point of $E$ and $U = E \setminus \{s\}$. To simplify the notation, we denote:  $\widehat{O} = \widehat{O}_s$, $\widehat{Q} = \widehat{Q}_{s}$ and $R = R_{(s)}$ as well as 
$\widehat{\lieA}= \widehat{\lieA}_s$,   $\lieA = \lieA_{(s)}$ and  $\widetilde{\lieA} = \widetilde{\lieA}_{s}$. Moreover, let $\PP^1 \stackrel{\nu}\lar E$ be the normalization map.

\smallskip
\noindent
Apart of Remark \ref{R:CuspidalCase}, we assume in this subsection that 
$E$ is nodal.  Let $s_\pm \in \PP^1$ be  such that $\nu(s_\pm) = s$. Next, let 
$\widehat{O}_\pm$ be the completion of the stalk of $\kO_{\PP^1}$ at $s_\pm$ and $\widehat{Q}_\pm$ be the fraction field of $\widehat{O}_\pm$. Then we have an injective homomorphism of $\CC$--algebras $\widehat{O} \stackrel{\nu^\ast}\lar  \widehat{O}_+ \times \widehat{O}_-$, which induces an isomorphism of the corresponding total rings of fractions $\widehat{Q} \stackrel{\nu^\ast}\lar  \widehat{Q}_+ \times \widehat{Q}_-$.

\smallskip
\noindent
We choose homogeneous coordinates $(w_+: w_-)$ on $\PP^1$ so that $s_+ = (0:1)$ and $s_- = (1: 0)$. Then the rational functions  $u = u_+ := \dfrac{w_+}{w_-}$ and $u_- := \dfrac{w_-}{w_+}$ are local parameters at the points $s_+$ and $s_-$, respectively.  In these terms we have an  algebra isomorphism
$$
R = \Gamma(U, \kO) \stackrel{\nu^\ast}\lar \Gamma\bigl(\nu^{-1}(U), 
\kO_{\PP^1}\bigr) \cong  \CC\bigl[u, u^{-1}\bigr]$$
as well as  
$
\widehat{O}_\pm \cong \CC\llbracket u_\pm\rrbracket$, $
\widehat{Q}_\pm \cong \CC\llbrace u_\pm\rrbrace$, $\widehat{Q} \cong  \CC\llbrace u_+\rrbrace \times \CC\llbrace u_-\rrbrace$ and 
$
\widehat{O} \cong \CC\llbracket u_+, u_-\rrbracket/(u_+ u_-).
$
We shall view the following rational  differential one-form on $\PP^1$ $$\omega:= \dfrac{du}{u} = \dfrac{du_+}{u_+} = -\dfrac{du_-}{u_-}$$   as a generator of $\Gamma(E, \Omega)$. It follows from the assumption that $\kA$ is weakly $\lieg$-locally free that the Killing form $\lieA \times \lieA \stackrel{K}\lar \widehat{Q}$ is non-degenerate. Hence, the Killing form
$\widetilde\lieA \times \widetilde\lieA \stackrel{\widetilde{K}}\lar \widehat{Q}$ is non-degenerate, too. Recall that the Rosenlicht residue map $\widehat{Q}\xrightarrow{\res_{s}^\omega} \CC$ 
 with respect to the form $\omega$ is given by the formula
\begin{equation}
\res^\omega_s(f) = \res_{s_+}\bigl(f_+ \omega) + \res_{s_-}\bigl(f_- \omega) = 
\res_{0}\left(f_+ \dfrac{du_+}{u_+}\right) - \res_{0}\left(f_- \dfrac{du_-}{u_-}\right),
\end{equation}
where we use the identifications $f = (f_+, f_-) \in  \widehat{Q} \cong \widehat{Q}_+ \times \widehat{Q}_- \cong \CC\llbrace u_+\rrbrace \times \CC\llbrace u_-\rrbrace$. Similarly to (\ref{E:FormAtp}),  we get an invariant symmetric bilinear form
$\widetilde\lieA \times \widetilde\lieA \stackrel{\widetilde{F}^\omega_s}\lar \CC$ given by 
\begin{equation}\label{E:Forma}
\widetilde\lieA \times \widetilde\lieA \stackrel{\widetilde{K}}\lar \widehat{Q} 
\stackrel{\res_s^\omega}\lar \CC.
\end{equation}
It is easy to see that $\widetilde{F}^\omega_s$ is non-degenerate. 

\smallskip
\noindent
It can be shown that the Mayer--Vietoris sequence (\ref{E:MayerVietoris}) is exact at the singular point $s$ as well; see e.g.~\cite[Theorem 3.1]{ThesisGalinat}. It follows from the cohomology vanishing 
$H^0(E, \kA) = 0 = H^1(E, \kA)$ that we have a Manin triple $\widetilde\lieA = 
\widehat\lieA \dotplus \lieA$. According to \cite[Proposition 4.12]{BurbanGalinat}
\begin{equation}\label{E:GeomCobracket}
\lieA  \stackrel{\delta}\lar \lieA \otimes \lieA, \; f \mapsto [f \otimes 1 + 1 \otimes f, \rho]
\end{equation}
is a Lie bialgebra cobracket, where $\rho \in \Gamma\bigl((U \times U) \setminus \Sigma, \kA \boxtimes \kA\bigr)$ is the geometric $r$-matrix. 

\begin{theorem}\label{T:main}
Let $(E, \kA)$ be a geometric CYBE datum, where $E$ is a nodal Weierstra\ss{} curve. Then the Lie bialgebra cobracket  (\ref{E:GeomCobracket})  is determined by the Manin triple $\widetilde{\lieA} = \widehat{\lieA} \dotplus \lieA$.
\end{theorem}

\begin{proof}
For any $k \in \NN$ we put:
\begin{itemize}
\item 
$P^{(k)} := \widehat{O}/\idm^k \otimes_{\CC} R$, $\widetilde{P}^{(k)}_\pm  := \widehat{O}_\pm/\idm_\pm^k \otimes_{\CC} R$ and $\widetilde{P}^{(k)} := 
\widetilde{P}^{(k)}_+ \times \widetilde{P}^{(k)}_-$.
\item $X^{(k)} := \Spec(P^{(k)})$, $\widetilde{X}^{(k)}_\pm  := \Spec(\widetilde{P}^{(k)}_\pm)$  and $\widetilde{X}^{(k)} := 
\widetilde{X}^{(k)}_+ \sqcup \widetilde{X}^{(k)}_-$.
\end{itemize}
Then we set: $P := \varprojlim (P^{(k)})$, $\widetilde{P}_\pm := \varprojlim (\widetilde{P}_\pm^{(k)})$, $\widetilde{P} = \widetilde{P}_+ \times \widetilde{P}_-$, $X := \Spec(P)$, $\widetilde{X}_\pm := \Spec(\widetilde{P}_\pm)$ and $\widetilde{X} := \Spec(\widetilde{P}) = 
\widetilde{X}_+ \sqcup \widetilde{X}_-$. Note that  $P \cong \CC[v, v^{-1}]\llbracket u_+, u_-\rrbracket/(u_+ u_-)$ and $\widetilde{P}_\pm  \cong 
\CC[v, v^{-1}]\llbracket u_\pm \rrbracket$. Finally, let $D := \CC\bigl[u, u^{-1}, v, v^{-1}\bigr]$, 
$S_\pm := \CC[v, v^{-1}]\llbrace u_\pm\rrbrace$, $Y_\pm := \Spec(S_\pm)$, 
$S := S_+ \times S_-$  and 
$Y := Y_+ \sqcup Y_-$. 
Consider the algebra homomorphism
$$
D \stackrel{\psi}\lar S,  \; u \mapsto (u_+, u_-^{-1}), v \mapsto (v, v).
$$
The formulae 
$
(u_+ - v) \sum\limits_{k = 0}^\infty  v^{-k-1} u_+^k = -1 $ and $
(u_-^{-1} - v) \sum\limits_{k = 0}^\infty   v^{k} u_-^{k+1} = 1
$
imply that $\psi(u-v)$ is a unit in $S$. As a consequence, $\psi$ can be extended to the algebra homomorphism
$
\CC\left[u, u^{-1}, v, v^{-1}, \dfrac{1}{u-v}\right] \stackrel{\widetilde\psi}\lar S.
$
Note that
\begin{equation}\label{E:powerseriesformula}
\begin{array}{c}
\xymatrix{
\widetilde{\psi}\left(\dfrac{v}{u-v}\right) = \left(- \sum\limits_{k = 0}^\infty v^{-k} u_+^k, \sum\limits_{k = 1}^\infty v^{k} u_-^{k}\right).
}
\end{array}
\end{equation}
Next,  we have a family of morphisms of schemes $\left(X^{(k)} \stackrel{\varepsilon_k}\lar (E \times U)\setminus \Sigma\right)_{k \in \NN}$. Taking the corresponding direct limit, we get a morphism $X \stackrel{\varepsilon}\lar (E \times U)\setminus \Sigma$. In a similar way, we have a family of morphisms $\left(\widetilde{X}^{(k)} 
\stackrel{\widetilde{\varepsilon}_k}\lar (\PP^1 \times U)\setminus \Sigma\right)_{k \in \NN}$
as well as the corresponding direct limit  
$\widetilde{X} \stackrel{\widetilde{\varepsilon}}\lar (\PP^1 \times U)\setminus  \Sigma$. 
Summing up, we get  the following commutative diagram in the category of schemes:
\begin{equation}\label{E:CommDiagramKeyNodalCase}
\begin{array}{c}
\xymatrix{
(E \times U)\setminus \Sigma & (\PP^1 \times U)\setminus \Sigma \ar[l]_-{\widetilde{\nu}} & 
\ar@{_{(}->}[l]_-\imath (U \times U)\setminus \Sigma \ar@{^{(}->}[r] & U \times U \\
X^{(k)} \ar[u]_-{\varepsilon_k} \ar[d]^-{\pi_k} & \widetilde{X}^{(k)} 
\ar[l]_{\bar{\nu}_k} 
\ar[u]^-{\widetilde{\varepsilon}_k} \ar[d]_-{\widetilde\pi_k} & & \\
X \ar@/^20pt/[uu]^-{\varepsilon} & \widetilde{X} \ar[l]_-{\bar{\nu}} \ar@/_20pt/[uu]_-{\widetilde\varepsilon} &  Y \ar[l]_-{\eta} \ar[uu]_-{
\widetilde\jmath} \ar[uur]_-\jmath & \\
}
\end{array}
\end{equation}
where  $\widetilde{\nu}$ is the restriction of $\nu \times \id$ on $(\PP^1 \times U)\setminus \Sigma$ and 
$\bar{\nu}_k$, $\bar{\nu}$, $\eta$, $\jmath$ and  $\widetilde{\jmath}$ are morphisms of affine schemes corresponding to the algebra embeddings 
 $P^{(k)} \hookrightarrow  \widetilde{P}^{(k)}$, 
 $P\hookrightarrow  \widetilde{P}$, $\widetilde{P} \hookrightarrow S$,  
 $\psi$ and $\widetilde{\psi}$, respectively. 

\smallskip
\noindent
Since $\kA$ is torsion free, we get an injective map
$$
\Gamma\bigl((E \times U) \setminus \Sigma, \kA \boxtimes \kA\bigr) 
\stackrel{\varepsilon^\ast}\lar 
\overline{\widehat{\lieA} \otimes_{\CC} \lieA} := \varprojlim 
\left(\widehat{\lieA}/\idm^k \widehat\lieA \otimes_{\CC} \lieA\right) \cong 
\Gamma\bigl(X, \varepsilon^\ast\bigl(\kA \boxtimes \kA\big|_{(E \times U)\setminus \Sigma}\bigr)\bigr).
$$

\smallskip
\noindent
Let $\Upsilon$ be a countable set and $(a_l)_{l \in \Upsilon}$ be a basis of $\lieA$ over $\CC$. Then there exists a uniquely determined family $(b_{l})_{l \in \Upsilon}$  of elements of $\widehat\lieA$ such that for any $k \in \NN$ there exists a finite subset 
$\Upsilon_k \subset \Upsilon$ satisfying the following properties:
\begin{itemize}
\item  the class $b_l^{(k)}$  of $b_l$ in $\widehat{\lieA}/\idm^k \widehat{\lieA}$ is zero for all $l \notin \Upsilon_k$ (i.e. $b_l \in \idm^k 
\widehat{\lieA}$ for $l \notin \Upsilon_k$)  and 
\item $\varepsilon_k^\ast(\varrho) = \sum\limits_{l \in \Upsilon_k} b_l^{(k)} \otimes a_l$.
\end{itemize}
In these terms we may informally write: $
\varepsilon^\ast(\varrho) = \sum\limits_{l \in \Upsilon} b_l \otimes a_l \in 
\overline{\widehat{\lieA} \otimes_{\CC} \lieA}$.

\noindent
Let $\Upsilon = \left\{(k, i) \, \big| \, k \in \ZZ, 1 \le i \le q \right\}$,
$(c_1, \dots, c_q)$ be a basis of $\lieA$ viewed as module over $R = \CC\bigl[v, v^{-1}\bigr]$ and $a_{(k, i)} := c_i v^k$ for $(k, i) \in \Upsilon$.  Then 
$(a_{(k,i)})_{(k, i) \in \Upsilon}$ is a basis of $\lieA$ viewed as a vector space over $\CC$.  From what was said above it follows that there exists a uniquely determined family of elements $(b_{(k, i)})_{(k, i) \in \Upsilon}$ of $\widehat\lieA$ such that 
\begin{equation}\label{E:RMatrixExpansion}
\varepsilon^\ast(\varrho) = \sum\limits_{(k, i) \in \Upsilon} b_{(k, i)} \otimes a_{(k, i)}.
\end{equation}
Let $(c_1^\ast, \dots, c_q^\ast)$ be the dual basis of $\lieA$ with respect to the Killing form $\lieA \times \lieA \stackrel{K}\lar R$. Then the tensor
$c_1^\ast \otimes c_1 + \dots + c_q^\ast \otimes c_q \in \lieA \otimes_{\CC} \lieA$ is mapped to the Casimir element of $\lieA$ under the canonical projection $\lieA \otimes_{\CC} \lieA \twoheadarrow   \lieA \otimes_{R} \lieA$. Since $\omega\big|_{U} = \dfrac{dv}{v}$, the geometric $r$-matrix $\rho$ has the following presentation:
\begin{equation}
\rho = \dfrac{v}{u-v} \sum\limits_{i = 1}^q c_i^\ast \otimes c_i +  h(u, v) \in 
\Gamma\left((U \times U)\setminus \Sigma, \kA \boxtimes \kA\right),
\end{equation}
where $h  \in \lieA \otimes_{\CC} \lieA$; see (\ref{E:GeomRMatrixLocal}). It follows from (\ref{E:powerseriesformula}) that we have the following expansion 
$$
\widetilde{\jmath}^\ast(\rho) = 
\sum\limits_{(k, i)\in \Upsilon} (w_{(k, i)} + h_{(k, i)}) \otimes a_{(k, i)},
$$
where $h_{(k, i)} \in \lieA \subset \widetilde\lieA = \widetilde\lieA_+ \times 
\widetilde\lieA_-
$
are determined by the expression  $h = \sum\limits_{(k, i) \in \Upsilon} h_{(k, i)} \otimes a_{(k, i)}$ (which is a finite sum in $\lieA \otimes_\CC \lieA$) and 
\begin{equation}\label{E:dualelements}
\widetilde\lieA_+ \times 
\widetilde\lieA_- \ni w_{(k, i)} = 
\left\{
\begin{array}{lcc}
\bigl(0, u_-^k c_i^\ast) & \mbox{if} & k \ge 1 \\
\bigl(-u_+^{-k} c_i^\ast, 0\bigr) & \mbox{if} & k \le 0. \\
\end{array}
\right.
\end{equation}
It follows from (\ref{E:CommDiagramKeyNodalCase})  that
$(\bar{\nu} \eta)^\ast\bigl(\varepsilon^\ast(\varrho)\bigr) = \widetilde{\jmath}^\ast(\rho)$. Hence,  for any $(k, i) \in \Upsilon$ we have:
\begin{equation}\label{E:expansiongeomrmatrnodalpoint}
\widehat\lieA \ni b_{(k, i)}  = w_{(k, i)} + h_{(k, i)} \in \widetilde\lieA = 
\widetilde\lieA_+ \times 
\widetilde\lieA_-.
\end{equation}
Since all $h_{(k, i)}$ but finitely many are zero, 
$b_{(k, i)}  = w_{(k, i)}$ for all but finitely many $(k, i) \in \Upsilon$. As $\lieA$ is an isotropic subalgebra of $\widetilde\lieA$, we deduce  from (\ref{E:dualelements}) the following relation: 
\begin{multline}\label{E:Orthogonality}
F\bigl(b_{(k', i')}, a_{(k'', i'')}\bigr) = F\bigl(w_{(k', i')}, a_{(k'', i'')}\bigr) = F\Bigl(w_{(k', i')}, \bigl(u_+^{k''} c_{i''}, u_-^{-k''} c_{i''}\bigr)\Bigr)  \\
=  -\delta_{k' k''} \delta_{i' i''}, 
\end{multline}
where
$F = \widetilde{F}^\omega_s$ is the form given by (\ref{E:Forma}).
This formula in particular implies that the elements $(b_{(k, i)})_{(k, i) \in \Upsilon}$ are linearly independent. It follows from the direct sum decomposition 
$\widetilde\lieA = \widehat\lieA \dotplus \lieA$ that $(b_{(k, i)})_{(k, i) \in \Upsilon}$ is in fact a topological basis of $\widehat\lieA$. 

\smallskip
\noindent
After establishing these preparatory results, we can proceed to the proof of the actual statement:
$
F\bigl(\delta(a), b' \otimes b'' \bigr) = 
F\bigl(a, [b', b''] \bigr)
$
for all $a \in \lieA$ and $b', b'' \in \widehat{\lieA}$.
Arguing as in the proof of Proposition \ref{P:BielgebraPowerSeriesDetermined}, we conclude that it is sufficient to prove the formula
\begin{equation}\label{E:DeltaIsDefined}
F\bigl(\delta(a), b_{(k', i')} \otimes b_{(k'', i'')} \bigr) = 
F\bigl(a, \bigl[b_{(k', i')},  b_{(k'', i'')}\bigr]\bigr)
\end{equation}
for any $(k', i'), (k'', i'') \in \Upsilon$. In order to use the expansion (\ref{E:RMatrixExpansion}), we embed $\widetilde\lieA \otimes \widetilde\lieA$ into a larger vector space $\overline{\widetilde\lieA \otimes \widetilde\lieA}$ defined as follows.

\smallskip
\noindent
Let $T^{+}_{\pm} := \CC\llbrace v_+ \rrbrace\llbrace u_\pm \rrbrace$, 
$T^{-}_{\pm} := \CC\llbrace v_- \rrbrace\llbrace u_\pm \rrbrace$, $T_\pm := T_\pm^+ \times T_\pm^-$
 and
$T := T_+ \times T_-$. Clearly, we have injective algebra homomorphisms
$
S_\pm \longhookrightarrow  T_\pm, u_\pm \mapsto u_\pm, v \mapsto \bigl(v_+, v_-^{-1}\bigr)$
which define  the embedding $S \longhookrightarrow  T$. Summing up, we have two chains
of algebra embeddings
$$
P \longhookrightarrow \widetilde{P} \longhookrightarrow S \longhookrightarrow T \quad \mbox{and} \quad D \longhookrightarrow S \longhookrightarrow T.
$$
Now we  put: $\overline{\widetilde\lieA \otimes \widetilde\lieA}:= T \otimes_{D} \bigl(\lieA \otimes_{\CC} \lieA\bigr)$ and $\overline{\widetilde\lieA \otimes \lieA}:= S \otimes_{P} \bigl(\overline{\widehat\lieA \otimes \lieA}\bigr)$. It is clear that $
 \overline{\widetilde\lieA \otimes \widetilde\lieA}  \cong T \otimes_{S} \bigl(\overline{\widetilde\lieA \otimes \lieA}\bigr)$. Moreover, we  have canonical injective linear maps 
$\widetilde{\lieA} \otimes \widetilde{\lieA} \hookrightarrow
\overline{\widetilde\lieA \otimes \widetilde\lieA} $ and 
$\overline{\widehat\lieA \otimes \lieA} \hookrightarrow
\overline{\widetilde\lieA \otimes \widetilde\lieA}$, which are moreover morphisms of $\lieA$-modules with respect to the adjoint action of $\lieA$.

\smallskip
\noindent
Consider the following   residue  map:
\begin{equation}\label{E:residuetwodim}
\CC\llbrace v\rrbrace \llbrace u\rrbrace \stackrel{\res}\lar \CC, \sum\limits_{k\ge -\infty} f_k(v) u^k \mapsto \res_0\left(f_0(v) \dfrac{dv}{v}\right).
\end{equation}
The Killing form $\lieA \times \lieA \stackrel{K}\lar R$ together with  the linear map
$T \stackrel{\res}\lar \CC$ defined by 
 (\ref{E:residuetwodim}) define  the bilinear form
$
\overline{\widetilde\lieA \otimes \widetilde\lieA} \times \overline{\widetilde\lieA \otimes \widetilde\lieA} \stackrel{F}\lar  \CC,
$
which extends  $
\bigl(\widetilde\lieA \otimes \widetilde\lieA\bigr) \times \bigl(\widetilde\lieA \otimes \widetilde\lieA\bigr) \stackrel{F}\lar  \CC$.

\smallskip
\noindent
Using  the power series expansion (\ref{E:RMatrixExpansion}), we can write
$\delta(a) = 
[a  \otimes 1 + 1 \otimes a, \rho] \in \lieA \otimes \lieA$ as
$$
\delta(a) = 
\sum\limits_{(k, i) \in \Upsilon}
\bigl[a, b_{(k, i)}\bigr] \otimes a_{(k, i)} + 
\sum\limits_{(k, i) \in \Upsilon} b_{(k, i)} \otimes \bigl[a, a_{(k, i)}\bigr] \in \overline{\widetilde\lieA \otimes \widetilde{\lieA}},
$$
Since $\widehat\lieA$ is an isotropic subspace of 
$\widetilde\lieA$, it follows  that $F(t, b' \otimes b'') = 0$ for any 
$t \in \overline{\widehat\lieA \otimes \lieA}$ and $b', b'' \in \widehat{\lieA}$. 
As a consequence, we have: 
$$
F\bigl(\delta(a), b_{(k', i')} \otimes b_{(k'', i'')} \bigr) =
F\left(\sum\limits_{(k, i) \in \Upsilon} \bigl[a, b_{(k, i)}\bigr] \otimes a_{(k, i)},  b_{(k', i')} \otimes b_{(k'', i'')}\right).
$$
Taking into account  the orthogonality relation (\ref{E:Orthogonality})
 as well
as invariance of the form $F$, we finally get:
$$
F\bigl(\delta(a), b_{(k', i')} \otimes b_{(k'', i'')} \bigr)
= - F\bigl([a, b_{(k'', i'')}],
b_{(k', i')}\bigr) = F\bigl(a, \bigl[b_{(k', i')},  b_{(k'', i'')}\bigr]\bigr),
$$
as asserted. \end{proof}

\smallskip
\noindent
Note that in the course of the proof of Theorem \ref{T:main} we have shown the following result.

\begin{theorem}\label{T:main2}
Let $(E, \kA)$
be as in  Theorem \ref{T:main},  $(c_1, \dots, c_q)$ be a basis of $\lieA$ viewed as module over $R$,
$(c^*_1, \dots, c^*_q)$ be its dual basis with respect to the Killing form $\lieA \times \lieA \stackrel{K}\lar R$, $\Upsilon := \left\{(k, i) \, \big| \, k \in \ZZ, 1 \le i \le q \right\}$, $a_{(k, i)} := c_i v^k$ for $(k, i) \in \Upsilon$ and 
$\bigl(b_{(k, i)}\bigr)_{(k, i) \in \Upsilon}$ be the topological basis of $\widehat{\lieA}$ dual to $(-a_{(k, i)})_{(k, i) \in \Upsilon}$. Then for any $(k, i) \in \Upsilon$ we have: 
$
b_{(k, i)} = w_{(k, i)} + h_{(k, i)},
$
where $w_{(k, i)}$ are given by the formula (\ref{E:dualelements}), $h_{(k, i)} \in \lieA$ and all but finitely many elements $h_{(k, i)}$ are zero. Moreover, the geometric $r$-matrix corresponding to $(E, \kA)$ is given by the following expression:
\begin{equation}\label{E:geomRMatrNodal}
\rho = \dfrac{v}{u-v} \sum\limits_{i = 1}^q c_i^\ast \otimes c_i +  \sum\limits_{(k, i) \in \Upsilon} h_{(k, i)}(u) \otimes v^k c_i .
\end{equation}
\end{theorem}

\begin{remark}\label{R:CuspidalCase} Let $(E, \kA)$ be  a geometric CYBE datum, where $E$ is a cuspidal  plane cubic curve. Then the  cobracket  
$\lieA \stackrel{\delta}\lar \lieA_{} \otimes \lieA_{}$ is determined by the Manin triple $\widetilde{\lieA} = \widehat{\lieA} \dotplus \lieA$.

\smallskip
\noindent
Fix an isomorphism $R = \Gamma(U, \kO_E) \cong \CC[v]$. Then 
$\omega = dv$ is a generator of $\Gamma(E, \Omega_E)$. Let $(c_1, \dots, c_q)$ be a basis of $\lieA$ and $(c^\ast_1, \dots, c^\ast_q)$ be the dual basis of $\lieA$ with respect to the Killing form $\lieA \times \lieA \stackrel{K}\lar R$. Now we put: $\overline{\Upsilon} = \left\{(k, i) \, \big| \, k \in \NN_0, 1 \le i \le q \right\}$. Then  $a_{(k, i)} := c_i v^k$ for $(k, i) \in \overline{\Upsilon}$
form a basis of $\lieA$ over $\CC$.  Let 
$\bigl(b_{(k, i)}\bigr)_{(k, i) \in \overline{\Upsilon}}$ be the topological basis of $\widehat{\lieA}$ dual to $\bigl(a_{(k, i)}\bigr)_{(k, i) \in \overline{\Upsilon}}$. Then for any $(k, i) \in \overline{\Upsilon}$ we have a decomposition
$
b_{(k, i)} = c_i^\ast v^{-k-1} + h_{(k, i)}
$
for some uniquely determined  $h_{(k, i)} \in \lieA$. Again, all but finitely many elements $h_{(k, i)}$ are zero. The geometric $r$-matrix corresponding to $(E, \kA)$ is given by the following expression:
\begin{equation}\label{E:geomRMatrCusp}
\rho = \dfrac{1}{u-v} \sum\limits_{i = 1}^q c_i^\ast \otimes c_i +  \sum\limits_{(k, i) \in \overline{\Upsilon}} h_{(k, i)}(u) \otimes v^k c_i.
\end{equation}
The corresponding proofs are  completely  analogous to the ones of Proposition \ref{P:BielgebraPowerSeriesDetermined} and Theorem \ref{T:main} and therefore are left to an interested reader. \hfill $\lozenge$
\end{remark}

\begin{remark}
Let $(E, \kA)$ be a geometric CYBE datum, where $E$ is an arbitrary Weierstra\ss{} curve. There are also other natural ways to attach to $(E, \kA)$ Lie bialgebras and Manin triples. For example,  let  $p_+ \ne p_- \in E$ be any pair of points such that $s \in \bigl\{p_+, p_-\bigr\}$ provided $E$ is singular, 
$R_{p_+, p_-}:= \Gamma\bigl(E \setminus \{p_+, p_-\}, \kO\bigr)$ and 
$\lieA_{(p_+, p_-)} := \Gamma\bigl(E \setminus \{p_+, p_-\}, \kA\bigr)$. Then we have a Manin triple
$
\lieA_{(p_+, p_-)}\, = \,  \lieA_{(p_+)}\, \dotplus \, \lieA_{(p_-)}, 
$
where the underlying bilinear form $\lieA_{(p_+, p_-)} \times \lieA_{(p_+, p_-)} 
\to  \CC$ is given by the composition
$$
\lieA_{(p_+, p_-)} \times \lieA_{(p_+, p_-)} \stackrel{K}\lar R_{p_+, p_-}
\stackrel{\res_{p_+}^\omega}\lar \CC.
$$
Here, as usual,  $K$ is the Killing form of $\lieA_{(p_+, p_-)}$, viewed as a Lie algebra over $R_{p_+, p_-}$. \hfill $\lozenge$
\end{remark}

\section{Geometrization of twists of the standard Lie bialgebra structure on loop algebras and trigonometric solutions of CYBE}\label{S:GeometrizationTrigonomSolutions}

\subsection{Some basic facts on torsion free sheaves on a  nodal Weierstra\ss{} curve}\label{SS:BasisBBDG}
Let $E$ be a nodal Weierstra\ss{} curve, $s$ be its singular point,  $\PP^1 \stackrel{\nu}\lar   E$ be a normalization morphism and 
$\nu^{-1}(s)  =  \left\{s_+, s_-\right\}$.
Then  the following  diagram in the category of schemes
\begin{equation}\label{E:Bicartesian}
\begin{array}{c}
\xymatrix{
\left\{s_+, s_-\right\} \ar@{^{(}->}[r]^-{\tilde\eta} \ar@{->>}[d]_-{\tilde\nu} & \PP^1 \ar@{->>}[d]^-{\nu} \\
\{s\} \ar@{^{(}->}[r]^-{\eta} & E
}
\end{array}
\end{equation}
is bicartesian, i.e.~ it it both pullback and pushout diagram.
For any torsion free coherent sheaf $\kF$ on $E$, we get the locally free sheaf  $\widetilde\kF := \nu^*\kF/t(\nu^*\kF)$ on $\PP^1$, where  $t(\nu^*\kF)$ denotes the torsion part of $\nu^*\kF$. It is not hard to show that
\begin{itemize}
\item the canonical linear map $\kF\Big|_{s} \lar 
    \widetilde{\kF}\Big|_{s_+} \oplus \widetilde{\kF}\Big|_{s_-}$ is injective.
\item the canonical morphism of $(\CC\times \CC)$--modules $\theta_{\kF}$ given as the composition 
$$
    \tilde{\nu}^*(\kF\Big|_{s}) \lar \tilde\eta^*(\nu^*\kF) \lar \tilde\eta^*(\widetilde\kF) =  \widetilde{\kF}\Big|_{s_+} \oplus \widetilde{\kF}\Big|_{s_-}$$
     is surjective;
\item the following diagram  in the category $\Coh(E)$ of coherent sheaves on $E$ 
\begin{equation*}
\begin{array}{c}
\xymatrix{
\kF \ar[r] \ar[d] & \kF\Big|_{s} \ar[d]\\
\nu_*(\widetilde\kF) \ar[r] & \widetilde{\kF}\Big|_{s_+} \oplus \widetilde{\kF}\Big|_{s_-}
}
\end{array}
\end{equation*}
 is a pullback diagram, where  all morphisms are the canonical ones and skyscraper sheaves supported at $s$ are identified with their stalks.
    \end{itemize}
Consider the comma category $\overline{\Tri}(E)$ associated  with a pair of functors
$$
\xymatrix{
\VB(\PP^1) \ar[r]^-{\FF} & \bigl(\CC\times \CC\bigr)-\mathsf{mod} &  \ar[l]_-{\GG} \CC-\mathsf{mod}},
$$
where $\FF(\kG):= \kG\Big|_{s_+} \oplus \kG\Big|_{s_-}$  for any $\kG \in \VB(\PP^1)$ and 
$\GG = (\CC \times \CC)\otimes_{\CC} \,-\,$.
By definition, any object of $\overline{\Tri}(E)$  is a triple $\bigl(\kG, V, \theta\bigr)$, where $\kG$ is a locally free  coherent sheaf on $\PP^1$, $V$ is a finite dimensional vector space over $\CC$  and $\GG(V) \stackrel{\theta}\lar \FF(\kG)$ 
is given by a pair of  linear maps
$V \stackrel{\theta_\pm}\lar \kG\Big|_{s_\pm}$. 
The definition of morphisms in $\overline{\Tri}(E)$ is straightforward.

\noindent
The following result is a special case of    \cite[Theorem 16]{Survey}; see also 
\cite[Theorem 3.2]{Thesis}.

\begin{theorem}\label{T:keyonTF}
The functor 
$
\TF(E) \stackrel{\EE}\lar  \overline{\Tri}(E),\;  \kF \mapsto \bigl(\widetilde\kF, \kF\big|_{s}, \theta_{\kF}\bigr)
$
is fully faithful. The essential image $\Tri(E)$ of $\TF(E)$ consists of  those triples $\bigl(\kG, V, \theta\bigr)$, for which both linear maps $\theta_\pm$ are surjective and the linear map $\tilde{\theta}= \left(\begin{array}{c} \theta_+ \\ \theta_- \end{array}\right): V  \lar \kG\Big|_{s_+} \oplus \kG\Big|_{s_-}$ is injective, whereas the  essential image of the category $\VB(E)$ consists of those triples $\bigl(\kG, V, \theta\bigr)$, for which $\theta$ is an isomorphism. In other words, the functor 
$
\TF(E) \stackrel{\EE}\lar  \Tri(E)
$
is an equivalence of categories. 
Conversely, given  an object $\kT = \bigl(\kG, V, \theta\bigr)$ of $\Tri(E)$,  consider the torsion free sheaf $\kF$ on $E$ defined as a pullback 
\begin{equation}\label{E:PullBack}
\begin{array}{c}
\xymatrix{
\kF \ar[r] \ar[d] & V \ar[d]^-{\tilde\theta}\\
\nu_*(\kG) \ar[r] & \kG\Big|_{s_+} \oplus \kG\Big|_{s_-}
}
\end{array}
\end{equation}
in the category $\Coh(E)$. Then we have: $\EE(\kF) \cong \kT$. 
\end{theorem}

\begin{remark}
Let  $(\kB, \liea, \theta)$ be an object of $\Tri(E)$, for which 
$\kB$ is a sheaf of Lie algebras on $\PP^1$, $\liea$ is a  Lie algebra  and $\theta$ is a morphism of Lie algebras. Then the torsion free coherent sheaf $\kA$
defined by the  pullback diagram (\ref{E:PullBack}) corresponding to $(\kB, \liea, \theta)$ is a sheaf of Lie algebras on $E$. It follows from (\ref{E:PullBack}) that  the following sequences of vector spaces is exact:
\begin{equation}\label{E:LongSequence}
0 \rightarrow \Gamma(E, \kA) \rightarrow \Gamma(\PP^1, \kB) \oplus \liea 
\xrightarrow{\left(\begin{smallmatrix} \ev_+ & \theta_+ \\ \ev_- & \theta_-\end{smallmatrix}\right)} 
\kB\Big|_{s_+} \oplus \kB\Big|_{s_-} \rightarrow H^1(E, \kA) \rightarrow H^1(\PP^1, \kB) \rightarrow 0, 
\end{equation}
where $\Gamma(\PP^1, \kB) \stackrel{\ev_\pm}\lar \kB\Big|_{s_\pm}$ denotes the canonical evaluation map at the point $s_\pm$. \hfill $\lozenge$
\end{remark}

\subsection{Geometrization of twists of the standard Lie bialgebra structure on twisted loop algebras}\label{SS:GeometrizationTwists}
Now we  return to the setting of Section \ref{S:StandardLiebialgStructure}. Let $\lieM = \lieD \dotplus \lieW$ be a Manin triple as in Theorem \ref{T:ManinTriplesLoopTwists}.  Let
 $\lieV_\pm \subset \lieL$ be Lie subalgebras from  Lemma \ref{L:OrdersFromMT2}. Recall that
 $\lieV_\pm$ is a free module of rank $q$ over $L_\pm = \CC\bigl[t_{\pm}\bigr] \subset R =  \CC\bigl[t, t^{-1}\bigr]$, where $t_\pm = t^{\pm 1}$.  In what follows, we shall view the projective line $\PP^1$ as the 
 pullback of the pair of morphisms 
 $$
 \Spec(L_+) \lar \Spec(R) \longleftarrow \Spec(L_-),
 $$
identifying $\Spec(L_\pm)$ with open subsets $U_\pm \subset \PP^1$ and $\Spec(R)$ with $U := U_+ \cap U_-$. Let $s_\pm \in U_\pm$ be the point corresponding to the maximal ideal $(t_{\pm}) \subset L_\pm$, then  $t_\pm$ is a local parameter at  $s_\pm$. 

\begin{proposition}\label{P:SheafLieAlgProjLine}
There exists a unique coherent sheaf of Lie algebras $\kB$ on $\PP^1$ such that
$\Gamma(V, \kB) \subset \CC(t) \otimes_R \lieL$ for any open subset $V \subseteq \PP^1$ and such that
the following diagram  of Lie algebras
\begin{equation}\label{E:SheafB}
\begin{array}{c}
\xymatrix{
\Gamma(U_+, \kB) \ar@{^{(}->}[r] \ar[d]_-{=} & \Gamma(U, \kB) \ar[d]^-{=} & \Gamma(U_-, \kB) \ar@{_{(}->}[l] \ar[d]^-{=} \\
\lieV_+ \ar@{^{(}->}[r] & \lieL & \lieV_- \ar@{_{(}->}[l]
}
\end{array}
\end{equation}
is commutative.  We have: \begin{equation}\label{E:PropertiesOfB}
\Gamma(\PP^1, \kB) = \lieV_+ \cap \lieV_- \quad \mbox{\rm and} \quad H^1(\PP^1, \kB) = 0.
\end{equation}
The completion of the stalk of $\kB$ at $s_{\pm}$ is naturally isomorphic to $\widehat{\lieW}_\pm$ as a Lie algebra over $\widehat{L}_\pm = \CC\llbracket t_\pm\rrbracket$, where $\lieW_\pm$ is the Lie algebra from 
Lemma \ref{L:OrdersFromMT2}. In particular, we can identify the fiber $\kB\Big|_{s_\pm}$ with the 
Lie algebra $\liew_\pm:= \widehat{\lieW}_\pm/t_\pm \widehat{\lieW}_\pm$. 
\end{proposition}

\begin{proof} Existence and uniqueness of $\kB$ characterized by (\ref{E:SheafB})  is clear. We have the Mayer--Vietoris exact sequence
$$
0 \lar \Gamma(\PP^1, \kB) \lar \Gamma(U_+, \kB) \oplus \Gamma(U_-, \kB) \lar \Gamma(U, \kB) \lar H^1(\PP^1, \kB) \lar 0. 
$$
According to Lemma \ref{L:OrdersFromMT2}, we have: $\lieL = \lieV_+ + \lieV_-$. If follows from
(\ref{E:SheafB}) that the formulae (\ref{E:PropertiesOfB}) are true. The remaining statements are obvious. 
\end{proof}

\smallskip
\noindent
Next, we can \emph{define}  $E$ via the pushout diagram (\ref{E:Bicartesian}). It follows that $E$ is a nodal Weierstra\ss{} curve. Let $\widehat{O}_\pm$ be the completion of the stalk of $\kO_{\PP^1}$ at $s_\pm$,  $\widehat{O}$ be the completion of the stalk of $\kO_{E}$ at $s$ and $\widehat{Q}$ be the total ring of quotients of $\widehat{O}$.  Then we have: $\widehat{O}_\pm \cong \CC\llbracket t_\pm \rrbracket$, 
$\widehat{O} \cong  \CC\llbracket t_+, t_-\rrbracket/(t_+ t_-)$ and 
$\widehat{Q} = \CC\llbrace t_+ \rrbrace \times \CC\llbrace t_- \rrbrace$. 
 According to Lemma \ref{L:OrdersFromMT1}, the completed Lie algebra $\widehat\lieW$ is an $\widehat{O}$--module. We put:
$$
\liew:= \widehat{\lieW}/(t_+, t_-)\widehat{\lieW} \subset \widehat{\lieW}_+/t_+ \widehat{\lieW}_+ \times \widehat{\lieW}_+/t_+ \widehat{\lieW}_- = \liew_+ \times \liew_-.
$$
Again, according to Lemma \ref{L:OrdersFromMT1}, the morphism of Lie algebras 
$\liew \stackrel{\theta_\pm}\lar \liew_\pm$ defined as the composition $\liew \hookrightarrow \liew_+ \times \liew_- 
\twoheadarrow \liew_\pm$ is surjective.  It follows that $(\lieB, \liew, \theta)$ is an object of the category $\Tri(E)$ from Theorem \ref{T:keyonTF}.  

\begin{proposition}\label{P:GeomDataFromTwists} Let  $\kA$ be the  sheaf of Lie algebras on $E$, corresponding to the triple
$(\lieB, \liew, \theta)$.
Then $(E, \kA)$ is a geometric CYBE datum.
\end{proposition}
\begin{proof} We keep the notation of Subsection \ref{SS:NodalManinTriples}.
First observe that
the canonical map 
$
\lieA = \Gamma(U, \kA) \rightarrow
\Gamma\bigl(U, \nu_\ast(\kB)\bigr) =  \lieL
$ is an isomorphism of Lie algebras. This implies that $\kA$ is $\lieg$--weakly locally free; see 
Proposition \ref{P:basicsonloops}. 
Next, 
by   Lemma \ref{L:OrdersFromMT2}, the linear map 
$$
\bigl(\lieV_+ \cap \lieV_-\bigr) \oplus \liew = 
\Gamma(\PP^1, \kB) \oplus \liew 
\xrightarrow{\left(\begin{smallmatrix} \ev_+ & \theta_+ \\ \ev_- & \theta_-\end{smallmatrix}\right)} 
\kB\Big|_{s_+} \oplus \kB\Big|_{s_-} \cong \liew_+ \oplus \liew_-
$$
is an isomorphism. Since $H^1(\PP^1, \kB) = 0$, the exact sequence (\ref{E:LongSequence}) implies  that
$H^0(E, \kA) = 0 = H^1(E, \kA)$.
Moreover, it follows from the construction of $\kA$ that  the canonical morphism of Lie algebras 
$\widehat\lieA = \widehat{\kA}_s \stackrel{\nu^\ast}\lar \widehat{\kB}_{s_+} \times \widehat{\kB}_{s_-} \cong \widehat{\lieW}_+ \times \widehat{\lieW}_-$ is injective and 
its image is  the Lie algebra  $\widehat\lieW$. Hence, 
$\widetilde\lieA = \widehat{Q} \otimes_{\widehat{O}} \widehat\lieA$ can be identified with the Lie algebra 
$\widehat{\lieM}$.

\smallskip
\noindent
It follows from the construction of $E$ that the differential form $\omega = \dfrac{dt}{t}$ is a generator of $\Gamma(E, \Omega_E)$. 
The following observation is crucial: under the  isomorphism $\widetilde{\lieA} \rightarrow \widehat\lieM$  the bilinear form $\widetilde\lieA \times \widetilde\lieA \stackrel{\widetilde{F}^\omega_s}\lar \CC$  given by (\ref{E:Forma})  gets identified  (up to an appropriate rescaling) with the bilinear form $\widehat{\lieM} \times \widehat{\lieM} \stackrel{\widehat{F}}\lar \CC$, given by (\ref{E:FormOnLiem})! Summing up, 
$\widetilde\lieA = \widehat\lieA \dotplus \lieA$ is a Manin triple, isomorphic to
the Manin triple $\widehat{\lieM} =  \widehat{\lieW} \dotplus \lieD$. In particular, $\widehat\lieA$ is an isotropic Lie subalgebra of $\widetilde\lieA$.

\smallskip
\noindent
All together, we have proven that   $\kA$ is an acyclic, $\lieg$--weakly locally free isotropic coherent sheaf of Lie algebras on $E$, as asserted. 
\end{proof}

\smallskip
\noindent
Let $(E, \kA)$ be a geometric datum as in Proposition \ref{P:GeomDataFromTwists} above
and $\rho  \in \Gamma\bigl(U \times U \setminus \Sigma, \kA \boxtimes \kA\bigr)$ the corresponding geometric 
$r$--matrix. Recall that the construction of $\kA$ also provides an isomorphism of Lie algebras
$\lieA \stackrel{\cong}\lar \lieL$. 
Let $\widetilde{U} = \Spec(\overline{R}) \stackrel{\pi}\lar U = \Spec(R)$ be the \'etale covering corresponding to the algebra extension $R \subseteq \overline{R}$. By Proposition
\ref{P:basicsonloops}, we have an isomorphism of Lie algebras
$\Gamma\bigl(\widetilde{U}, \pi^\ast(\kA)\bigr) \cong \overline{R} \otimes_{R} \lieL \cong \overline{\lieL}$. The pullback \begin{equation}\label{E:GeomTrigonom}
\tilde{\rho}:= (\pi \times \pi)^\ast(\rho) \in \Gamma\bigl(\widetilde{U} \times \widetilde{U} \setminus \widetilde{\Sigma}, \pi^\ast(\kA) \boxtimes \pi^\ast(\kA)\bigr)
\end{equation}
satisfies the equalities (\ref{E:SheafyCYBE1}) and (\ref{E:SheafyCYBE2}), where
$\widetilde\Sigma = (\pi \times \pi)^{-1}(\Sigma)$. 
Trivializing $\pi^\ast(\kA)$ as above,  we get from $\tilde{\rho}$ a genuine skew-symmetric non-degenerate solution of the classical Yang--Baxter equation (\ref{E:CYBE}). Our next goal is to compute this solution explicitly. 

\subsection{Geometric $r$--matrix corresponding to twists of the standard Lie bialgebra structure of a twisted loop algebra} \label{SS:GeomRMatTwists}
Recall our notation:  $\lieg$ is  a finite dimensional complex simple Lie algebra of dimension $q$,  $\sigma \in \Aut_{\CC}(\lieg)$ is an automorphism of order $m$, 
$
\lieg = \oplus_{k = 0}^{m-1} \lieg_{k}$ the corresponding  decomposition of $\lieg$ into a direct sum of eigenspaces of $\sigma$, $
\gamma = \sum\limits_{k = 0}^{m-1} \gamma_{k}
$
the decomposition of the Casimir element $\gamma \in \lieg \otimes \lieg$ 
with components  $\gamma_{k} \in \lieg_{k} \otimes \lieg_{-k}$. 
Let 
$
\lieg_{0} = \lieg_{0}^+ \oplus \lieh \oplus \lieg_{0}^-
$
be a triangular decomposition as in Remark \ref{R:Triangular}. We denote by $\gamma_{0}^{0}$ 
and $\gamma_{0}^{\pm}$ the projections of $\gamma_0$ on $\lieh \otimes \lieh$ and 
$\lieg_{0}^\pm \otimes \lieg_{0}^\mp$, respectively. 

\begin{proposition}\label{P:GeomStandardStructure}
Let $\lieM = \lieD \dotplus \lieW^\circ$ be the Manin triple (\ref{E:StandardManinTripleLoops}), corresponding to the standard Lie bialgebra cobracket $\lieL 
\stackrel{\delta_\circ}\lar 
 \wedge^2(\lieL)$ and $(E, \kA_\circ)$ be the corresponding geometric CYBE datum defined in Proposition \ref{P:GeomDataFromTwists}. Then the trivialization of the corresponding geometric 
$r$-matrix (\ref{E:GeomTrigonom}) gives the following solution of (\ref{E:CYBE}):
\begin{equation}\label{E:RMatStandard}
r_\circ(x, y) = \left(\frac{\gamma_{0}^{0}}{2} + \gamma_{0}^-\right) + \frac{y^m}{x^m - y^m} \sum\limits_{k = 0}^{m-1} \left(\frac{x}{y}\right)^k \gamma_k.
\end{equation}
\end{proposition}
\begin{proof}
Let $q_k = \dim_{\CC}\bigl(\lieg_k\bigr)$ for $k \in \ZZ$.  By Lemma \ref{L:Orthogonality}, we can choose 
a basis  $\bigl(g_k^{(1)}, \dots, g_k^{(q_k)}\bigr)$ of  $\lieg_k$ such that 
$\kappa\Bigl(g_k^{(i)}, g_{-k}^{(j)}\Bigr) = \delta_{ij}$ for all $1 \le i, j \le q_k$. For $k = 0$ we make an additional  choice: let $(h_1, \dots, h_r)$ be a basis of $\lieh$ and $(e_1^\pm, \dots, e_{p}^\pm)$ a basis of $\lieg_0^\pm$ such that 
$$
\kappa(h_\imath, h_\jmath) = \delta_{\imath \jmath} \; \mbox{for all}\; 1 \le \imath, \jmath \le r \; \mbox{and}\; 
\kappa(e_\imath^+, e_\jmath^-) = \delta_{\imath \jmath} \; \mbox{for all}\; 1 \le \imath, \jmath \le p.
$$
Then we have the following basis of $\lieL = \bigoplus\limits_{k \in \ZZ} \lieg_k z^k$ viewed as a module over 
$R = \CC\bigl[t, t^{-1}\bigr]$:
\begin{equation}\label{E:BasisOfL}
\bigl(e_1^+, \dots, e_p^+, h_1, \dots, h_r, e_1^-, \dots, e_p^-, 
g_{1}^{(1)} z, \dots, g_{1}^{(q_1)} z, \dots,  g_{m-1}^{(1)} z^{m-1}, \dots, g_{m-1}^{(q_{m-1})} z^{m-1}\bigr)
\end{equation}
 where $t = z^m$. 
As usual, let  $\lieL \times \lieL
\stackrel{K}\lar R$ be the Killing form. 
For any $\lambda \in \CC^\ast$, let $\bigl(R/(t-\lambda)\bigr) \otimes_{R} \lieL \stackrel{\varepsilon_\lambda}\lar \lieg$ be the Lie algebra isomorphism from Proposition
\ref{P:basicsonloops} and $R \stackrel{\ev_\lambda}\lar \CC$ be the evaluation map.  Then the diagram
$$
\xymatrix{
\lieL \times \lieL \ar[r]^-K \ar[d]_{\varepsilon_\lambda \times \varepsilon_\lambda} 
& R \ar[d]^-{\ev_\lambda} \\
\lieg \times \lieg \ar[r]^-{\kappa} & \CC
}
$$
is commutative and 
$$
\bigl(e_1^-, \dots, e_p^-, h_1, \dots, h_r, e_1^+, \dots, e_p^+, 
g_{-1}^{(1)} z^{-1}, \dots, g_{-1}^{(q_1)} z^{-1}, \dots,  g_{1-m}^{(1)} z^{1-m}, \dots, g_{1-m}^{(q_{m-1})} z^{1-m}\bigr)
$$
is the basis of $\lieL$ over $R$ which is dual to (\ref{E:BasisOfL}) with respect to the Killing form $K$. Hence, 
$$
\chi = \left(\sum\limits_{\imath = 1}^p \bigl(e_\imath^+ \otimes e_\imath^- + 
e_\imath^- \otimes e_\imath^+ \bigr) + \sum\limits_{l = 1}^r h_l \otimes h_l\right) +
\left(\sum\limits_{k = 1}^{m-1} \sum\limits_{j = 1}^{q_k} g_k^{(j)} z^k \otimes 
g_{-k}^{(j)} z^{-k}\right) 
 \in \lieL \otimes_R \lieL
$$
is the canonical Casimir element of $\lieL$ (viewed as a Lie algebra over $R$).

\smallskip
\noindent
We identify $\rho$ with  $\tilde{\rho} \in \Gamma\bigl(\widetilde{U} \times \widetilde{U} \setminus \widetilde{\Sigma}, \pi^\ast(\kA) \boxtimes \pi^\ast(\kA)\bigr)$. To proceed with computations,  we make the following choices: let $(u, v)$ be coordinates on $\CC^\ast \times \CC^\ast \cong 
U \times U $  and $(x, y)$ be  coordinates on the \'etale covering $\CC^\ast \times \CC^\ast \cong   \widetilde{U} \times \widetilde{U}$. We have: $u = x^m$ and $v = y^m$. Consider the following expression: 
$$
\widetilde{\chi} :=  \sum\limits_{k = 0}^{m-1} \gamma_k   \left(\dfrac{x}{y}\right)^k \in \lieL \otimes_{\CC} \lieL \subseteq (\lieg \otimes \lieg)\bigl[x, x^{-1}, y, y^{-1}\bigr].
$$ 
It is easy to see that $\widetilde{\chi}$ is mapped to 
 $\chi$ under the canonical  linear map $\lieL  \otimes_{\CC} \lieL
 \twoheadarrow  \lieL\otimes_{R} \lieL$.

\smallskip
\noindent
Recall that the geometric $r$-matrix $\rho$ corresponding to $(E, \kA)$ is given by the formula  (\ref{E:geomRMatrNodal}).
For any $(k, i) \in \Upsilon$ we have  $w_{(k, i)} \in \widehat{\lieM}$, given by the formula (\ref{E:dualelements}) with respect to the  $R$-basis of $\lieL$ fixed above. 
Then there exist uniquely determined $b_{(k, i)} \in \widehat{\lieW}^\circ$ and 
$h_{(k, i)} \in \lieL \cong \lieD$ such that $b_{(k,\, i)} = w_{(k,\, i)} + 
h_{(k,\, i)}$. It is not hard to see  that $h_{(k, i)} = 0$ for all $k \ne 0$. 
For $k = 0$, we have the following decompositions: 
$$
\widehat{\lieW}^\circ \ni \left\{
\begin{array}{ccl}
 (0, e_\imath^-) & = & (-e_\imath^-, 0) + (e_\imath^-, e_\imath^-) \\
  (-e_\imath^+, 0) & = & (-e_\imath^+, 0) + (0, 0) \\
\left(-\frac{1}{2} h_l, \frac{1}{2} h_l\right) & = & (-h_l, 0) + \left(\frac{1}{2} h_l, \frac{1}{2} h_l\right).
\end{array}
\right.
$$
All together, taking into account the formulae (\ref{E:dualelements}), (\ref{E:expansiongeomrmatrnodalpoint}) and (\ref{E:Orthogonality}), we obtain from (\ref{E:geomRMatrNodal})  the following explicit expression:
$$
r_\circ(x, y) =  \frac{y^m}{x^m - y^m} \sum\limits_{k = 0}^{m-1} \left(\frac{x}{y}\right)^k \gamma_k + \left(\sum\limits_{\imath = 1}^p e_\imath^- \otimes e_\imath^+ + \sum\limits_{l = 1}^r \frac{1}{2} h_l \otimes h_l\right),
$$
which coincides with the formula (\ref{E:RMatStandard}), as asserted.
\end{proof}

\smallskip
\noindent
We get the following corollary, which seems to be well-known to the experts (another, more direct proof,  can be found in \cite{AbedinMaximov}).

\begin{corollary}\label{C:StandradStructure} We have the following formula for the 
standard Lie bialgebra cobracket:
$$
\lieL \stackrel{\delta_\circ}\lar  \lieL \wedge \lieL, \quad f(z) \mapsto 
\bigl[f(x) \otimes 1 + 1 \otimes f(y), r_\circ(x,y)\bigr],
$$
where $r_\circ(x, y)$ is the \emph{standard} $r$-matrix given by (\ref{E:RMatStandard}). 
\end{corollary}

\begin{remark}\label{R:RemarkBDsetting}
Let $\lieg = \lien^{+} \dotplus \tilde\lieh \dotplus \lien^{-}$ be a fixed triangular decomposition of the  finite dimensional simple complex Lie algebra $\lieg$ corresponding to a  Dynkin diagram $\Gamma$. Then any $\phi \in \Aut(\Gamma)$ defines
an automorphism $\tilde\phi \in \Aut_{\CC}(\lieg)$. Let $\sigma \in \Aut_{\CC}(\lieg)$ be a \emph{Coxeter automorphism} corresponding  to  $\phi$ and $m$ be the order of $\sigma$; see \cite[Section 6]{BelavinDrinfeld} for an explicit description of $\sigma$. Then we have: 
 $\lieL := \lieL(\lieg, \sigma) \cong \lieL(\lieg, \tilde\phi)$; 
  see \cite[Proposition 8.1]{Kac}. An advantage  to use the Coxeter automorphism $\sigma$ to define twisted loop algebra is due to the fact that the fixed point Lie algebra 
$\left\{a \in \lieg \, \big| \, \sigma(a) =  a \right\}$  
  is abelian. In particular, the standard $r$-matrix (\ref{E:RMatStandard}) takes  the following shape:
  \begin{equation}\label{E:StandardFormKulish}
r_\circ(x, y) = \frac{\gamma_{0}}{2} + \frac{y^m}{x^m - y^m} \sum\limits_{k = 0}^{m-1} \left(\frac{x}{y}\right)^k \gamma_k   = \frac{\gamma_{0}}{2} + \frac{1}{\exp(w) -1} \sum\limits_{k = 0}^{m-1} \exp\left(\frac{k w}{m}\right) \gamma_j,
\end{equation}
where $\exp\left(\dfrac{w}{m}\right) = \dfrac{x}{y}$. For $\phi = \mathsf{id}$, this 
 solution was discovered for the first time  by Kulish (see \cite[formula (38)]{Kulish})  and generalized 
 by  Belavin and Drinfeld  (see \cite[Proposition 6.1]{BelavinDrinfeld}) for an arbitrary $\phi$.
\end{remark}

\begin{remark}\label{R:KPSST-Setting}
Let $\lieg = \lien^{+} \dotplus \tilde\lieh \dotplus \lien^{-}$ be again a fixed triangular decomposition of  $\lieg$, $\Phi_+$ be the  set of positive roots of $(\lieg, \lieh)$ and $\sigma = \mathsf{id}$. Then $\lieL = \lieL(\lieg, \sigma) = 
\lieg\bigl[z, z^{-1}\bigr]$ and the standard $r$-matrix (\ref{E:RMatStandard}) takes the following form: 
\begin{equation}\label{E:StandardFormJimbo}
r_{\circ}(x, y) = \dfrac{1}{2}\Bigl(\dfrac{x+y}{x-y} \gamma + \sum\limits_{\alpha \in \Phi_+} e_{-\alpha} \wedge e_{\alpha}\Bigr),
\end{equation}
which can be for instance found in \cite{KPSST}. It can be shown that the solution (\ref{E:StandardFormJimbo})
is equivalent (in the sense of (\ref{E:equiv1}) and (\ref{E:equiv2})) to the solution (\ref{E:StandardFormKulish}) (for the identity automorphism of the Dynkin diagram of $\lieg$); see for instance \cite{AbedinMaximov} for details. \hfill $\lozenge$
\end{remark}

\begin{theorem}\label{T:TwistsCYBEGeometry} For any skew-symmetric tensor $\ttau \in \wedge^2\lieL \subset (\lieg \otimes \lieg)\bigl[x, x^{-1}, y, y^{-1}\bigr]$ we put: 
\begin{equation}\label{E:StandardTwist}
\delta_\ttau = \delta_\circ + \partial_{\ttau} \quad \mbox{and} \quad 
r_{\ttau}(x, y) = r_\circ(x, y) + \ttau(x, y).
\end{equation}
Then $\lieL \stackrel{\delta_\ttau}\lar  \lieL \wedge \lieL$ is a Lie bialgebra cobracket if and only 
if $r_\ttau(x, y)$ is a solution of the classical Yang--Baxter equation (\ref{E:CYBE}). 
In this case, let $\lieM = \lieW_\ttau \dotplus \lieD$ be the corresponding Manin triple (see 
Theorem \ref{T:ManinTriplesLoopTwists}) 
and
$(E, \kA_\ttau)$ be the corresponding geometric  CYBE datum (see Proposition 
\ref{P:GeomDataFromTwists}). Then  the geometric $r$-matrix $\rho_\ttau$ of  $(E, \kA_\ttau)$ with respect to the trivialization, described at the end of Subsection \ref{SS:GeometrizationTwists}, coincides with $r_\ttau(x, y)$. 
\end{theorem}

\begin{proof}
By Proposition \ref{P:basicsonloops}, $\lieL^{\otimes 3}$ does not have any non-zero ad-invariant elements. Hence, Proposition \ref{P:TwistLieBialgStructures} implies that  $\delta_{\ttau}$ is a Lie bialgebra cobracket if and only if $\ttau$ satisfies the twist equation (\ref{E:TwistEquation}). On the other hand, since $r_\circ$ solves the CYBE, we can rewrite the CYBE for $r_\ttau$ as
$$[[\ttau, \ttau]] + \bigl[r_\circ^{12}, \ttau^{13} + \ttau^{23}\bigr] + \bigl[r_\circ^{13}, \ttau^{23} + \ttau^{21}\bigr] + 
\bigl[r_\circ^{23}, \ttau^{21} + \ttau^{31}\bigr] = 0.
$$ 
We have: $\bigl[r_\circ^{12}, \ttau^{13} + \ttau^{23}\bigr] = - (\delta_\circ \otimes \mathbbm{1})(\ttau)$. It follows that $r_\ttau$ solves the CYBE if and only if 
$\alt\bigl((\delta_\circ \otimes \mathbbm{1})(\ttau)\bigr)  = [[\ttau, \ttau]],
$
implying the first statement. 

\smallskip
\noindent
As it was explained in the proof of Proposition \ref{P:GeomDataFromTwists}, 
the Manin triple $\widehat{\lieM} = \widehat{\lieW}_\ttau \dotplus \lieD$ is isomorphic to the geometric Manin triple $\widetilde\lieA = 
\widehat\lieA_{\ttau} \dotplus \lieA$. Let $\tilde{r}_\ttau(x, y)$ be the trivialization of the geometric $r$-matrix $\rho_\ttau$ with respect to the trivialization $\lieA \cong \lieL$
introduced at the end of Subsection \ref{SS:GeometrizationTwists}. Then we get the geometric Lie bialgebra cobracket
$$
\lieL \stackrel{\delta}\lar \lieL \wedge \lieL, \; f(z) \mapsto \bigl[f(x) \otimes 1 + 1 \otimes f(y), \tilde{r}_\ttau(x, y)\bigr].
$$
On the other hand, Corollary \ref{C:StandradStructure} implies that 
$$
\delta_\ttau(f)  :=  \delta_\circ(f) + \bigl[f(x) \otimes 1 + 1 \otimes f(y), \ttau(x, y)\bigr] = \bigl[f(x) \otimes 1 + 1 \otimes f(y), r_\circ(x, y) + \ttau(x, y)\bigr].
$$
According to Proposition \ref{P:CompletedTwistsManinTriples} and  Theorem \ref{T:main}, both Lie bialgebra cobrackets $\delta$ and $\delta_\ttau$ are determined by the same Manin triple $\widehat{\lieM} = \widehat{\lieW}_\ttau \dotplus \lieD$. It follows that $\delta = \delta_\ttau$.
Since $\lieL^{\otimes 2}$ has no  non-zero ad-invariant elements (see Proposition \ref{P:basicsonloops}), we conclude that
$\tilde{r}_\ttau(x, y) = r_\circ(x, y) + \ttau(x, y) = r_\ttau(x, y)$, as asserted. 
\end{proof}

\subsection{On the theory of trigonometric solutions of CYBE}\label{SS:ReviewBDTheory}
Consider the setting of Remark \ref{R:RemarkBDsetting}.
Let $\lieg = \lien^{+} \dotplus \tilde\lieh \dotplus \lien^{-}$ be a  triangular decomposition of  $\lieg$,  $\Gamma$ be the Dynkin diagram of $\lieg$ and  $\phi \in \Aut(\Gamma)$. Let $\sigma \in \Aut_{\CC}(\lieg)$ be a \emph{Coxeter automorphism} corresponding  to  $\phi$, $m$ be the order of $\sigma$ and  $\lieL := \lieL(\lieg, \sigma)$. 
Recall that 
 $\lieg_{0} = \lieh$  is an abelian Lie algebra.  For  $1 \le k \le m-1$ and $\alpha \in \lieh^\ast$, let $
\lieg_{k}^\alpha := \bigl\{x \in \lieg_{k} \, \big| \, [h, x] = \alpha(h)x \; \mbox{\rm for all} \; h \in \lieh \bigr\}.
$
We put 
$$
\Lambda_k := \bigl\{\alpha \in \lieh^\ast \, \big| \, \lieg_{k}^\alpha \ne 0 \bigr\}
\quad \mbox{\rm and} \quad \Xi := \left\{(\alpha, k) \, \big| \, 1 \le k \le m - 1 \;  \mbox{and} \; \alpha \in \Lambda_k \right\}.
$$
Then we have a direct sum decomposition
\begin{equation}\label{E:SplittingLieg}
 \lieg =  \lieh \oplus
 \bigoplus_{(\alpha, k) \in \Xi} \lieg_{k}^{\alpha},
\end{equation}
and the vector  space 
$\lieg_{k}^\alpha$ is  one-dimensional for any $(\alpha, k) \in \Xi$. 

\smallskip
\noindent
The main  advantage to define the twisted loop algebra $\lieL$ corresponding to 
$\nu \in \Aut(\Gamma)$ 
using a  Coxeter automorphism (even for $\phi = \mathsf{id}$) is due to the following special structure of the set $\Pi$ of positive simple roots of $(\lieL, \lieh)$:
$
\Pi = \bigl\{(\alpha, 1) \, \big| \, \alpha \in \Lambda_1\bigr\}.
$
In particular, we have: $ \big|\Lambda_1\big| = r + 1= \dim_{\CC}(\lieh) + 1$ and the elements of $\Lambda_1$ are in a bijection with the vertices of the affine Dynkin diagram $\widehat{\Gamma}$ such that $\lieL \cong \lieG_{\widehat{\Gamma}}$ via the Gabber--Kac isomorphism (\ref{E:GabberKac}). 

\smallskip
\noindent
Recall  that a \emph{Belavin--Drinfeld triple} is a tuple $\bigl(\Gamma_1, \Gamma_2, \tau\bigr)$, where 
 $\Gamma_i \subsetneqq \Lambda_1$ for $i = 1, 2$ are subsets and 
$\Gamma_1 \stackrel{\tau}\lar \Gamma_2$ is a bijection satisfying the following conditions:
\begin{itemize}
\item $\kappa\bigl(\tau(\alpha), \tau(\beta)\bigr) = \kappa(\alpha, \beta)$ for all $\alpha, \beta \in \Gamma_1$;
\item for any $\alpha \in \Gamma_1$ there exists $l = l(\alpha) \in \NN$ such that  $\alpha, \dots, \tau^{l-1}(\alpha) \in \Gamma_1$ but $\tau^l(\alpha) \notin \Gamma_1$. 
\end{itemize}

\smallskip
\noindent
For $i = 1, 2$,  let $\lien_i$ be the Lie subalgebra of $\lieg$ generated by the vector subspace $\oplus_{\alpha \in \Gamma_i} \lieg_1^{\alpha}$.
Then  $\lien_i$ is isomorphic to the positive part
of the semisimple Lie algebra defined by the Dynkin diagram $\Gamma_i$
and  we  have  a direct sum decomposition
\begin{equation}\label{E:SplittingLiea}
\lien_i = \bigoplus_{(\alpha, k) \in \Xi_i} \lieg_{k}^{\alpha}.
\end{equation}
for an appropriate subset  $\Xi_i  \subset \Xi$.
Fixing  non-zero elements in $\left(\lieg_1^\alpha\right)_{\alpha \in \Lambda_1}$, one can extend the bijection $\Gamma_1 \stackrel{\tau}\lar \Gamma_2$  to an isomorphism of Lie algebras $\lien_1 \stackrel{\tilde\tau}\lar  \lien_2$.  

\smallskip
\noindent
Let 
$\lieg \stackrel{\vartheta}\lar \lieg$ be a linear map defined as the composition
$
\lieg \stackrel{\pi}\rightarrowdbl \lien_1 \stackrel{\tilde\tau}\lar \lien_2 \stackrel{\imath}\longhookrightarrow \lieg,
$
where $\pi$ and $\imath$ are the canonical projection and embedding with respect to the direct sum decompositions (\ref{E:SplittingLieg})  and (\ref{E:SplittingLiea}).
Then $\vartheta$ is nilpotent and $\vartheta(\lieg_k) \subset \lieg_k$ for all $1 \le k \le m-1$. 
Let 
$
\psi  = \dfrac{\vartheta}{\mathbbm{1} - \vartheta} = \sum\limits_{l = 1}^\infty \vartheta^l.
$
It follows that $\psi(\lieg_k) \subset \lieg_k$ for all $1 \le k \le m-1$ as well.

\smallskip
\noindent
For any  Belavin--Drinfeld triple $\bigl(\Gamma_1, \Gamma_2, \tau\bigr)$,  the system of linear equations 
\begin{equation}\label{E:Consistency}
 \bigl(\tau(\alpha) \otimes \mathbbm{1} + \mathbbm{1} \otimes \alpha\bigr)\Bigl(\ttaus+ \frac{\gamma_{0}}{2}\Bigr) = 0 \quad \mbox{for all} \; \alpha \in \Gamma_1
\end{equation}
for $\ttaus \in \lieh \wedge \lieh$ is consistent; see \cite[Lemma 6.8]{BelavinDrinfeld}. According to 
\cite[Theorem 6.1]{BelavinDrinfeld}, trigonometric solutions of (\ref{E:CYBE1}) are parametrized by \emph{Belavin--Drinfeld quadruples} 
$Q = \bigl(\bigl(\Gamma_1, \Gamma_2, \tau\bigr), \ttaus\bigr)$,  where $\bigl(\Gamma_1, \Gamma_2, \tau\bigr)$ is a Belavin--Drinfeld triple and $\ttaus \in \lieh \wedge \lieh$ satisfies  (\ref{E:Consistency}). The solution of (\ref{E:CYBE1}) corresponding to $Q$ is given by the following formula: 
\begin{equation}\label{E:BDtrigonometricformula}
\varrho_Q(w) = \varrho_\circ(w) + \ttaus +
\sum\limits_{j = 1}^{m-1} \left(- \exp\left(\frac{jw}{m}\right)(\psi \otimes \mathbbm{1}) \gamma_j + \exp\left(-\frac{jw}{m}\right) (\mathbbm{1} \otimes \psi) \gamma_{-j}\right),
\end{equation}
where $\varrho_\circ(w)$ is given by (\ref{E:StandardFormKulish}).

\smallskip
\noindent
Let us rewrite the formula (\ref{E:BDtrigonometricformula}) in different terms. 
Choose elements $\bigl(g_{(\alpha, k)} \in \lieg_{k}^\alpha\bigr)_{(\alpha, k) \in \Xi}$ such that  
$\kappa\bigl(g_{(\alpha, k)}, g_{(\beta, l)}\bigr) = \delta_{\alpha+\beta, 0} \delta_{k+l, 0}$. Then for any $1 \le k \le m-1$ we have:
$
\gamma_{\pm k} = \sum\limits_{\alpha \in \Xi_{\pm k}} g_{(\pm \alpha, \pm k)} \otimes g_{(\mp \alpha, \mp k)}.
$
It follows that 
$$
\left\{
\begin{array}{l}
(\psi \otimes \mathbbm{1})(\gamma_k) \; \,  = 
\sum\limits_{l = 1}^\infty  \sum\limits_{\alpha \in \Xi_k} 
\vartheta^l\bigl(g_{(\alpha, k)}) \otimes g_{(-\alpha, -k)} \\
(\mathbbm{1} \otimes \psi)(\gamma_{-k}) = 
\sum\limits_{l = 1}^\infty \sum\limits_{\alpha \in \Xi_k} 
 g_{(-\alpha, -k)} \otimes \vartheta^l\bigl(g_{(\alpha, k)}).
\end{array}
\right.
$$
Consider the following expression
\begin{equation}\label{E:twist}
t_Q(x, y) := \ttaus + \sum\limits_{l = 1}^\infty  \sum\limits_{(\alpha, k) \in \Xi} 
\bigl(-\vartheta^l\bigl(g_{(\alpha, k)}\bigr) \otimes g_{(-\alpha, -k)} \left(\frac{x}{y}\right)^k + g_{(-\alpha, -k)} \otimes \vartheta^l\bigl(g_{(\alpha, k)}\bigr) \left(\frac{y}{x}\right)^k\Bigr).
\end{equation}
Then we have:
\begin{equation}\label{E:TwistsCoxCase}
r_Q(x, y) := r_\circ(x, y) +  t_Q(x, y) = \varrho_Q(w)
\end{equation}
where $x, y$ and $w$ are related by the formula $\dfrac{x}{y} = \exp\left(\dfrac{w}{m}\right)$.  In other words, $r_Q$ is the solution of the classical Yang--Baxter equation (\ref{E:CYBE}) corresponding to the Belavin--Drinfeld quadruple $Q = \bigl((\Gamma_1, \Gamma_2, \tau), \ttaus\bigr)$. 

\begin{corollary}\label{R:BDasTwist}
For $az^k, bz^l \in \overline{\lieL}$ we put:  
$az^k \wedge b z^l := a x^k \otimes b y^l - b x^l \otimes a y^k \in \overline{\lieL} \wedge \overline{\lieL}$. Then  $t_Q$ given by (\ref{E:twist}) can be viewed as an element of 
$\wedge^2(\lieL)$.
As a consequence, the trigonometric solution $r_Q(x, y)$ is of the form (\ref{E:StandardTwist}) and can be realized as the geometric $r$-matrix defined by an appropriate geometric CYBE datum $(E, \kA)$, where $E$ is a nodal Weierstra\ss{} curve. 
 \end{corollary}

\smallskip
\noindent
A proof of the following result is analogous to \cite{BelavinDrinfeld2} and 
\cite[Theorem 19]{KPSST}.

\begin{proposition}\label{P:TwistIsTrigonometric}
Let
$
r(x, y) = \dfrac{y^m}{x^m - y^m} \sum\limits_{j = 0}^{m-1} \left(\dfrac{x}{y}\right)^j \gamma_j + g(x, y)
$
be a solution of (\ref{E:CYBE}), where $\CC^2 \stackrel{g}\lar   \lieg \otimes \lieg$ is a holomorphic  function. Then $r$ is equivalent (in the sense of Subsection \ref{SS:SurveyCYBE}) to a trigonometric solution of (\ref{E:CYBE1}).  
\end{proposition}
\begin{proof} For $a, b, c, d \in \lieg$ put:
$
\lconvol a \otimes b , c \otimes d \rconvol:= [a, c] \otimes [b, d].
$
Proceeding similarly to \cite{BelavinDrinfeld2},
one can deduce from (\ref{E:CYBE}) the following identities:
$$
\left\{
\begin{array}{l}
\lconvol r(x, y), r(x,y)\rconvol + \left[r(x, y), 1 \otimes f(y)\right] + \dfrac{y}{m} \dfrac{\partial r}{\partial y}(x, y) =  0 \\
\lconvol r(x, y), r(x,y)\rconvol - \left[r(x, y), f(x)  \otimes 1\right] - \dfrac{x}{m} \dfrac{\partial r}{\partial x}(x, y) =  0,
\end{array}
\right.
$$
where $f(z) := \lfloor g(z, z) + \frac{1}{m} \sum\limits_{k= 1}^{m-1} k \gamma_k \rfloor$ (here, $\lfloor a \otimes b \rfloor = [a, b]$ for $a, b \in \lieg$). 
It follows that
$$
\bigl[f(x) \otimes 1 + 1 \otimes f(y), r(x, y)\bigr] = \dfrac{x}{m} \dfrac{\partial r}{\partial x}(x, y) + \dfrac{y}{m} \dfrac{\partial r}{\partial y}(x, y).
$$
Let 
$\tilde{r}(u, v) := r\left(\exp\left(\dfrac{u}{m}\right),  
\exp\left(\dfrac{v}{m}\right)\right)$ and
$h(u) := f\left(\exp\left(\dfrac{u}{m}\right)\right)$. Then  
$\CC^2 \stackrel{\tilde{r}}\lar  \lieg \otimes \lieg$ is a meromorphic solution of (\ref{E:CYBE}) equivalent to $r$ (whose set of poles is given by the union of lines $\left\{(u, v) \in \CC^2 \, \big| \, u - v = 2\pi i k \right\}$ for $k \in \ZZ$), $\CC \stackrel{h}\lar \lieg$ is a holomorphic function and 
$$
\bigl[h(u) \otimes 1 + 1 \otimes h(v), \tilde{r}(u, v)\bigr] = \left(
\dfrac{\partial}{\partial u} + \dfrac{\partial}{\partial v}\right) \tilde{r}(u, v).
$$
Let $(\CC, 0) \stackrel{\varphi}\lar \End_{\CC}(\lieg)$ the germ of a holomorphic function satisfying the differential equation 
 $\dot\varphi = \ad_h \circ \varphi$ and the initial condition $\varphi(0) = \mathbbm{1}$, where  $\CC \stackrel{\ad_h}\lar \End_{\CC}(\lieg)$ is given by the rule $\bigl(\ad_h(u)\bigr)(\xi) = \bigl[h(u), \xi\bigr]$ for $u \in \CC$ and $\xi \in \lieg$. 
Then $\varphi$ can be extended to a holomorphic function  on the entire complex plane (see \cite[Theorem 19]{KPSST}).  
The initial condition $\varphi(0) = \mathbbm{1}$ and the continuity of $\varphi$ imply that $\det\bigl(\varphi(u)\bigr) = 1$ for all $u \in \CC$ (see the proof of 
\cite[Proposition 2.2]{BelavinDrinfeld}). Hence, we have  an entire function 
$\CC \stackrel{\varphi}\lar \Aut_{\CC}(\lieg)$. 
Let $$
\tilde\rho(u, v) := \bigl(\varphi(u)^{-1} \otimes \varphi(v)^{-1}\bigr) \tilde{r}(u, v)
$$
It follows that $\left(
\dfrac{\partial}{\partial u} + \dfrac{\partial}{\partial v}\right) \tilde\rho(u, v) = 0$, i.e. $\tilde{\rho}(u, v) = \varrho(u-v)$ for some meromorphic solution
$\CC \stackrel{\varrho}\lar \lieg \otimes \lieg$ of (\ref{E:CYBE1}), whose set of poles is $2 \pi i\ZZ$. It follows that $\varrho$ is a trigonometric solution of 
(\ref{E:CYBE1}).
\end{proof}

\subsection{Concluding remarks on the geometrization of trigonometric solutions}
Let $(E, \kA)$ be a geometric CYBE datum as in Proposition \ref{P:GeomDataFromTwists}. Within that construction, we additionally made the following choices.
\begin{itemize}
\item $\PP^{1} \stackrel{\nu}\lar E$ is a fixed normalization map. We have  fixed homogeneous coordinates $(w_+: w_-)$ on $\PP^1$ such that
$\nu^{-1}(s) = \left\{s_+, s_-\right\}$, where $s_+ = (0: 1)$ and $s_-  = (1: 0)$. 
\item We have an algebra isomorphism $\Gamma(U, \kO) \cong  \CC\bigl[u, u^{-1}\bigr]$ as well as an $\Gamma(U, \kO)$-$\CC\bigl[u, u^{-1}\bigr]$-equivariant  isomorphism of Lie algebras 
$\lieA \cong\lieL = \bigoplus\limits_{k \in \ZZ} \lieg_k x^k,$ where $u = \dfrac{w_+}{w_-} = x^m$. We also put: $\omega = \dfrac{du}{u}$. 
\end{itemize}
Let $p:= \nu\bigl((1: 1)\bigr) \in E$. Equipping $U \subset E$ with the usual group law 
(on the set of smooth point of a singular Weierstra\ss{} curve)  with $p$ being neutral element,  the map $\CC^\ast \rightarrow U, t \mapsto \nu\bigl(1:t\bigr)$ becomes a group isomorphism. 

\smallskip
\noindent
 Consider the algebra homomorphism
$\CC\bigl[u, u^{-1}\bigr] \rightarrow \CC\llbracket z\rrbracket, u \mapsto \exp(z).$ 
As $(\exp(z) -1) \in \CC\llbracket z\rrbracket$ is a local parameter, we get an induced algebra isomorphism $\widehat{O}_p \rightarrow \CC\llbracket z\rrbracket$.
In these terms, the differential form $\widehat{\omega}_p$ gets identified with 
 $dz$. 
 Moreover, the linear map
$
\widehat{\lieA}_p \rightarrow \lieg\llbracket z\rrbracket, \; a x^k \mapsto a \exp\left(\dfrac{z}{m}k\right)
$
is a ($\widehat{O}_p$--$\CC\llbracket z\rrbracket$)--equivariant isomorphism of Lie algebras. Consider the \'etale covering $\CC^\ast = \widetilde{U} \rightarrow U = \CC^\ast$  of degree $m$, given by the formula $x \mapsto x^m = u$. It extends to a finite morphism
$\PP^1 \stackrel{\tilde{\pi}}\lar  \PP^1, (w_+: w_-) \mapsto (w_+^m: w_-^m)$. Since  $\tilde{\pi}(s_\pm) = s_\pm$ and (\ref{E:Bicartesian}) is a pulldown diagram, 
we obtain an  induced finite morphism $E \stackrel{\pi}\lar E$. Let $\widetilde\kA = \pi^\ast(\kA)$. 
 Then we have the following commutative diagram. 
\begin{equation}
\begin{array}{c}
\xymatrix{
\Gamma\bigl((E \times \widetilde{U}) \setminus \widetilde{\Sigma}, \widetilde\kA \boxtimes
\widetilde\kA\bigr) \ar@{^{(}->}[rr] & & \Gamma\bigl((\widetilde{U} \times \widetilde{U}) \setminus \widetilde{\Sigma}, \widetilde\kA  \boxtimes
\widetilde\kA\bigr) \\
\Gamma\bigl((E \times U) \setminus \Sigma, \kA \boxtimes
\kA\bigr) \ar@{^{(}->}[rr] \ar@{^{(}->}[u] \ar@{_{(}->}[d] & & \Gamma\bigl((U \times U) \setminus \Sigma, \kA \boxtimes
\kA\bigr) \ar@{_{(}->}[u] \ar@{^{(}->}[d] \\
\overline{\lieA_{(p)} \otimes \widehat{\lieA}_p} \ar@{^{(}->}[rr] 
\ar@{_{(}->}[rd] & & 
\overline{\lieA_{(p)}^\circ \otimes \widehat{\lieA}_p} \ar@{^{(}->}[ld] \\
& \overline{\widetilde{\lieA}_{(p)} \otimes \widehat{\lieA}_p} & 
}
\end{array}
\end{equation}
Then   $\varrho \in \Gamma\bigl((E \times U) \setminus \Sigma, \kA \boxtimes
\kA\bigr)$, $\rho \in \Gamma\bigl((U \times U) \setminus \Sigma, \kA \boxtimes
\kA\bigr)$, $\tilde\rho \in \Gamma\bigl((\widetilde{U} \times \widetilde{U}) \setminus \widetilde{\Sigma}, \widetilde\kA  \boxtimes
\widetilde\kA\bigr)$ and $\bar\rho \in \overline{\widetilde{\lieA}_{(p)} \otimes \widehat{\lieA}_p}$ are identified with each other under the corresponding maps.
Taking  the trivialization $\Gamma(\widetilde{U}, \kA) \cong \overline{\lieL} = \lieg\bigl[x, x^{-1}\bigr]$, we get a solution of (\ref{E:CYBE}) 
$$
r(x, y) = r_\circ(x, y) + \ttau(x, y) =  \left(\frac{y^m}{x^m - y^m} \sum\limits_{k = 0}^{m-1} \left(\frac{x}{y}\right)^k \gamma_k\right)  + \frac{\gamma_{0}}{2} + \ttau(x, y),
$$
where $\ttau \in \wedge^2\lieL \subset (\lieg \otimes \lieg)\bigl[x, x^{-1}, y, y^{-1}\bigr]$.
Making the substitutions $x = \exp\left(\dfrac{z}{m}\right)$ and $y = 
\exp\left(\dfrac{w}{m}\right)$, we obtain the solution
\begin{equation}\label{E:TrigonomBDForm}
r(z, w) = \left(\frac{1}{\exp(z-w) -1}\sum\limits_{k = 0}^{m-1} \exp\left( \frac{z-w}{m} k\right) \gamma_k\right)  + \frac{\gamma_{0}}{2} + \ttau\left(\exp\left(\dfrac{z}{m}\right), \exp\left(\dfrac{w}{m}\right)\right).
\end{equation}
The corresponding element of $\bigl(\lieg\llbrace z\rrbrace \otimes \lieg\bigr)\llbracket w\rrbracket$
viewed as a solution of (\ref{E:CYBEformali}), 
coincides with the image of $\bar\rho$ under the isomorphism $
\overline{\widetilde{\lieA}_{(p)} \otimes \widehat{\lieA}_p} \cong 
 \bigl(\lieg\llbrace z\rrbrace \otimes \lieg\bigr)\llbracket w\rrbracket$.

\begin{remark} The set of Manin triples $\lieL \times \lieL^\ddagger = \lieD \, \dotplus \,
\lieW$ from Theorem \ref{T:ManinTriplesLoopTwists} admits a natural involution $\lieW \mapsto \lieW^\ddagger$ induced by the Lie algebra automorphism 
\begin{equation}\label{E:InvolutionManinTriples}
\lieL \times \lieL^\ddagger \lar \lieL \times \lieL^\ddagger, \; (f, g) 
\mapsto(g^\ddagger, f^\ddagger)
\end{equation}
Note that  (\ref{E:InvolutionManinTriples})  is an involution which fixes   the Lie subalgebra $\lieD$.  Let $(E, \kA)$ end
 $(E, \kA^\ddagger)$ be the geometric CYBE data from Proposition 
\ref{P:GeomDataFromTwists},  
 corresponding to $\lieW$ and $\lieW^\ddagger$, respectively.  It is not hard to see that
 $\kA^\ddagger \cong \imath^\ast(\kA)$, where $E \stackrel{\imath}\lar E$ is the involution, induced by the involution 
 $\PP^1 \rightarrow \PP^1, \, (w_+: w_-) \mapsto (w_-: w_+)$. It is clear that $\imath(p) = p$. 
 Moreover, the solutions $r(z, w)$ and $r^\ddagger(z, w)$ corresponding to $\lieW$ and $\lieW^\ddagger$ and given by  (\ref{E:TrigonomBDForm})  are related by the formula:
 $r^\ddagger(z, w) = r(-z, -w)$.
\hfill $\lozenge$
\end{remark}

\smallskip
\noindent
\textbf{Summary}.
Let $\ttau \in \wedge^2\lieL$
be a twist of the standard Lie bialgebra cobracket $\lieL \stackrel{\delta_\circ}\lar  \lieL \wedge \lieL$. Then $r_\ttau(x, y) = r_\circ(x, y) + \ttau(x, y)$ is a solution of (\ref{E:CYBE}), which is equivalent to a trigonometric solution $\varrho_\ttau$ of (\ref{E:CYBE1}) with respect to the equivalence relations (\ref{E:equiv1}) and (\ref{E:equiv2}). On the other hand, any 
trigonometric solution of (\ref{E:CYBE1}) is equivalent to a solution $r_\ttau(x, y)$ for some 
$\ttau \in \wedge^2\lieL$. Moreover, it was shown in  \cite{AbedinMaximov}  that 
for two twists $\ttau', \ttau'' \in \wedge^2\lieL$ of $\delta_\circ$  the corresponding Lie bialgebras $(\lieL, \delta_{\ttau'})$ and $(\lieL, \delta_{\ttau''})$ are related  by an  $R$--linear automorphism of $\lieL$   if and only if the solutions $\varrho_{\ttau'}$ and $\varrho_{\ttau''}$ are equivalent.

\begin{remark}
The presented way of geometrization of twists of the standard Lie bialgebra structure can be viewed as an alternative approach to classification of trigonometric solutions of (\ref{E:CYBE1}). On the other hand,  methods developed in this work are adaptable for a study of analogues of trigonometric solutions of (\ref{E:CYBE1})  for simple Lie algebras defined over algebraically non-closed fields like $\mathbb{R}$ (what is interesting  because  of applications to classical integrable systems \cite{BabelonBernardTalon, ReymanST}) or $\CC\llbrace h\rrbrace$ (motivated by the  problem of quantization of Lie bialgebras; see \cite{EtingofKazhdan, KKPS, KPS}).
We are going to return to these questions in the  future.$\lozenge$
\end{remark}

\section{Explicit computations}\label{S:ExplicitComputations}
\subsection{On explicit geometrization of certain solutions for $\mathfrak{sl}_n(\CC)$}
Let $\kP$ be a simple vector bundle on a Weierstra\ss{} curve $E$ (i.e.~$\End_E(\kP) = \CC$) of rank $n$ and degree $d$. 
Then 
  $\mathsf{gcd}(n, d) = 1$ and for any other simple vector bundle $\kQ$ with the same rank and degree  there exists a line bundle
 $\kL \in \Pic^0(E)$ such that $\kQ \cong \kP \otimes \kL$. Conversely,  for any $(n,d) \in \NN \times \ZZ$ satisfying the condition $\gcd(n, d) = 1$, there exists a simple vector bundle of rank $n$ and degree $d$ on $E$; see \cite{Atiyah, Burban1, BodnarchukDrozd} for the case when $E$ is elliptic, nodal and cuspidal, respectively. In what follows, we put $c := n - d$. 

 \smallskip
 \noindent
 Let $\kA  = \mbox{\it Ad}_E(\kP)$ be the sheaf of Lie algebras on $E$ given by the short exact sequence
\begin{equation}
0 \lar \kA \lar \mbox{\it End}_E(\kP) \stackrel{\mathsf{tr}}\lar \kO \lar 0.
\end{equation}
From what was said above we see  that  $\kA = \kA_{(c,d)}$ does not depend (up to an automorphism) on the particular choice of simple vector bundle $\kP$ and is uniquely determined
by the pair $(c, d)$. For any $p \in E$ we have: 
$\kA\big|_{p} \cong \lieg = \mathfrak{sl}_n(\CC)$. Simplicity of $\kP$ implies 
that $H^0(E, \kA) = 0 = H^1(E, \kA)$. It follows that  the pair $(E, \kA)$ is a geometric CYBE datum.

\smallskip
\noindent
Let $
K = K_{(c,d)} :=  \left(
\begin{array}{cc}
0 & I_d \\
I_c & 0
\end{array}
\right)
$
 and $T = T_{(c, d)}(u_-) = \left(
\begin{array}{cc}
I_c & 0  \\
0 & u_-^{-1} I_d
\end{array}
\right)
$, where $c = n-d$. We put:
$
\lied_{(c, d)} := \bigl\{\bigl(a, \mathsf{Ad}_{K}(a)\bigr) \, \big| \, a \in \lieg \bigr\}
$ (where $\mathsf{Ad}_{K}(a) := K a K^{-1}$) and $$\widehat{\lieW}_{(c, d)}^{\mathrm{trg}} = \bigl(\mathbbm{1} \times \mathsf{Ad}_{T}\bigr)\Bigl(\bigl(u_+ \lieg\llbracket u_{+}\rrbracket \times \{0\}\bigr)  + \bigl(\{0\} \times u_{-} \lieg\llbracket u_{-}\rrbracket\bigr) + \lied_{(c, d)}\Bigr) \subseteq \widehat{\lieM} = \widehat{\lieL}_+ \times \widehat{\lieL}_-,$$
where $\widehat{\lieL}_\pm = \lieg\llbrace u_\pm\rrbrace$. 

\begin{theorem}\label{T:CremmerGervaisGeometry} Let $E$ be a nodal Weierstra\ss{} curve, $s$ be its singular point and $\kA = 
\kA_{(c, d)}$
be a sheaf of Lie algebras attached to the pair $(c, d)$, where $c, d \in \NN$ are coprime. Then the Manin triple
$\widetilde\lieA_s = \widehat\lieA_s \dotplus \lieA_{(s)}$ is isomorphic to the Manin triple
 $\widehat{\lieM} = 
\widehat{\lieW}_{(c, d)}^{\mathrm{trg}} \dotplus \lieD$.
\end{theorem}

\begin{proof} Let us first  recall our notation and give an explicit description of the sheaf $\kA$. 
We choose homogeneous coordinates $(w_+: w_-)$ on $\PP^1$ and view them  as global sections:
$w_\pm \in \Gamma\bigl(\PP^1, \kO_{\PP^1}(1)\bigr)$. Let $s_\pm \in \PP^1$ be the point of vanishing of $w_\pm$, i.e. $s_+ = (0: 1)$ and $s_- = (1: 0)$. We put: 
$U_\pm := \PP^1 \setminus \{s_\pm\}$, $U = U_+ \cap U_-$  and $u_\pm := \dfrac{w_\pm}{w_\mp}$. It is clear that
$s_\pm \in U_\pm$ and that the rational function $u_\pm$ is a local parameter at $s_\pm$. We put:
$L_\pm := \Gamma\bigl(U_\pm, \kO_{\PP^1}\bigr) \cong \CC[u_\pm]$. Let $\widehat{O}_{\pm}$ be the completion of the stalk of $\kO_{\PP^1}$ at $s_\pm$ and $\widehat{Q}_\pm$ be the corresponding quotient field. Then we have:  
$\widehat{O}_{\pm} \cong \CC\llbracket u_\pm\rrbracket$ and $\widehat{Q}_{\pm} \cong \CC\llbrace u_\pm\rrbrace$ Finally, let
$R:= \Gamma\bigl(U, \kO_{\PP^1}\bigr) \cong \CC\bigl[u_\pm, u_\pm^{-1}\bigr] = 
\CC\bigl[u, u^{-1}\bigr],$
where $u = u_+ = u_-^{-1}$. We fix  the following trivializations:
\begin{equation}\label{E:TrivializationKeyChoice}
\Gamma\bigl(U_\pm, \kO_{\PP^1}(1)\bigr) \stackrel{\xi_\pm}\lar L_\pm, \; f \mapsto \dfrac{f}{w_\mp\big|_{U_\pm}}.
\end{equation}
As a consequence, for any $c, d \in \NN_0$ and $\kG  = \kG_{(c, d)}:= \kO_{\PP^1}^{\oplus c} \oplus \bigl(\kO_{\PP^1}(1)\bigr)^{\oplus d}$ we have the induced trivializations
$\Gamma\bigl(U_\pm, \kG\bigr) \stackrel{\xi_\pm^\kG}\lar L_\pm^{\oplus n}$, where $n = c + d$.
Let $\kB = \kB_{(c, d)} := \Ad(\kG)$. Then $\xi^\kG_\pm$ induces  trivializations 
$\Gamma\bigl(U_\pm, \kB\bigr) \stackrel{\xi_\pm^\kB}\lar \lieg[u_\pm]$. Let $\widehat{\kB}_{\pm}$ be the completion of the stalk of $\kB$ at $s_{\pm}$, $\widetilde{\kB}_{\pm}$ its rational envelope and $\kB\big|_{s_\pm}$ the fiber of $\kB$ over $s_\pm$.
Then we get induced isomorphisms 
$$
\widehat{\kB}_\pm \stackrel{\widehat{\xi}^\kB_\pm}\lar  \lieg\llbracket u_\pm \rrbracket, \, 
\widetilde{\kB}_\pm \stackrel{\widetilde{\xi}^\kB_\pm}\lar  \lieg\llbrace u_\pm \rrbrace\quad
\mbox{and}\quad \kB\big|_{s_\pm} \stackrel{\bar{\xi}^\kB_\pm}\lar \lieg.
$$ 
We define a nodal Weierstra\ss{} curve $E$ via the pushout diagram (\ref{E:Bicartesian}). We recall now the description of the sheaf $\kA$ given in \cite[Proposition 3.3]{BurbanGalinatStolin} (see also 
\cite[Section 5.1.2]{BK4}). 
Consider the embedding of Lie algebras
$
\lieg \stackrel{\tilde{\theta}_{(c,d)}}\lar \lieg \times \lieg, \; a \mapsto \bigl(a, \mathsf{Ad}_{K}(a)\bigr). 
$
Then  $\kA$ is defined via the following pullback diagram in the category $\Coh(E)$:
\begin{equation}\label{E:SheafA}
\begin{array}{c}
\xymatrix{
\kA \ar[r] \ar[d] & \lieg \ar@{^{(}->}[d]^-{\tilde\theta_{(c, d)}}\\
\nu_*(\kB) \ar[r]^-{\bar{\xi}} & \lieg \times \lieg
}
\end{array}
\end{equation}
where we view $\lieg$ and $\lieg \times \lieg$ as skyscraper sheaves supported at  $s$ and $\bar\xi$ is the composition 
$$
\nu_*(\kB)  \stackrel{\ev}\lar \kB\Big|_{s_+} \times \kB\Big|_{s_-} \xrightarrow{\bar{\xi}^\kB_+ \times \bar{\xi}^\kB_-} \lieg \times \lieg.
$$
In the notation of Theorem \ref{T:keyonTF}, $\bigl(\kB, \lieg, (\mathbbm{1}, \mathsf{Ad}_K)\bigr)$ is a triple corresponding to $\kA$. Let $\widehat{O}$ be the completion of the stalk of $\kO_E$ at $s$ and
$\widehat{Q}$ be the corresponding total ring of fractions. Then we have:
$
\Gamma\bigl(U, \kO_E\bigr) \cong \Gamma\bigl(\nu^{-1}(U), \kO_{\PP^1}\bigr) \cong \CC\bigl[u, u^{-1}\bigr]$, $\widehat{O} \cong \CC\llbracket u_+, u_-\rrbracket/(u_+ u_-)$ and $\widehat{Q} \cong \widehat{Q}_+ \times \widehat{Q}_- = \CC\llbrace u_+\rrbrace \times \CC\llbrace u_-\rrbrace$.

\smallskip
\noindent 
From (\ref{E:SheafA}) we get  the following commutative diagram of Lie algebras:
\begin{equation}\label{E:ComputingSheafA}
\begin{array}{c}
\xymatrix{
\widehat{\kA}_s \ar[r] \ar@{_{(}->}[d] & \lieg \ar@{^{(}->}[d]^-{\tilde\theta_{(c, d)}}\\
\widehat{\kB}_+ \times \widehat{\kB}_- \ar@{->>}[r]  \ar[d]_-{\widehat{\xi}_+^{\kB} \times 
\widehat{\xi}_-^{\kB}} & \kB\Big|_{s_+} \times \kB\Big|_{s_-}
 \ar[d]^-{\bar{\xi}_+^\kB \times 
\bar{\xi}_-^\kB}\\
 \lieg\llbracket u_+ \rrbracket \times \lieg\llbracket u_- \rrbracket \ar@{->>}[r] & \lieg \times \lieg.
}
\end{array}
\end{equation}
It follows that the image   of $\widehat{\kA}_s$ in  $\lieg\llbracket u_+\rrbracket \times \lieg\llbracket u_-\rrbracket$ under the composition of two left vertical maps in (\ref{E:ComputingSheafA})
 is the Lie algebra $\bigl(u_+ \lieg[u_{+}] \times \{0\}\bigr)  + \bigl(\{0\} \times u_{-} \lieg[u_{-}]\bigr) + \lied_{(c, d)}$. Passing to the rational hulls, we end up with  the embedding of Lie algebras
$$
\widehat{\kA}_s \longhookrightarrow \widetilde{\kA}_s \stackrel{\cong}\lar \widetilde{\kB}_+ \times \widetilde{\kB}_- \xrightarrow{\widetilde{\xi}_+^\kB \times 
\widetilde{\xi}_-^\kB} \lieg\llbrace u_+\rrbrace \times \lieg\llbrace u_-\rrbrace.
$$
On the other hand, the trivialization  $\Gamma\bigl(U_+, \kB\bigr) \stackrel{\xi_+^\kB}\lar \lieg[u]$
restricts to an isomorphism $\Gamma(U, \kB) \stackrel{\xi_\circ^\kB}\lar  \lieg\bigl[u, u^{-1}\bigr]$
and  induces 
the   isomorphisms of Lie algebras $\Gamma(U, \kA) \stackrel{\xi}\lar \lieg \bigl[u, u^{-1}\bigr]$ given as the composition 
$
\Gamma(U, \kA) \stackrel{\nu^\ast}\lar \Gamma(U, \kB) \stackrel{\xi_\circ^\kB}\lar  \lieg\bigl[u, u^{-1}\bigr].
$
We get
the induced   isomorphism  $\widetilde{\kB}_+ \times \widetilde{\kB}_- 
\xrightarrow{\widetilde{\xi}_+^\kB \times \breve{\xi}^\kB_+} \lieg\llbrace u_+\rrbrace \times \lieg\llbrace u_-\rrbrace$ as well 
as the following commutative diagram:
$$
\xymatrix{
\Gamma(U, \kA) \ar@{_{(}->}[d] \ar[r]^-{\nu^\ast} & \Gamma(U, \kB) \ar[r]^-{\xi_\circ^\kB} \ar@{^{(}->}[d] & \lieg\bigl[u, u^{-1}\bigr] \ar@{^{(}->}[d]\\
\widetilde{\kA}_s \ar[r] & \widetilde{\kB}_+ \times \widetilde{\kB}_- 
\ar[r]^-{\widetilde{\xi}_+^\kB \times \breve{\xi}^\kB_+}
& \lieg\llbrace u_+\rrbrace \times \lieg\llbrace u_-\rrbrace.
}
$$
It follows that the image of $\Gamma(U, \kA)$
under the embedding
$$
\Gamma(U, \kA) \longhookrightarrow \widetilde{\kA}_s \longhookrightarrow \widetilde{\kB}_+ \times \widetilde{\kB}_- \xrightarrow{\widetilde{\xi}^\kB_+ \times 
\breve{\xi}^\kB_+} \lieg\llbrace u_+\rrbrace \times \lieg\llbrace u_-\rrbrace
$$
 is the Lie algebra $\lieD = \bigl\{(a u_+^n, a u_-^n) \, \big| \, a \in \lieg, n \in \ZZ \bigr\}$.
 
\smallskip
\noindent
The formal trivializations  $\widetilde{\xi}^\kB_-$ and $\breve{\xi}^\kB_+$ are related by the following 
commutative diagram
$$
\xymatrix{
& \widetilde{\kB}_-  \ar[rd]^-{\breve{\xi}^\kB_+} \ar[ld]_-{\widetilde{\xi}_-^\kB} &  \\
\lieg\llbrace u_-\rrbrace \ar[rr]^-{\mathsf{Ad}_T} & & \lieg\llbrace u_-\rrbrace. 
}
$$
It follows that the image of $\widehat{\kA}_s$ under the embedding 
$$
 \widehat{\kA}_s \longhookrightarrow \widetilde{\kA}_s \longhookrightarrow \widetilde{\kB}_+ \times \widetilde{\kB}_- \xrightarrow{\widetilde{\xi}^\kB_+ \times 
\breve{\xi}^\kB_+} \lieg\llbrace u_+\rrbrace \times \lieg\llbrace u_-\rrbrace,
$$
is the Lie algebra 
$
\widehat{\lieW}_{(c, d)}^{\mathrm{trg}} = 
\bigl(\mathbbm{1} \times \mathsf{Ad}_{T}\bigr)\Bigl(\bigl(u_+ \lieg\llbracket u_{+}\rrbracket \times \{0\}\bigr)  + \bigl(\{0\} \times u_{-} \lieg\llbracket u_{-}\rrbracket\bigr) + \lied_{(c, d)}\Bigr),
$
as asserted. 
\end{proof}

\begin{example}
Let $\lieg = \mathfrak{sl}_2(\CC)$ and  $
h =
\left(
\begin{array}{cc}
1 & 0 \\
0 & -1
\end{array}
\right),$
$
e =
\left(
\begin{array}{cc}
0 & 1 \\
0 & 0
\end{array}
\right)$, $
f =
\left(
\begin{array}{cc}
0 & 0 \\
1 & 0
\end{array}
\right)$. Then  $\widehat{\lieW}_{(1, 1)}^{\mathrm{trg}} = \bigl(u_+ \lieg\llbracket u_+\rrbracket \times \{0\}\bigr) + 
\bigl(\{0\} \times u_-^2 \lieg\llbracket u_-\rrbracket\bigr) + \lieu$, where   
$$
\lieu =  \bigl\langle (0, u_{-} h), (0, u_{-} f), (0, f), 
(f, u_{-} e), (e, u^{-1}_{-} f), (h, -h)\bigr\rangle_{\CC}.
$$
The formula (\ref{E:geomRMatrNodal}) gives the following solution of (\ref{E:CYBE}):
\begin{equation}\label{E:SolutionForPair11}
r_{(1, 1)}^{\mathrm{trg}}(u, v) = \dfrac{1}{4} \dfrac{u+v}{u-v} h \otimes h + \dfrac{u}{u-v} f \otimes e +
\dfrac{v}{u-v} e \otimes f + (v-u) f \otimes f.
\end{equation}
\end{example}

\begin{remark}\label{R:RemarkCremmerGervais} Let $E$ be a Weierstra\ss{} curve and $\kA = \kA_{(n ,d)}$ be the sheaf of Lie algebras
attached to a pair $(n, d)$, where $0 < d < n$ and $\mathsf{gcd}(n, d) = 1$. Explicit expressions for the corresponding geometric $r$-matrix $\rho^{E}_{(n, d)}$ are known.

\smallskip
\noindent
1. Let $E$ be an elliptic curve.  The corresponding solution $r^{\mathrm{ell}}_{(c, d)}(x, y)$ of 
(\ref{E:CYBE}) is an elliptic solution discovered by Belavin \cite{Belavin}; see e.g.~\cite[Theorem 5.5]{BH}. For any $p \in E$, we have the Manin triple
$\widetilde{\lieA}_p = \widehat{\lieA}_p \dotplus \lieA_{(p)}$, which can be identified
with a Manin triple of the form $\lieg\llbrace z \rrbrace = \lieg\llbracket z \rrbracket \dotplus \lieW_{(c, d)}^{\mathrm{ell}}$ for an appropriate Lagrangian subalgebra $\lieW_{(c, d)}^{\mathrm{ell}} \subset \lieg\llbrace z \rrbrace$. This Manin triple appeared for the first time in the work of Reyman and Semenov-Tyan-Shansky \cite{ReymanST}. A description of the Lie algebra
$\lieW_{(c, d)}^{\mathrm{ell}}$ via generators and relations was given for $(c, d) = (1, 1)$ by 
Golod \cite{Golod}, and for arbitrary $(c, d)$ by Skrypnyk \cite{SkrypnykEllAlg}.

\smallskip
\noindent
2. Let $E$ be nodal.
The  (quasi-)trigonometric solution $r^{\mathrm{trg}}_{(c, d)}(x, y)$  of  (\ref{E:CYBE})  was computed in \cite[Theorem A]{BurbanGalinatStolin}. We recall the corresponding formula.
 Let
$$\bar{\Phi} := \bigl\{(i, j) \in \NN^2 \,\big|\, 1 \le i, j \le n \bigr\} \cong 
\ZZ_n \times \ZZ_n \quad \mbox{and} \quad 
\Phi_+ := \bigl\{(i, j) \in \bar{\Phi} \,\big|\,  i < j \bigr\}.$$ Then we have a permutation
$
\bar{\Phi} \stackrel{\tau}\lar \bar{\Phi}, \; (i, j) \mapsto (i+c, j+c)$ of order $n$. For any $\alpha \in \Phi_+$, let
$
p(\alpha) = \min\left\{k \in \NN \, \big|\, \tau^k(\alpha) \notin \Phi_+\right\}.
$
For any $1 \le i \le n-1$, we put:
$
q_i := \tau^{i}(\varepsilon) - \tau^{i-1}(\varepsilon)$ and $f_i := \frac{1}{2}\bigl(\tau^{i}(\varepsilon) + \tau^{i-1}(\varepsilon)\bigr) - \frac{1}{n} I,
$
where $I$ is the identity matrix and $\varepsilon = e_{11}$ is the first matrix unit. Then $(q_1, \dots, q_{n-1})$ is a basis of the standard Cartan part $\lieh$ of the Lie
algebra $\lieg$. Let $(q_1^\ast, \dots, q_{n-1}^\ast)$ be the dual basis of  $\lieh$ with respect to the trace form. The solution of (\ref{E:CYBE}) corresponding
to $(E, \kA)$ is given by the formula 
\begin{equation}\label{E:TrigonomCremmerGervais}
r_{(c, d)}^{\mathrm{trg}}(x, y) = r_{\circ}(x, y) + \ttau_{(c, d)}(x, y),
\end{equation}
where
$
r_{\circ}(x, y) 
$
is the standard trigonometric $r$-matrix (\ref{E:StandardFormJimbo}) and 
$$
\ttau_{(c, d)}(x,y) = \sum\limits_{\alpha \in \Phi_+}\Bigl(\Bigl(\sum\limits_{k = 1}^{p(\alpha)-1} e_{\tau^k(\alpha)} \wedge e_{-\alpha}\Bigr) +
x e_{\tau^{p(\alpha)}(\alpha)} \otimes
e_{-\alpha}  - y e_{-\alpha} \otimes x e_{\tau^{p(\alpha)}(\alpha)}\Bigr) + \sum\limits_{i=1}^{n-1} q_i^\ast \otimes f_i.
$$
For $(c, d) = (1, 1)$ we recover the formula (\ref{E:SolutionForPair11}) above. 

\smallskip
\noindent
3. Let $E$ be cuspidal.  The corresponding rational solution $r^{\mathrm{rat}}_{(c, d)}(x, y)$ of
(\ref{E:CYBE}) was computed  in \cite[Theorem 9.6 and Example 9.7]{BH}. The Manin triple 
$\widetilde{\lieA}_s = \widehat{\lieA}_s \dotplus \lieA_{(s)}$ (where $s$ is the singular point of $E$) has the form 
$
\lieg\llbrace z^{-1}\rrbrace = \widehat{\lieW}_{(c, d)}^{\mathrm{rat}} \dotplus \lieg[z]
$
and the corresponding Lagrangian subalgebra $\widehat{\lieW}_{(c, d)}^{\mathrm{rat}} \subset
\lieg\llbrace z^{-1}\rrbrace$ was explicitly described in  \cite[Lemma 9.2]{BH}. \hfill $\lozenge$
\end{remark}

\subsection{Explicit geometrization of quasi-constant solutions of CYBE}
Let $\lieg$ be a simple Lie algebra. According to the Whitehead's lemma, we have: $H^1\bigl(\lieg, \wedge^2(\lieg)\bigr) = 0$.
Moreover, it can be shown  that any Lie bialgebra structure $\lieg \stackrel{\delta}\lar \lieg \otimes \lieg$ is of the form
$\delta = \partial_\ttau$,  where $\ttau \in \lieg \otimes \lieg$ is such that
\begin{equation}\label{E:quasiconstant}
[\ttau^{12}, \ttau^{13}] + [\ttau^{12}, \ttau^{23}] + [\ttau^{13}, \ttau^{23}] = 0 \quad \mbox{\rm and} \quad 
\ttau^{12} + \ttau^{21} = \lambda \gamma
\end{equation}
for some $\lambda \in \CC$, i.e.~$\ttau$ is a  solution of the classical Yang--Baxter equation for constants (cCYBE); see e.g. \cite[Section 5.1]{EtingofSchiffmann}. Of course, without  loss of generality we may assume that $\lambda \in \left\{0, 1\right\}$.

\smallskip
\noindent
The following result is due to Stolin \cite{StolinBialg}. 

\begin{theorem}\label{T:StolinConstantSolutions} Solutions of cCYBE can be described in the following terms. 
\begin{itemize}
\item[(a)] Tensors $\ttau \in \lieg \otimes \lieg$ satisfying 
\begin{equation}\label{E:trigonomconstant}
[\ttau^{12}, \ttau^{13}] + [\ttau^{12}, \ttau^{23}] + [\ttau^{13}, \ttau^{23}] = 0 \quad \mbox{\rm and} \quad 
\ttau^{12} + \ttau^{21} = \gamma
\end{equation}
stand in bijection with Manin triples $\liem = \lied \dotplus \liew$, where 
$$\lied = \bigl\{(a, a) \, |\, a \in \lieg\bigr\} \subset \liem := \lieg \times \lieg$$ and the 
bilinear form
$\liem \times \liem \stackrel{F}\lar \CC$ is given by the rule:
$$F\bigl((a', b'), (a'', b'')\bigr) = \kappa(a', a'') - \kappa(b', b'').$$
\item[(b)] Tensors $\ttau \in \lieg \otimes \lieg$ satisfying
\begin{equation}\label{E:rationalconstant}
[\ttau^{12}, \ttau^{13}] + [\ttau^{12}, \ttau^{23}] + [\ttau^{13}, \ttau^{23}] = 0 \quad \mbox{\rm and} \quad 
\ttau^{12} + \ttau^{21} = 0
\end{equation}
stand in bijection with Manin triples $\liem = \lied \dotplus \liew$, where 
$$\lied = \bigl\{a \, |\, a \in \lieg\bigr\} \subset \liem := \lieg[\varepsilon]/(\varepsilon^2)$$ and the 
bilinear form bilinear form
$\liem \times \liem \stackrel{F}\lar \CC$ is given  by the rule:
$$
F\bigl((a' + \varepsilon b'),  (a'' + \varepsilon b'')\bigr) = \kappa(a', b'') + 
\kappa(a'', b').
$$
\end{itemize}
\end{theorem}

\smallskip
\noindent
\emph{Comment to the proof}. The correspondence between solutions of cCYBE and Manin triples is as follows. Let $(g_1, \dots, g_q)$ be a basis of $\lieg$. 
\begin{itemize}
\item[(a)] Let $((w_1^+, w_1^-), \dots, (w_q^+, w_q^-)\bigr)$ be the basis of $\liew
\subset \liem = \lieg \times \lieg$, which is dual to the basis $\bigl((g_1, g_1), \dots, (g_q, g_q)\bigr)$ of $\lied$. Then the solution of (\ref{E:trigonomconstant}) corresponding to $\liew$ is given by the formula
\begin{equation}\label{E:trigonomquasiconstantant}
\ttau := \sum\limits_{i = 1}^q  g_i \otimes w_i^+; 
\end{equation}
see \cite[Section 6]{StolinBialg}.
\item[(b)] Similarly, let 
$\bigl( h_1  + \varepsilon g_1^\ast, \dots,  h_q + \varepsilon  g_q^\ast \bigr)$ 
be the 
basis of $\liew
\subset \liem = \lieg[\varepsilon]/(\varepsilon^2)$, which is dual to the basis $\bigl(g_1, \dots, g_q)$ of $\lied$. Then the solution of (\ref{E:rationalconstant}) corresponding to $\liew$ is given by the formula
\begin{equation}
\ttau := \sum\limits_{i = 1}^q g_i \otimes h_i = - \sum\limits_{i = 1}^q h_i \otimes g_i;
\end{equation}
see 
\cite[Theorem 3.12]{StolinBialg}.
\end{itemize}

\begin{remark}\label{R:Wildness}
All solutions of (\ref{E:trigonomconstant}) were classified by Belavin and Drinfeld in \cite[Section 6]{BelavinDrinfeldBook}. On the other hand, let  $\lieg = \mathfrak{sl}_n(\CC)$ and
$a, b \in \lieg$ be such that $[a, b] = 0$. Then $\ttau = a \wedge b$ satisfies
(\ref{E:rationalconstant}). This implies that classification of all solutions of  
(\ref{E:rationalconstant}) is a representation-wild problem; see \cite{GelfandPonomarev}. \hfill $\lozenge$
\end{remark}

\begin{remark}
Any  solution $\ttau \in \lieg \otimes \lieg$ of cCYBE  defines a solution of CYBE. 
\begin{itemize}
\item[(a)] If $\ttau \in \lieg \otimes \lieg$ satisfies (\ref{E:trigonomconstant}) then
$
r(x, y) = \dfrac{y}{x-y} \gamma 
 + \ttau
$
satisfies (\ref{E:CYBE}). 
\item[(b)] If $\ttau \in \lieg \otimes \lieg$ satisfies (\ref{E:rationalconstant}) then
$
r(x, y) = \dfrac{1}{x-y} \gamma 
 + \ttau
$
satisfies (\ref{E:CYBE}). 
\end{itemize}
Such solutions of CYBE  are called \emph{quasi-constant}.  \hfill $\lozenge$
\end{remark}

\begin{theorem}\label{T:QuasiConstantGeometrization} Let $\lieg \times \lieg = \lied \dotplus \liew$ be a Manin triple as in Theorem \ref{T:StolinConstantSolutions} and $\ttau \in \lieg \otimes \lieg$ be the corresponding solution of (\ref{E:trigonomconstant}), given by the formula 
(\ref{E:trigonomquasiconstantant}). 
  Choose homogeneous coordinates on $\PP^1$ and define a nodal Weierstra\ss{} curve 
$E$ via  the pushout diagram (\ref{E:Bicartesian}), where $s_+ = (0: 1)$ and $s_- = (1: 0)$.  Define the sheaf of Lie algebras $\kA$ as the pullback 
\begin{equation}\label{E:PullBackNeu}
\begin{array}{c}
\xymatrix{
\kA \ar[r] \ar[d] & \liew \ar@{^{(}->}[d]\\
\kB  \ar[r]^-{\ev} & \lieg \times \lieg 
}
\end{array}
\end{equation}
in the category $\Coh(E)$, 
where $\kB := \lieg \otimes_{\CC} \bigl(\nu_\ast(\kO_{\PP^1})\bigr)$, whereas $\liew$ and $\lieg \times \lieg$ are considered as skyscraper shaves supported at the singular point $s \in E$  and $\ev$ is induced by the canonical isomorphisms
$\kO_{\PP^1}\Big|_{s_\pm} \cong \CC$. Then $(E, \kA)$ is a geometric CYBE datum and the corresponding geometric $r$-matrix is  the quasi-constant solution $
r(x, y) = \dfrac{y}{x-y} \gamma 
 + \ttau
$
of  (\ref{E:CYBE}).
\end{theorem}
\begin{proof}
It follows from the definition of $\kA$ that $\lieA = \Gamma(U, \kA) = \Gamma(U, \kB) \cong \lieL = \lieg\bigl[z, z^{-1}\bigr]$. Next, $\Gamma(E, \kB) \cong \lieg$ and
$H^1(E, \kB) = 0$. From (\ref{E:PullBackNeu}) we obtain an exact sequence
\begin{equation*}
0 \lar H^0(E, \kA) \lar  \lied \dotplus \liew
\stackrel{\cong}\lar (\lieg \times \lieg) \lar 
  H^1(E, \kA) \lar 0, 
\end{equation*}
which implies that $H^0(E, \kA) = 0 = H^1(E, \kA)$. From (\ref{E:MayerVietoris}) we get a direct sum decomposition $\widetilde\lieA = \widehat\lieA \dotplus \lieA$, where $\widehat\lieA$ is the completion of the stalk of $\kA$ at $s$ and $\widetilde\lieA$ is its rational hull. We have: $\widetilde\lieA \cong \lieg\llbrace x_+\rrbrace \times \lieg\llbrace x_-\rrbrace$ and 
$
\lieA \cong \bigl\{\bigl(a x_+^k, a x_-^{-k}\bigr) \, \big| \, a \in \lieg, k \in \NN_0 \bigr\}.
$ Moreover, it follows from (\ref{E:PullBackNeu}) that 
$
\widehat\lieA \cong x_+ \lieg \llbracket x_+\rrbracket + x_- \lieg \llbracket x_-\rrbracket + \liew.
$
In particular,  $\widehat\lieA$ is a Lagrangian Lie subalgebra 
of $\widetilde\lieA$ and $(E, \kA)$ is a geometric CYBE datum, as asserted. 

\smallskip
\noindent
The recipe to compute the geometric $r$-matrix of $(E, \kA)$ is given by (\ref{E:geomRMatrNodal}). Let $(g_1, \dots, g_q)$ be a basis of $\lieg$, $(g^*_1, \dots, g^*_q)$ be the corresponding dual basis with respect to the Killing form and
$\bigl(a_{(k, i)} = g_i z^k \, | \, 1 \le i \le q, k \in \ZZ \bigr)$ be the corresponding basis of $\lieL$. Note  that the elements $w_{(k,\,  i)}$ defined by
(\ref{E:dualelements}) belong to $x_+ \lieg \llbracket x_+\rrbracket + x_- \lieg \llbracket x_-\rrbracket + \liew$ for $k \ne 0$. As a consequence,  the elements
$h_{(k, i)}$ given by (\ref{E:expansiongeomrmatrnodalpoint}) are zero for $k \ne 0$.

\smallskip
\noindent
Let $\bigl((w_1^+, w_1^-), \dots, (w_q^+, w_q^-)\bigr)$ be a basis of $\liew$ dual to the basis $\bigl((g_1, g_1), \dots, (g_q, g_q)\bigr)$ of $\lied$. 
For any $1 \le i \le q$ there exists a uniquely determined element $v_i \in \lieg$ such that $(-g_i^\ast, 0) + (v_i, v_i) = (-w_i^+, -w_i^-)$. It follows from 
(\ref{E:expansiongeomrmatrnodalpoint}) that $h_{(0, i)} = v_i = - w_i^-$ for all $1 \le i \le q$ and $w_i^+ = g_i^\ast + w_i^-$ (here we use that $K\bigl(g_i, g^*_j) = \kappa(g_i, g^\ast_j) = \delta_{ij}$). From (\ref{E:geomRMatrNodal}) we conclude that
$$
r(x, y) = \frac{y}{x-y} \gamma + \sum\limits_{i = 1}^q w_i^- \otimes g_i = 
\frac{y}{x-y} \gamma + \sum\limits_{i = 1}^q (w_i^+ - g_i^\ast) \otimes g_i
= \frac{x}{x-y} \gamma - \sum\limits_{i = 1}^q w_i^+ \otimes g_i.
$$
Since $r(x, y)$ is skew-symmetric, we have:
$$
r(x, y) = - r^{21}(y, x) = \frac{y}{x-y} \gamma + \sum\limits_{i = 1}^q g_i \otimes w_i^+,
$$
as asserted. 
\end{proof}

\begin{remark}\label{R:CuspidalCaseQuasiConst}
An analogous statement is true for the rational quasi-constant solutions. 
Let $\lieg[\varepsilon]/(\varepsilon^2) = \lied \dotplus \liew$ be a Manin triple as in Theorem \ref{T:StolinConstantSolutions} and $\ttau \in \lieg \otimes \lieg$ be the corresponding solution of (\ref{E:rationalconstant}).
Choose homogeneous coordinates on $\PP^1$ and define a cuspidal Weierstra\ss{}  curve 
$E$ via  the pulldown diagram 
\begin{equation}\label{E:BicartesianCusp}
\begin{array}{c}
\xymatrix{
\Spec\bigl(\CC[\varepsilon]/(\varepsilon^2)\bigr) \ar@{^{(}->}[r]^-{\tilde\eta} \ar@{->>}[d]_-{\tilde\nu} & \PP^1 \ar@{->>}[d]^-{\nu} \\
\Spec(\CC) \ar@{^{(}->}[r]^-{\eta} & E
}
\end{array}
\end{equation}
where the image of $\tilde{\eta}$ is the scheme supported at $(1: 0)$. 
Similarly to the nodal case,  we define the sheaf of Lie algebras $\kA$ as the pullback
\begin{equation}\label{E:PullBackNeuCusp}
\begin{array}{c}
\xymatrix{
\kA \ar[r] \ar[d] & \liew \ar@{^{(}->}[d]\\
\kB  \ar[r]^-{\ev} & \lieg[\varepsilon]/(\varepsilon^2)
}
\end{array}
\end{equation}
where $\kB := \lieg \otimes_{\CC} \bigl(\nu_\ast(\kO_{\PP^1})\bigr)$. Let $U$ be the regular part of $E$. Then we have: $\lieA = \Gamma(U, \kA) \cong \lieg[z]$. 
As in the nodal case, it follows that $\widetilde{\lieA}_s = \widehat{\lieA}_s \dotplus \lieA_{(s)}$ is a Manin triple, which can be identified with the Manin triple
$
\lieg\llbrace z^{-1} \rrbrace = \widehat{\lieA} \dotplus \lieg[z],
$
where the symmetric non-degenerate bilinear form $\widetilde{\lieA}_s \times 
\widetilde{\lieA}_s \stackrel{F_s^\omega}\lar \CC$ can be identified with 
\begin{equation}\label{E:FormSeries2}
\lieg\llbrace z^{-1} \rrbrace  \times \lieg\llbrace z^{-1} \rrbrace  \stackrel{F}\lar \CC, (a z^k, b z^l) \mapsto \delta_{k+l+1, 0} \kappa(a, b). 
\end{equation}
In these terms, we have: $\widehat{\lieA} = z^{-2} \lieg\llbracket z^{-1}\rrbracket + \liew$, where we identify $\liew \subseteq \lieg[\varepsilon]/(\varepsilon^2)$ with a subspace of $\lieg + z^{-1} \lieg$. It is precisely the setting of Stolin's theory of rational solutions \cite{Stolin}. As in the nodal case, one can derive from the formula  (\ref{E:geomRMatrCusp}) that the corresponding geometric $r$-matrix is given by the formula
$$
r(x, y) = \dfrac{1}{x-y}\gamma - \ttau = \dfrac{1}{x-y}\gamma + \sum\limits_{i = 1}^q h_i \otimes g_i,
$$
where $\bigl(h_1 + \varepsilon g_1^\ast, \dots, h_q + \varepsilon  g_q^\ast\bigr)$ is 
the basis of $\liew
\subset \liem = \lieg[\varepsilon]/(\varepsilon^2)$ dual to  $\bigl(g_1, \dots, g_q)$.
\end{remark}

\section{Appendices}\label{S:Appendix}

\subsection{Road map to this work}
 Let $\lieK$ be the Kac--Moody Lie algebra over $\mathbb{C}$ associated with an arbitrary  symmetrizable generalized Cartan matrix $A$. It is well-known $\lieK$ admits a natural triangular decomposition
$
\lieK = \lieK_+ \oplus \lieH \oplus \lieK_-.
$
Moreover, $\lieK$ has finite dimensional root spaces as well as an essentially unique non-degenerate symmetric invariant bilinear form $\lieK \times \lieK \stackrel{B}\lar \mathbb{C}$ (which coincides with the Killing form if $\lieK$ is finite dimensional); see \cite{Kac}. As discovered by Drinfeld \cite{Drinfeld},  $\lieK$
has a structure of a Lie bialgebra $\lieK \stackrel{\delta_\circ}\lar \lieK \otimes \lieK$, called in this paper \emph{standard}. Existence of $\delta_\circ$ follows  from the  root space decomposition of $\lieK$ and as well as invariancy and non-degeneracy of the  bilinear form $B$. The action of  $\delta_\circ$  on the Cartan--Weyl  generators of $\lieK$ can be expressed purely in terms of the entries  of the matrix $A$. 

\smallskip
\noindent
The Lie algebra $\lieE = \lieK \times \lieK$ is also equipped with a symmetric non-degenerate invariant  bilinear form 
$$
\lieE \times \lieE \stackrel{F}\lar \CC, \; \bigl((a_1, b_1), (a_2, b_2)\bigr)\mapsto B(a_1, a_2) - B(b_1, b_2).
$$
Identifying 
$\lieK$ with the diagonal in  $\lieE$, we get a direct sum decomposition 
$
\lieE = \lieK \dotplus \lieW_\circ, 
$
where $\lieW_\circ = \bigl\{\bigl((c_+, h), (c_-, -h)\bigr) \in  (\lieK_+ \oplus \lieH)\times (\lieK_- \oplus \lieH) \, \big| \, c_\pm \in \lieK_\pm, h\in \lieH\bigr\}.$  Moreover,  
$\lieK$ and $\lieW_\circ$ are Lagrangian Lie subalgebras of $\lieE$ with respect to the form $F$. 
The   Manin triple $
\lieE = \lieK \dotplus \lieW_\circ
$ ``determines'' the cobracket $\delta_\circ$ in the following sense: 
$$
F\bigl(\delta_\circ(c), w_1 \otimes w_2\bigr) = B\bigl(c, [w_1, w_2]\bigr) \quad \mbox{\rm for all} \quad c \in \lieK \; \mbox{and} \;  w_1, w_2 \in \lieW_\circ.
$$
Following the  work of Karolinsky and Stolin \cite{KarolinskyStolin}, we study   ``twisted'' Lie bialgebra cobrackets of the form
$\delta_{\ttau} = \delta_\circ + \partial_{\ttau}$, where $\ttau \in \wedge^2(\lieK)$ and $\partial_{\ttau}(a) = \bigl[a\otimes 1 + 1 \otimes a, \ttau\bigr]$ for $a \in \lieK$. By  Proposition \ref{P:TwistLieBialgStructures} (see also \cite[Theorem 7]{KarolinskyStolin}), $\delta_\ttau$ is a Lie bialgebra cobracket if and only if the tensor 
$$
\bigl(\alt\bigl((\delta_\circ \otimes \mathbbm{1})(\ttau)\bigr) - [\ttau^{12}, \ttau^{13}] - [\ttau^{12}, \ttau^{23}] -  [\ttau^{13}, \ttau^{23}] \in \lieK^{\otimes 3}
$$
is ad-invariant, where $\alt(a\otimes b \otimes c) := a\otimes b \otimes c + c \otimes a \otimes b + b \otimes c \otimes a$ for $a, b, c \in \lieK$.
In Section \ref{S:LieBilgLagrDec}, we  elaborate a general framework to study twists of a given Lie bialgebra structure  (generalizing and extending results known 
in the finite dimensional case \cite{KarolinskyStolin}) and prove  that such $\ttau$ are parametrized by  Manin triples of the form
$
\lieE = \lieK \dotplus \lieW,
$
where $\lieW$ is a Lie subalgebra of $\lieE$ commensurable with $\lieW_\circ$; see Theorem \ref{T:ManinTriplesTwists}.

From the point of view of applications in the theory of classical integrable systems as well as from the purely algebraic point of view, the most interesting and rich case is when $\lieK =  \doublewidetilde{\lieG}$ is an \emph{affine} Kac--Moody algebra.
 Then the center $\lieZ$ of the Lie algebra $\doublewidetilde{\lieG}$ is one-dimensional.  
 Let $\lieG = \widetilde{\lieG}/\lieZ$  be the ``reduced'' affine Lie algebra, where  $\widetilde{\lieG} = \bigl[\doublewidetilde{\lieG}, \doublewidetilde{\lieG}\bigr]$.
 It follows from the explicit formulae for $\delta_\circ$ that one gets an induced Lie bialgebra cobracket $\lieG \stackrel{\delta_\circ}\lar \lieG \otimes \lieG$. An  \emph{inconspicuous but decisive advantage} to pass from 
$\lieK$ to $\lieG$ is due to the fact that for any $n \in \NN$, the $n$-fold tensor product $\lieG^{\otimes n}$ does not have non-zero ad-invariant elements; see Proposition \ref{P:basicsonloops}. 
As a consequence,  $\ttau \in \wedge^2(\lieG)$ defines  a twisted Lie bialgebra cobracket $\lieG \stackrel{\delta_\ttau}\lar \lieG \otimes \lieG$ if and only if it satisfies  the  \emph{twist equation}
$$
\bigl(\alt\bigl((\delta_\circ \otimes \mathbbm{1})(\ttau)\bigr) - [\ttau^{12}, \ttau^{13}] - [\ttau^{12}, \ttau^{23}] -  [\ttau^{13}, \ttau^{23}]  = 0
$$
 introduced in \cite{KarolinskyStolin},
which is an incarnation  of the classical Yang--Baxter equation
$$
\bigl[{r}^{12}(x_1, x_2), {r}^{23}(x_2, x_3)\bigr] + \bigl[{r}^{12}(x_1, x_2),
{r}^{13}(x_1, x_3)\bigr]
+ \bigl[{r}^{13}(x_1, x_3), {r}^{23}(x_2, x_3)\bigr] = 0. 
$$
To see the latter statement,  recall that the ``reduced'' affine Lie algebra $\lieG$ is isomorphic to  a twisted 
loop algebra $\lieL = \lieL(\lieg, \sigma)$, where $\lieg$ is a finite dimensional simple Lie algebra and $\sigma$ is an automorphism of its Dynkin diagram \cite{Carter, Kac}. 

Let us for simplicity assume that the affine Cartan matrix $A$ corresponds to an extended Dynkin diagram. In this case, the automorphism $\sigma$ is trivial and
$\lieL = \lieg\bigl[z, z^{-1}\bigr]$ is the usual loop algebra. We have  a non-degenerate invariant bilinear form
$$
\lieL \times \lieL \stackrel{B}\lar \CC, \quad B(a z^k, bz^l) = \kappa(a, b) \delta_{k+l, 0},
$$
where $\kappa$ denotes the Killing form of $\lieg$.  A theorem due to Gabber and Kac asserts that there exists an isomorphism of Lie algebras $\lieG \stackrel{\cong}\lar \lieL$ identifying both non-degenerate invariant  bilinear forms on $\lieG$ and $\lieL$ up to a rescaling; see \cite[Theorem 8.5]{Kac}. We show (see Corollary \ref{C:StandradStructure})  that under this identification, the standard Lie bialgebra cobracket $\delta_\circ$  on $\lieL$ is given by the formula
$$\lieL \stackrel{\delta_\circ}\lar \wedge^2(\lieL),  \, f(z) \mapsto 
\bigl[f(x) \otimes 1 + 1 \otimes f(y), r_\circ(x, y)\bigr],$$
where 
$
r_{\circ}(x, y) = \dfrac{1}{2}\Bigl(\dfrac{x+y}{x-y} \gamma + \sum\limits_{\alpha} e_{-\alpha} \wedge e_{\alpha}\Bigr)
$
is the  ``standard'' solution   of CYBE.
As a consequence, 
twists of the standard Lie bialgebra cobracket $\lieL \stackrel{\delta_\circ}\lar \wedge^2(\lieL)$ have the form
\begin{equation*}\label{E:LieBialgebra}
\lieL \stackrel{\delta_\ttau}\lar \wedge^2(\lieL),  \, f(z) \mapsto 
\bigl[f(x) \otimes 1 + 1 \otimes f(y), r_\ttau(x, y)\bigr],
\end{equation*} where  
$\ttau(x, y)  \in (\lieg \otimes \lieg)\bigl[x, x^{-1}, y, y^{-1}\bigr]$ is such that
$
r_\ttau(x, y) = r_\circ(x, y) + \ttau(x, y)
$
is a solution of CYBE; see Theorem \ref{T:TwistsCYBEGeometry}. It turns out that any such solution of CYBE is equivalent (with respect to the  equivalence relation given by (\ref{E:equiv1}) and 
(\ref{E:equiv2})) to a trigonometric solution of CYBE with one spectral parameter (\ref{E:CYBE1}); see Proposition \ref{P:TwistIsTrigonometric}. Trigonometric solutions of CYBE were completely classified by Belavin and Drinfeld \cite{BelavinDrinfeld}.  However, our  work  is completely independent of that classification  and in particular provides an alternative approach to the theory of trigonometric solutions of CYBE.

The latter point is explained by the algebro-geometric perspective on Lie bialgebra structures on twisted loop algebras. To proceed to this, we first show  that twists $\ttau \in \wedge^2(\lieL)$ of the standard Lie bialgebra cobracket $\lieL \stackrel{\delta_\circ}\lar \wedge^2(\lieL)$ are in bijection with Manin triples
\begin{equation*}
\mathfrak{D}  = \lieC \, \dotplus \, \lieW, \; \lieW \asymp 
\lieW^\circ,
\end{equation*}
where $\mathfrak{D} = \lieL_+ \times \lieL_- = \lieL \times \lieL^\ddagger$ and $\lieC = \bigl\{(f, f^\ddagger) \,|\, f \in \lieL\bigr\}$; see Theorem  \ref{T:ManinTriplesLoopTwists}. If $\lieL = \lieL(\lieg, \sigma) \subseteq 
\lieg\bigl[z_+, z_+^{-1}\bigr]$ then 
$\lieL^\ddagger := \lieL\bigl(\lieg, \sigma^{-1}\bigr) \subseteq 
\lieg\bigl[z_-, z_-^{-1}\bigr]$ and  $\bigl(az_+^k)^\ddagger =  
a z_{-}^{-k}$.  
The key statement is that $\lieW$ is stable under multiplications by the elements of   the algebra 
$$\CC[t_+, t_-]/(t_+ t_-) \cong \bigl\{(f_+, f_-) \in \CC\bigl[t_+] \times \CC[t_-] \,|\, f_+(0) = f_-(0) \bigr\},
$$ where $t_\pm = z_\pm^m$ and $m$ is the order of the automorphism $\sigma$; see Lemma \ref{L:OrdersFromMT1}.  Its proof  uses  the fact that any bounded coisotropic Lie subalgebra of $\lieL$ is stable under the multiplication by the elements of $\CC[t]$; see Theorem \ref{P:CoisotropicSubalgebras}. In its turn, the proof of   Theorem \ref{P:CoisotropicSubalgebras} is  based on properties of affine root systems  as well as on the result  of Kac and Wang \cite[Proposition 2.8]{KacWang}. 

\medskip
\noindent
The crux of our work  is that  Manin triples $\mathfrak{D}  = \lieC \, \dotplus \, \lieW, \; \lieW \asymp 
\lieW^\circ$ 
are of algebro-geometric nature. Projecting the Lie algebra $\lieW$ to each factor $\lieL_\pm$ of $\mathfrak{D}$, we get  a pair of Lie algebras $\lieW_\pm \subset \lieL_\pm$, which can be glued  to a Lie algebra bundle $\kB$ on the projective line $\PP^1$, whose generic fibers are isomorphic to the Lie algebra $\lieg$; see Proposition 
\ref{P:SheafLieAlgProjLine}.
Let  $\liew = \lieW/(t_+, t_-)\lieW$,  $\liew_\pm = \lieW_\pm/t_\pm\lieW$ and 
$\liew \stackrel{\theta}\hookrightarrow \liew_+ \times \liew_-$ be the canonical embedding.
Using the  theory of torsion free sheaves on singular projective curves developed in \cite{Survey, Thesis}, we  attach to the datum $(\kB, \liew, \theta)$ a sheaf of Lie algebras $\kA$ on a plane nodal cubic curve $E = \overline{V(u^2 - v^3 - v^2)} \subset \PP^2$; see Proposition \ref{P:GeomDataFromTwists}. This sheaf has the following properties. 
\begin{itemize}
\item $\kA\big|_p \cong \lieg$ for all $p \in \breve{E}$, where $\breve{E}$ is the smooth part of $E$. 
\item $\kA$ has vanishing cohomology: $H^0(E, \kA) = 0 = H^1(E, \kA)$. 
\item $\kA_s$ is a Lagrangian  Lie subalgebra of the rational hull of $\kA$ (which is a simple Lie algebra over the field  of rational functions of $E$), where $s$ is the unique singular point of $E$. 
\end{itemize}
The constructed geometric datum  $(E, \kA)$ fits precisely into the framework of the algebro-geometric theory of solutions of CYBE developed by Burban and Galinat \cite[Theorem 4.3]{BurbanGalinat}. In that work, the authors constructed  a canonical section (called geometric $r$-matrix)
$$
\rho \in \Gamma\bigl(\breve{E} \times \breve{E}\setminus \Sigma, \kA \boxtimes \kA), \;  \mbox{\rm where} \; \Sigma \subset \breve{E} \times \breve{E} \;  \mbox{\rm is the diagonal}, \;
$$
which satisfies a sheaf-theoretic version of the classical Yang--Baxter equation:
$$
\bigl[\rho^{12}, \rho^{13}\bigr] + \bigl[\rho^{13}, \rho^{23}\bigr] + \bigl[\rho^{12}, \rho^{23}\bigr] = 0 \; \mbox{and} \;   \rho(p_1, p_2) = - \rho^{21}(p_2, p_1) \; \mbox{\rm for} \; p_1, p_2 \in \breve{E}.
$$
 In \cite[Proposition 4.12]{BurbanGalinat} it was shown that $\Gamma(\breve{E}, \kA)$ is a Lie bialgebra: the linear map
\begin{equation*}
\Gamma(\breve{E}, \kA) \stackrel{\delta_\rho}\lar \Gamma(\breve{E}, \kA) \otimes \Gamma(\breve{E}, \kA), \; f(t) \mapsto \bigl[f(u) \otimes 1 + 1 \otimes f(v), \rho(u, v)\bigr]
\end{equation*}
is a skew-symmetric one-cocycle satisfying the co-Jacobi identity. It follows from the construction of  $(E, \kA)$ that $\Gamma(\breve{E}, \kO_E) \cong 
\CC\bigl[t, t^{-1}\bigr]$ and $\Gamma(\breve{E}, \kA) \cong \lieL$.  In Theorem \ref{T:TwistsCYBEGeometry} we show
that Lie bialgebras $\bigl(\Gamma(\breve{E}, \kA), \delta_\rho\bigr)$ and $(\lieL, \delta_\ttau)$ are isomorphic. 
This statement also  allows to identify the trivialized geometric $r$-matrix $\rho$ with the solution $r_\ttau(x,y) = r_\circ(x, y) + \ttau(x, y)$ of CYBE. The latter fact in particular means that any trigonometric solution of CYBE arises from an appropriate  geometric datum $(E, \kA)$, concluding the geometrization programme started in \cite{BurbanGalinat}. 

\smallskip
\noindent
In Section \ref{S:ExplicitComputations}, we deal with concrete examples. 
In Theorem \ref{T:QuasiConstantGeometrization}, we describe Manin triples 
\begin{equation*}
\lieg\llbrace z_+ \rrbrace \times \lieg\llbrace z_- \rrbrace = 
\lieg\bigl[z, z^{-1} \bigr] \dotplus \widehat{\lieW},
\end{equation*}
corresponding to quasi-constant trigonometric solutions of CYBE.
In Theorem \ref{T:CremmerGervaisGeometry},  we describe the corresponding Lagrangian  subalgebras $\widehat{\lieW}$
for a special class of (quasi)-trigono\-met\-ric solutions of CYBE for 
 $\lieg = \mathfrak{sl}_n(\CC)$, which were  obtained  in 
 \cite[Theorem A]{BurbanGalinatStolin}.

\subsection{Infinite dimensional Lie bialgebras}
As usual, let $\lieg$ be a finite dimensional simple complex Lie algebra and $r(x, y)$ be a solution of the classical Yang--Baxter equation (\ref{E:CYBE}). There are several essentially different possibilities to attach to $r(x, y)$ a Lie bialgebra. 

\smallskip
\noindent
1. There is a ``universal procedure'',  applicable for all three types of solutions of (\ref{E:CYBE1}):  elliptic, trigonometric and rational. As was explained in Subsection \ref{SS:SurveyCYBE}, any solution of (\ref{E:CYBEformali})  defines a Manin triple of the form $\lieg\llbrace z\rrbrace = \lieg\llbracket z\rrbracket \dotplus \lieW$ and the linear map 
$$
\lieW \stackrel{\delta_r}\lar \lieW \otimes \lieW, \, w(z) \mapsto \bigl[w(x) \otimes 1 + 1 \otimes w(y), r(x, y)\bigr]
$$
is a Lie bialgebra cobracket on $\lieW$. 
For elliptic solutions, such Manin triples appeared for the first time  in \cite{ReymanST}. A description of the corresponding Lie algebras $\lieW$ via generators and relations was given in \cite{Golod, SkrypnykEllAlg}.

\smallskip
\noindent
2. Let  $\varrho(z)$ be a trigonometric solution of CYBE with the lattice of poles 
$2\pi i \mathbb{Z}$. Then there exists $\sigma  \in \Aut_{\CC}(\lieg)$ such that 
$$
\varrho(z + 2\pi i) = \bigl(\sigma \otimes \mathbbm{1}_{\lieg}\bigr) \varrho(z) = \bigl(\mathbbm{1}_{\lieg} \otimes \sigma^{-1}\bigr) \varrho(z).
$$
Moreover, there exists $m \in \NN$ such that $\sigma^m = \mathbbm{1}_{\lieg}$; see \cite[Theorem 6.1]{BelavinDrinfeld}. It turns out that (after an appropriate change of coordinates) $\varrho$ defines a Lie bialgebra cobracket on the twisted loop algebra
$\lieL = \lieL(\lieg, \sigma)$, which is a {twist} of the {standard Lie bialgebra structure} on $\lieL$.
In this paper we prove that  such twists 
 are classified  by Manin triples of the form 
$$
\lieL \times \lieL^\ddagger = \lieD \dotplus \lieW, \quad \lieW \asymp 	 \lieW^\circ,
$$
where $\lieD = \bigl\{(f, f^\ddagger) \, \big|\, f \in \lieL \bigr\} \cong \lieL$ and $\lieW^\circ$ is the Lie algebra corresponding  to the standard Lie bialgebra cobracket on $\lieL$.  From this perspective, the theory of trigonometric solutions of CYBE  appears in a parallel way to the theory of of solutions of cCYBE. 
 Methods developed in this work should be applicable to study analogues of trigonometric solutions of CYBE  for simple Lie algebras defined over arbitrary fields.

\smallskip
\noindent
3. Lie bialgebra structures on the Lie algebra $\lieg\llbracket z\rrbracket$ were studied  in \cite{MSZ}. For any 
\begin{equation}\label{E:MSZList}
r(x, y) \in  \left\{0,\, 
\dfrac{1}{2}\Bigl(\dfrac{x+y}{x-y} \gamma + \sum\limits_{\alpha \in \Phi_+} e_{-\alpha} \wedge e_{\alpha}\Bigr), \, 
\dfrac{1}{x-y} \gamma, \, 
\dfrac{xy}{x-y} \gamma
\right\}
\end{equation}
we have  the corresponding  Lie bialgebra cobrackets 
$\lieg\llbracket z\rrbracket \stackrel{\delta_r}\lar \lieg\llbracket x\rrbracket \otimes \lieg\llbracket y\rrbracket$. It turns out that for any other Lie bialgebra cobracket $\lieg\llbracket z\rrbracket \stackrel{\delta}\lar \lieg\llbracket x\rrbracket \otimes \lieg\llbracket y\rrbracket$, the corresponding Drinfeld double $\mathfrak{D}\bigl(\lieg\llbracket z\rrbracket, \delta\bigr)$ is isomorphic to $\mathfrak{D}\bigl(\lieg\llbracket z\rrbracket, \delta_r\bigr)$
for some $r(x, y)$ from the list (\ref{E:MSZList}); see \cite[Theorem 2.10]{MSZ}.

\smallskip
\noindent
4. Let $r(x, y) = r_{\mathsf{st}}(x, y) + p(x, y)$ be a solution of CYBE, where $p(x, y) \in (\lieg \otimes \lieg)[x, y]$ and 
$$
r_{\mathsf{st}}(x, y) = \left\{
\begin{array}{ll}
\dfrac{\gamma}{x-y} & \mbox{\rm rational case}\\
\dfrac{y}{x-y} \gamma & \mbox{\rm quasi-trigonometric case} \\
\dfrac{xy}{x-y} \gamma & \mbox{\rm quasi-rational case}.
\end{array}
\right.
$$
For any such $r(x, y)$ we have a Lie bialgebra cobracket
$
 \lieg[z] \stackrel{\delta_r}\lar \lieg[x] \otimes \lieg[y].
 $
 Such Lie bialgebra structures of $\lieg[z]$ are controlled  by Manin triples of different shapes (depending on $r_{\mathsf{st}}(x, y)$).
 According to  \cite{Stolin}, 
rational solutions of (\ref{E:CYBE}) are parametrized by Manin triples of the form
$$
\lieg\llbrace z^{-1}\rrbrace = \lieg[z] \dotplus \lieW, \quad \lieW \asymp z^{-1} \lieg\llbracket z^{-1}\rrbracket.
$$
The theory of Manin triples for quasi-trigonometric and quasi-rational solutions od CYBE is given in  \cite{KPSST}  and \cite{StolinMagnusson}, respectively. It turns out that any quasi-trigonometric solution is equivalent (with respect to the transformation rules (\ref{E:equiv1}) and 
 (\ref{E:equiv2})) to a trigonometric solution of (\ref{E:CYBE1}); see \cite{KPSST}. Therefore, quasi-trigonometric solutions of CYBE can be used to define Lie bialgebra 
 cobrackets both on $\lieg[z]$ and 
 $\lieg\bigl[z, z^{-1}\bigr]$.

\smallskip
\noindent
5. A relation between trigonometric and quasi-trigonometric solutions was also explored in \cite[Section 4.2 and Section 4.3]{AbedinMaximov}.  In particular, let  $\lieg = \mathfrak{sl}_n(\CC)$ and $\varrho(z)$ be a trigonometric solution of (\ref{E:CYBE1}) such that the corresponding monodromy automorphism $\sigma  \in \Aut_{\CC}(\lieg)$ induces the trivial automorphism of the Dynkin diagram of $\lieg$. Then 
$\varrho(z)$ is equivalent to a quasi-trigonometric solution; see
\cite[Lemma 4.10 and Remark 4.11]{AbedinMaximov}.

\smallskip
\noindent

\subsection{Twists of the standard Lie bialgebra structure on a twisted loop algebra}
Let $\lieg$ be a finite dimensional simple complex Lie algebra, $\sigma \in \Aut_{\CC}(\lieg)$ be an automorphism of finite order $m$ and $\lieL = \lieL(\lieg, \sigma)$ be the corresponding  twisted loop algebra. In \cite{AbedinMaximov} it is shown that results of this work (in particular,  Proposition \ref{P:TwistIsTrigonometric} and Theorem \ref{T:TwistsCYBEGeometry}) can be used to extend the Belavin--Drinfeld classification of trigonometric solutions of CYBE   to a classification of twists of the standard Lie bialgebra cobracket $\lieL \stackrel{\delta_\circ}\lar \wedge^2(\lieL)$. The key  observation is thereby that for two classical twists $\ttau,\ttau' \in \wedge^2(\lieL)$ of $\delta_\circ$ the Lie bialgebras $(\lieL,\delta_\ttau)$ 
and $(\lieL,\delta_{\ttau'})$ are isomorphic via some $R$--linear automorphism of $\lieL$ if and only if there exists a holomorphic germ $(\CC,0) \stackrel{\phi}\lar \Aut_\CC(\lieg)$ such that 
\begin{equation}
     r_{\ttau'}\left(\exp\left(\dfrac{u}{m}\right),  
    \exp\left(\dfrac{v}{m}\right)\right) = (\phi(u) \otimes \phi(v))r_\ttau\left(\exp\left(\dfrac{u}{m}\right),  
    \exp\left(\dfrac{v}{m}\right)\right);
\end{equation}
see \cite[Theorem 3.7 and Theorem 5.8]{AbedinMaximov}. A proof of this statement uses the algebro-geometric theory of the CYBE developed in Section \ref{S:ReviewGeomCYBE} and Theorem \ref{T:TwistsCYBEGeometry}. In particular, as an intermediate step it is shown that the sheaves of Lie algebras $\kA_\ttau$ and  $\kA_{\ttau'}$ from Theorem \ref{T:TwistsCYBEGeometry} are isomorphic in this case.

\smallskip
\noindent
In the setting of Remark \ref{R:RemarkBDsetting} (i.e.~when  $\sigma$ is a Coxeter automorphism of a diagram automorphism of $\lieg$) this fact already yields the desired classification of classical twists of $\delta_\circ$. Combining Proposition \ref{P:TwistIsTrigonometric} with the classification of trigonometric solutions of (\ref{E:CYBE1}) presented in Section \ref{SS:ReviewBDTheory} it follows that $r_\ttau$ is equivalent to $r_Q$ given by formula (\ref{E:TwistsCoxCase}) for an appropriate Belavin--Drinfeld quadruple $Q$. It follows that $(\lieL,\delta_\ttau)$ is isomorphic to $(\lieL,\delta_Q)$, where 
$\delta_Q= \delta_\circ + \partial_{t_Q}$ and $t_Q$ is given by (\ref{E:twist}).

\smallskip
\noindent
For an arbitrary automorphism $\sigma$ this classification needs a slight adjustment; see  \cite[Lemma 3.2]{AbedinMaximov} as well as \cite[Lemma 6.22]{BelavinDrinfeld}. We keep the notation of Subsection \ref{SS:BasicsTwistedLoops}.
In this setting, a Belavin--Drinfeld quadruple $Q = \bigl((\Gamma_1,\Gamma_2,\tau),\ttaus\bigr)$ consists of (possibly empty) proper subsets $\Gamma_1,\Gamma_2$ of the set 
$
\Pi \subset 
\lieh^\ast \times \NN_0$ of simple roots of $(\lieL, \lieh)$, a bijection $\Gamma_1 \stackrel{\tau}\lar \Gamma_2$ and a tensor $\ttaus \in \wedge^2(\lieh)$ satisfying the following conditions:
\begin{itemize}
\item $\kappa\bigl(\tau(\alpha), \tau(\alpha')\bigr) = \kappa(\alpha, \alpha')$ for all $(\alpha, k), (\alpha', k')\in \Gamma_1$;
\item for any $(\alpha, k) \in \Gamma_1$ there exists $l  \in \NN$ such that  $(\alpha, k), \dots, \tau^{l-1}(\alpha, k) \in \Gamma_1$ but $\tau^l(\alpha, k) \notin \Gamma_1$;
\item $
\bigl(\beta \otimes \mathbbm{1} + \mathbbm{1} \otimes \alpha\bigr)\Bigl(\ttaus+ \dfrac{\gamma_{0}}{2}\Bigr) = 0$ for all $(\alpha, k) \in \Gamma_1$, where $\tau(\alpha, k) = (\beta, t)$.
\end{itemize}
For $i \in \{1,2\}$ consider Lie algebras $\lies_i^\pm   := \llangle x_j^\pm \, \big| \, j \in \Gamma_i \rrangle
\subset
\lieL$ and $\lies_i := \llangle x_j^+, x_j^- \, \big| \, j \in \Gamma_i \rrangle
\subset
\lieL$, where $x_j^\pm \in \lieL_{(\pm \alpha_j, \pm s_j)} = \lieg_{\pm \alpha_j} z^{\pm s_j}$ are Chevalley generators of $\lieL$ corresponding to 
$(\pm \alpha_j, \pm s_j) \in \Pi^\pm$.  Since $\Gamma_i$ is a proper subset of $\Pi$, the Lie algebra  $\lies_i$  is finite dimensional and  semisimple.  It is clear that  $\tau$ induces an isomorphism $\lies_1 \stackrel{\tilde{\tau}}\lar \lies_2$ given by the formula 
$x_j^\pm \longmapsto x_{\tau(j)}^\pm$ for all $j \in \Gamma_1$ (where we identify $\Pi$ with  $\bigl\{0, 1, \dots, r\bigr\}$). We have:  
$\tilde\tau(\lies_1^\pm) = \lies_2^\pm$. 

\smallskip
\noindent
It is clear that there exists a  finite subset $\Phi_i \subset \Phi\setminus\{(0,0)\}$ and a Lie subalgebra $\lieh_i \subset \lieh$ such that $\lies_i = \lieh_i \oplus \oplus_{(\alpha,k) \in \Phi_i}\lieL_{(\alpha,k)}$ and $\lies_i^\pm  =  \oplus_{(\alpha,k) \in \Phi^\pm_i}\lieL_{(\alpha,k)}$, where $\Phi_i^\pm = \Phi_i \cap \Phi_\pm$. 
Let $\vartheta$ be the  nilpotent $\CC$-linear endomorphism of $\lieL$ given as  the  composition
$$
\lieL \stackrel{\pi}\rightarrowdbl \lies_1^+ \stackrel{\tilde\tau}\lar \lies^+_2 \stackrel{\imath}\longhookrightarrow \lieL,
$$
where $\pi$ and $\imath$ are the canonical projection and embedding with respect to the direct sum decomposition
$\lieL = \oplus_{(\alpha,k) \in \Phi}\lieL_{(\alpha,k)}$. We put:
$
\psi  = \dfrac{\vartheta}{\mathbbm{1} - \vartheta} = \sum\limits_{l = 1}^\infty \vartheta^l \in \End_\CC(\lieL)
$. Finally, let us choose  a family $\bigl(b_{(\alpha,k)} \in \lieL_{(\alpha,k)}\bigr)_{(\alpha,k) \in \Phi_1}$ such that 
$B\bigl(b_{(\alpha,k)},b_{(\beta,t)}\bigr) = \delta_{\alpha+\beta,0}\, \delta_{k+t,0}$ for all $(\alpha,k), (\beta, t) \in \Phi_1$.
The following statement is one of main results of  \cite{AbedinMaximov}.

\smallskip
\noindent
\textbf{Theorem}. Let $Q = \bigl((\Gamma_1,\Gamma_2,\tau),\ttaus\bigr)$ be a Belavin--Drinfeld quadruple and 
$$
\ttau_Q = \ttaus + \sum_{(\alpha,k)\in\Phi_1^+} b_{(-\alpha,-k)} \wedge \psi(b_{(\alpha,k)} \in \wedge^2(\lieL).
$$
Then $\delta_Q = \delta_\circ + \partial_{\ttau_Q}$ is a twist of the  standard Lie bialgebra cobracket $\lieL \stackrel{\delta_\circ}\lar \wedge^2(\lieL)$. Conversely, let $\ttau \in \wedge^2(\lieL)$ be such $\lieL \stackrel{\delta_\ttau}\lar 
\wedge^2(\lieL)$ is a Lie bialgebra cobracket. Then there exists  a Belavin--Drinfeld quadruple $Q$ and an $R$-linear automorphism of $\lieL$ giving an isomorphism of Lie bialgebras $(\lieL, \delta_\ttau)$ and $(\lieL, \delta_Q)$. 

\smallskip
\noindent
Note that 
$$
r_Q(x, y) = \left(\frac{\gamma_{0}^{0}}{2} + \gamma_{0}^-\right) + \frac{y^m}{x^m - y^m} \sum\limits_{k = 0}^{m-1} \left(\frac{x}{y}\right)^k \gamma_k + \ttaus + \sum_{(\alpha,k)\in\Phi_1^+} b_{(-\alpha,-k)} \wedge \psi(b_{(\alpha,k)}.
$$
is a solution of CYBE.  In \cite{AbedinMaximov} these  solutions are called \emph{$\sigma$-trigonometric}.

\end{document}